\documentclass[10pt]{article}
\usepackage[utf8]{inputenc}
\usepackage{lmodern}
\usepackage[T1]{fontenc}
\usepackage{amsmath}
\usepackage{amsfonts}
\usepackage{amssymb}
\usepackage{amsthm}
\usepackage{mathrsfs}
\usepackage{stmaryrd}
\usepackage{centernot}
\usepackage{mathtools}
\usepackage{graphicx}
\usepackage[left=2.9cm,right=2.9cm,top=3cm,bottom=3cm]{geometry}
\usepackage{enumitem}
\usepackage{tikz}
\usepackage{dsfont}
\usepackage{xcolor}
\usepackage{xstring}

\newcommand{\PP}[1]{
  \StrLeft{\detokenize{#1}}{1}[\firstchar]
  \IfInteger{\firstchar}{
    \mathsf{P}_{\! #1}
  }{
    \mathsf{P}_{\! \! #1}
  }
}

\newcommand{\Shift}[1]{
  \StrLeft{\detokenize{#1}}{1}[\firstchar]
  \IfInteger{\firstchar}{
    \mathcal{S}_{#1}
  }{
    \mathcal{S}_{\! #1}
  }
}

\usepackage[maxbibnames=99]{biblatex}
\DeclareMathOperator{\Ker}{Ker}
\renewcommand{\Re}{\operatorname{Re}}
\renewcommand{\epsilon}{\varepsilon}
\DeclareMathOperator{\supp}{supp}

\DeclareMathOperator{\Op}{Op}
\DeclareMathOperator{\OpT}{\mathrm{Op}_{\intercal}}
\DeclareMathOperator{\Div}{div}
\DeclareMathOperator{\inte}{Int}
\DeclareMathOperator{\diag}{diag}
\newcommand{\hooklongrightarrow}{\lhook\joinrel\longrightarrow}

\theoremstyle{plain}
\newtheorem{thm}{Theorem}
\newtheorem{lem}[thm]{Lemma}
\newtheorem{prop}[thm]{Proposition}
\newtheorem{cor}[thm]{Corollary}

\theoremstyle{definition}
\newtheorem{defn}[thm]{Definition}

\theoremstyle{remark}
\newtheorem{rem}[thm]{Remark}

\begin{document}

\begin{center}
\large{\textbf{\uppercase{Change of regularity in controllability and observability of systems of wave equations}}}

\vspace{\baselineskip}

\large{\textsc{Thomas Perrin}}

\vspace{\baselineskip}

\large{\today}

\vspace{\baselineskip}
\end{center}

\paragraph{Abstract.} Solutions of a system of wave equations are constructed for both homogeneous and inhomogeneous Dirichlet boundary conditions at every regularity level. We prove that boundary observability, and thus boundary exact controllability, at some regularity level is equivalent to boundary observability at all levels. The main ingredient is the ellipticity of a time-derivative on the Neumann trace of the solution, which is proved by microlocal techniques.

\tableofcontents

\section*{Introduction}

Let $(M, g)$ be a $n$-dimensional compact Riemannian manifold with boundary. We write $\partial M$ for its boundary and $\inte M = M \backslash \partial M$. Let $N$ be a positive integer. Consider a first-order differential operator $X \in \mathscr{C}^\infty(M, TM \otimes \mathbb{C}^{N \times N})$, acting on functions from $M$ to $\mathbb{C}^N$, given in a coordinate chart $(U, x)$ by
\[X = X^j \frac{\partial }{\partial x^j}, \quad \text{ with }X^j \in \mathscr{C}^\infty(U, \mathbb{C}^{N \times N}) \text{ for } j \in \llbracket 1, n \rrbracket.\]
Consider also $q \in \mathscr{C}^\infty(M, \mathbb{C}^{N \times N})$, and write $\mathsf{P}$ for the operator $\mathsf{P} = \Delta - X - q$, where $\Delta = \Delta \mathrm{Id}_{\mathbb{C}^N}$ is the (vectorial) Laplace-Beltrami operator associated with the metric $g$. Denote by $\mathsf{P}^\ast$ the adjoint of $\mathsf{P}$. It has the same form as $\mathsf{P}$. We will define a family of spaces of Sobolev type, written $\mathcal{K}^{s}$ and $\mathcal{K}_\ast^{s}$, corresponding to compatibility conditions adapted to $\mathsf{P}$ and $\mathsf{P}^\ast$.

Consider $s \in \mathbb{R}$, $T > 0$, $\Theta = \left( \Theta_1, \cdots, \Theta_N \right) \in \mathscr{C}^\infty_\mathrm{c}((0, T) \times \partial M, \mathbb{C}^N)$, and set
\[\diag (\Theta) =
\begin{pmatrix}
\Theta_1 & 0 & \cdots & 0 \\
0 & \Theta_2 & \cdots & 0 \\
\vdots & \vdots & \ddots & \vdots \\
0 & 0 & \cdots & \Theta_N
\end{pmatrix}.\]
Solutions of the wave equations
\begin{equation}\label{eq_intro_dirichlet_hom}
\left \{
\begin{array}{rcccl}
\partial_t^2 u - \mathsf{P} u & = & 0 & \quad & \text{in } (0, T) \times M, \\
\left( u(0, \cdot), \partial_t u(0, \cdot) \right) & = & \left( u^0, u^1 \right) & \quad & \text{in } M, \\
u & = & 0 & \quad & \text{on } (0, T) \times \partial M,
\end{array}
\right.
\end{equation}
\begin{equation}\label{eq_intro_dirichlet_inhom}
\left \{
\begin{array}{rcccl}
\partial_t^2 v - \mathsf{P}^\ast v & = & 0 & \quad & \text{in } (0, T) \times M, \\
\left( v(0, \cdot), \partial_t v(0, \cdot) \right) & = & 0 & \quad & \text{in } M, \\
v & = & \diag(\Theta) f & \quad & \text{on } (0, T) \times \partial M,
\end{array}
\right.
\end{equation}
are given, for $\left( u^0, u^1 \right) \in \mathcal{K}^{s + 1} \times \mathcal{K}^{s}$ and $f \in H^s((0, T) \times \partial M, \mathbb{C}^N)$. We write $\partial_\nu u = \left( \partial_\nu \mathrm{Id}_{\mathbb{C}^N} \right) u$ for the normal derivative of $u$.

\begin{defn}[$H^s$-observability for $\Theta$]
We say that $H^s$-observability for $\Theta$ holds if there exists $C > 0$ such that for all $\left( u^0, u^1 \right) \in \mathcal{K}^{s + 1} \times \mathcal{K}^{s}$, 
\[\left\Vert \left( u^0, u^1 \right) \right\Vert_{\mathcal{K}^{s + 1} \times \mathcal{K}^{s}} \leq C \left\Vert \diag(\Theta) \partial_\nu u \right\Vert_{H^s((0, T) \times \partial M, \mathbb{C}^{N})}.\]
\end{defn}

\begin{defn}[$H^s$-exact controllability for $\Theta$]
We say that $H^s$-exact controllability for $\Theta$ holds if for all $(\varphi^0, \varphi^1) \in \mathcal{K}_\ast^{s} \times \mathcal{K}_\ast^{s - 1}$, there exists $f \in H^s((0, T) \times \partial M, \mathbb{C}^{N})$ such that 
\[(v(T), \partial_t v(T)) = (\varphi^0, \varphi^1).\]
\end{defn}

A duality property for solutions of (\ref{eq_intro_dirichlet_hom}) and (\ref{eq_intro_dirichlet_inhom}) implies that the classical controllability - observability equivalence is satisfied.

\begin{lem}\label{lem_equivalence_obs_control}
For $s \in \mathbb{R}$, $H^s$-exact controllability for $\Theta$ and $H^{- s}$-observability for $\Theta$ are equivalent.
\end{lem}

The main result of this article is the following. 

\begin{thm}\label{thm_main_change_reg}
Consider $s_1, s_2 \in \mathbb{R}$. If $s_1 < s_2$ then for all $\Theta \in \mathscr{C}^\infty_\mathrm{c}((0, T) \times \partial M, \mathbb{C}^N)$, $H^{s_1}$-observability for $\Theta$ implies $H^{s_2}$-observability for $\Theta$. If $s_1 > s_2$ then for all $\Theta^1 = \left( \Theta^1_1, \cdots, \Theta_N^1 \right) \in \mathscr{C}^\infty_\mathrm{c}((0, T) \times \partial M, \mathbb{C}^N)$ and $\Theta^2 = \left( \Theta^2_1, \cdots, \Theta_N^2 \right) \in \mathscr{C}^\infty_\mathrm{c}((0, T) \times \partial M, \mathbb{C}^N)$ such that for all $k \in \llbracket 1, N \rrbracket$, $\Theta_k^2 \neq 0$ on $\supp \Theta_k^1$, $H^{s_1}$-observability for $\Theta^1$ implies $H^{s_2}$-observability for $\Theta^2$.
\end{thm}

An analogue of Theorem \ref{thm_main_change_reg} for internal controllability holds, with a simpler proof (see Appendix A). For $N = 1$, Theorem \ref{thm_main_change_reg} follows from the equivalence of $H^s$-observability with the celebrated Geometric Control Condition (in short, GCC), as proven in \cite{BLR} and \cite{BurqGerard}. Here, our proof, that covers any $N \in \mathbb{N}^\ast$, does not rely on the GCC. If $\Theta = \left( \theta, \cdots, \theta \right)$ for some $\theta \in \mathscr{C}^\infty_\mathrm{c}((0, T) \times \partial M, \mathbb{C})$, then one can expect that $H^s$-observability for $\Theta$ holds if the support of $\theta$ fulfils the GCC. In the case of internal controllability, the authors of \cite{Cui-Laurent-Wang} characterize $L^2$-observability of a system of wave equations, and the condition they obtain is more involved than the usual GCC. Together with the analogue of Theorem \ref{thm_main_change_reg} for internal controllability, this provides a characterization of internal controllability of some wave systems at every regularity level. Theorem \ref{thm_main_change_reg}, along with an analogue of \cite{Cui-Laurent-Wang} for boundary controllability, would give a characterization of $H^s$-exact controllability for $s \in \mathbb{R}$.

An ongoing project with Lauri Oksanen focuses on inverse problems for systems of wave equations. The method employed relies on constructing solutions of (\ref{eq_intro_dirichlet_inhom}) concentrated within a specific spatial region, both at the regularity level $s = 0$ and at a high regularity level $s \gg 1$. By Theorem \ref{thm_main_change_reg}, we only need to check that $L^2$-exact controllability holds for the systems considered in that project.

In \cite{Ervedoza-Zuazua}, the authors prove a change of regularity result in a very general setup: they find a way to construct smooth controls for smooth data for a time-reversible semi-group. However, the main objective of the article \cite{Ervedoza-Zuazua} is to find a control which is defined in the same way at all levels of regularity, and which naturally inherits the regularity of the data. Similarly, in the article \cite{Dehman-Lebeau}, the authors study the regularity of a fixed control, constructed independently of the regularity of the data. This goal is quite different from ours. Note also that in \cite{Ervedoza-Zuazua}, in the case of boundary controllability of a wave equation (\cite{Ervedoza-Zuazua}, Theorem 5.4), the regularity of the control at a level of regularity $s \in \mathbb{R}$ is
\[H^s((0, T), L^2(\partial M)) \cap \bigcap_{k = 0}^{\lfloor s \rfloor} \mathscr{C}^k([0, T], H^{s - k - \frac{1}{2}}(\partial M)),\]
and not $H^s((0, T) \times \partial M)$. The result of the present article thus improve upon \cite{Ervedoza-Zuazua} with that respect.

\paragraph{Main ideas of the proof.} Consider the simple case $N = 1$ and $\mathsf{P} = \Delta$. If $u$ is a solution of the wave equation with homogeneous Dirichlet boundary condition, at the Sobolev regularity level $s \in \mathbb{R}$, then $\Delta u$ and $\Delta^{-1} u$ are solutions of the same wave equation at regularity levels $s - 2$ and $s + 2$. If observability holds for $\Delta u$ or $\Delta^{-1} u$, it is natural to wonder whether or not it implies observability for $u$. In the case of $\Delta^{-1} u$, the proof is easy and is essentially based on the ellipticity of $\Delta$. In the other case, the proof relies on the ellipticity of $\partial_t^2$ on the Neumann trace of solutions at the boundary (see Theorem \ref{thm_dt_ell_traces} for a precise statement), which can be proved using microlocal techniques.

\paragraph{Outline.} In Section 1, we gather some basic results about the spaces $\mathcal{K}^{s}$ and wave systems, and we prove the controllability / observability equivalence (Lemma \ref{lem_equivalence_obs_control}). In Section 2, we show the ellipticity estimate for $\partial_t^2$ acting on Neumann traces of solutions. In Section 3, we prove Theorem \ref{thm_main_change_reg}. In Appendix A, we briefly explain how the methods of this article can be adapted to the case of internal observability. Proofs of the results of Section 1 are provided in Appendix B. 

\paragraph{Notation.} For $x \in M$ and $U, V \in T_x M$, we write $\langle U, V \rangle_g$ for the inner product of $U$ and $V$ with respect to the metric $g$. The gradient with respect to $g$ of a function $u: M \rightarrow \mathbb{C}$ is denoted by $\nabla u$, and the divergence with respect to $g$ of a vector field $X$ on $M$ is denoted by $\Div X$. We write $\mathrm{d} V_g$ for the Riemannian density on $M$. Finally, $\langle \cdot, \cdot \rangle_{\mathcal{X}^\prime, \mathcal{X}}$ denotes the bilinear duality product between a Banach space $\mathcal{X}$ and its dual space $\mathcal{X}^\prime$, and $\langle \cdot, \cdot \rangle_{\mathcal{H}}$ denotes the inner product of a Hilbert space $\mathcal{H}$, which is linear in the first variable and antilinear in the the second. We write $(\pi_1, \cdots, \pi_N)$ for the projections associated with the canonical basis of $\mathbb{C}^N$.

\paragraph{Acknowledgements.} I warmly thank Thomas Duyckaerts and Jérôme Le Rousseau for their constant support and guidance. I also thank Lauri Oksanen for introducing me to the questions of exact controllability of systems of wave equations at different levels of regularity, and Sylvain Ervedoza for an enlightening discussion which allowed me to understand correctly the links between \cite{Ervedoza-Zuazua} and our results.

\vspace{\baselineskip}
\noindent
\textit{Keywords:} systems of wave equations, existence of solutions of hyperbolic equations, regularity of solutions of hyperbolic equations, controllability for systems.
\newline
\textit{MSC2020:}  35L52, 35L53, 35B65, 93B05, 93B07

\section{Controllability - observability equivalence}

\subsection{Adjoint operator}

Here, we give the precise expression of the adjoint of $\mathsf{P}$. The operator $X$ is compatible with change of coordinates, meaning that for a second set of coordinates $(\tilde{x}^1, \cdots, \tilde{x}^n)$, with $X = \tilde{X}^i \frac{\partial }{\partial \tilde{x}^i}$, one has
\begin{equation}\label{change_coord_X}
\tilde{X}^i = \frac{\partial \tilde{x}^i}{\partial x^j} X^j, \quad i \in \llbracket 1, n \rrbracket.
\end{equation}

Denote by $(e_1, \cdots, e_N)$ the canonical basis of $\mathbb{C}^N$, and take $u = (u^1, \cdots, u^N) \in \mathscr{C}^\infty(M, \mathbb{C}^N)$. We use Einstein summation convention in $M$ (for indices between $1$ and $n$) but not on $\mathbb{C^N}$ (for indices between $1$ and $N$). One writes 
\[Xu = \sum_{k, \ell = 1}^N X_{k \ell}^j \frac{\partial u^\ell}{\partial x^j} e_k,\]
where $X_{k \ell}^j$ is the coefficient $(k, \ell)$ of the matrix $X^j$. For $k, \ell \in \llbracket 1, N \rrbracket$, write $X_{k \ell} = X_{k \ell}^j \frac{\partial}{\partial x^j}$, so that
\[Xu = \sum_{k, \ell = 1}^N \left\langle \nabla u^\ell, X_{k \ell} \right\rangle_g e_k.\]
Using formula (\ref{change_coord_X}), one sees that $X_{k \ell}$ is a vector field for all $k, \ell \in \llbracket 1, N \rrbracket$. We will use the notation
\[\left\langle X, V \right\rangle_g u = \sum_{k, \ell = 1}^N \left\langle X_{k \ell}, V \right\rangle_g u^\ell e_k \in \mathbb{C}^N,\]
if $V$ is a vector field on $M$ and $u = (u^1, \cdots, u^N)$ is a function with values in $\mathbb{C}^N$. With integration by parts, one derives the following results.

\begin{lem}
For $u, v \in H^1(M, \mathbb{C}^N)$, one has 
\[\left\langle X u, v \right\rangle_{L^2(M, \mathbb{C}^N)} = \left\langle u, X^\ast v \right\rangle_{L^2(M, \mathbb{C}^N)} + \left\langle \langle X, \nu \rangle_g u, v \right\rangle_{L^2(\partial M, \mathbb{C}^N)},\]
where $X^\ast$ is the first-order differential operator given by
\[\left( X^\ast v \right)^\ell = - \sum_{k = 1}^N \left\langle \nabla v^k, \overline{X_{k \ell}} \right\rangle_g - \left( \left( \Div \overline{X} \right) v \right)^\ell, \quad \ell \in \llbracket 1, N \rrbracket.\]
\end{lem}

\begin{lem}
For $u, v \in H^2(M, \mathbb{C}^N)$, one has
\begin{align*}
\left\langle \mathsf{P} u, v \right\rangle_{L^2(M, \mathbb{C}^N)} = \ & \left\langle u, \mathsf{P}^\ast v \right\rangle_{L^2(M, \mathbb{C}^N)} + \left\langle \langle X, \nu \rangle_g u, v \right\rangle_{L^2(\partial M, \mathbb{C}^N)} \\
& + \left\langle \partial_\nu u, v \right\rangle_{L^2(\partial M, \mathbb{C}^N)} - \left\langle u, \partial_\nu v \right\rangle_{L^2(\partial M, \mathbb{C}^N)},
\end{align*}
with $\mathsf{P}^\ast = \Delta - X^\ast - q^\ast$, where $q^\ast$ denotes the adjoint of $q$.
\end{lem}

\begin{rem}
In particular, if $u, v \in H^2(M, \mathbb{C}^N) \cap H_0^1(M, \mathbb{C}^N)$, then 
\[\left\langle \mathsf{P} u, v \right\rangle_{L^2(M, \mathbb{C}^N)} = \left\langle u, \mathsf{P}^\ast v \right\rangle_{L^2(M, \mathbb{C}^N)}.\]
\end{rem}

\begin{rem}
The operators $\mathsf{P}$ and $\mathsf{P}^\ast$ are of the same form. Indeed, set $\tilde{X} = \left( X^j \right)^\ast \frac{\partial}{\partial x^j}$ and let $\tilde{q} \in \mathscr{C}^\infty(M, \mathbb{C}^{N \times N})$ be given by $\tilde{q}_{k \ell} = - \Div \overline{X_{\ell k}} + \overline{q_{\ell k}}$ for $k, \ell \in \llbracket 1, N \rrbracket$. One has $\mathsf{P}^\ast = \Delta - \tilde{X} - \tilde{q}$.
\end{rem}

\subsection{A family of spaces of Sobolev type}

Denote by $\PP{\mathscr{D}^\prime}$ the action of $\mathsf{P}$ on distributions, that is,
\[\PP{\mathscr{D}^\prime}: \mathscr{D}^\prime(M, \mathbb{C}^N) \longrightarrow \mathscr{D}^\prime(M, \mathbb{C}^N),\]
with $\langle \PP{\mathscr{D}^\prime} u, \phi \rangle_{\mathscr{D}^\prime, \mathscr{D}} = \langle u, \mathsf{P}^\ast \phi \rangle_{\mathscr{D}^\prime, \mathscr{D}}$ for $u \in \mathscr{D}^\prime(M, \mathbb{C}^N)$ and $\phi \in \mathscr{D}(M, \mathbb{C}^N)$.

We define a Sobolev-like regularity, adapted to the operator $\mathsf{P}$. Write $\mathcal{K}^0 = L^2(M, \mathbb{C}^N)$ and for $m \in \mathbb{N}^\ast$, set
\[\mathcal{K}^{m} = \left\{ u \in H^m(M, \mathbb{C}^N), \PP{\mathscr{D}^\prime}^k u \in H_0^1(M, \mathbb{C}^N) \text{ for } k \in \left\llbracket 0, \left\lfloor \frac{m - 1}{2} \right\rfloor \right\rrbracket \right\},\]
endowed with the $H^m$ inner product. Here, $\lfloor \cdot \rfloor$ is the floor function. Note that $\mathcal{K}^1 = H_0^1(M, \mathbb{C}^N)$ and $\mathcal{K}^2 = H^2(M, \mathbb{C}^N) \cap H_0^1(M, \mathbb{C}^N)$. For $m \in \mathbb{N}$, one checks that $\mathcal{K}^{m}$ is a complete subspace of $H^m(M, \mathbb{C}^N)$, and, in particular, is a Hilbert space.

The space $\mathcal{K}^{s}$ is defined for $s \geq 0$ by interpolation. Since the operators $\mathsf{P}$ and $\mathsf{P}^\ast$ are of the same form, one can define the space $\mathcal{K}_\ast^{s}$ for $s \geq 0$, by replacing $\mathsf{P}$ with $\mathsf{P}^\ast$ in the previous definitions. Then, for $s < 0$, define $\mathcal{K}^{s}$ as the dual of $\mathcal{K}_\ast^{- s}$, and $\mathcal{K}_\ast^{s}$ as the dual of $\mathcal{K}^{- s}$. Note that for $m \in \mathbb{N}$ sufficiently large, $H_0^m(M, \mathbb{C}^N)$ is not dense in $\mathcal{K}^{m}$, so that $\mathcal{K}^{-m} \nsubseteq H^{-m}(M, \mathbb{C}^N)$. For $s \geq 0$, set
\[\Vert u \Vert_{\mathcal{K}^{s}} = \Vert u \Vert_{H^s(M, \mathbb{C}^N)},\]
and for $s < 0$, set 
\[\Vert u \Vert_{\mathcal{K}^{s}} = \sup \left\{ \left\vert \langle u, v \rangle_{\mathcal{K}^{s}, \mathcal{K}_\ast^{- s}} \right\vert , \Vert v \Vert_{\mathcal{K}_\ast^{- s}} \leq 1 \right\},\]
the usual norm of a dual space. For $s \in \mathbb{R}$, a norm on $\mathcal{K}_\ast^{s}$ is defined similarly.

Next, we define the natural action of $\mathsf{P}$ on $\mathcal{K}^{s}$. With interpolation, the definition of $\PP{s}: \mathcal{K}^{s + 1} \longrightarrow \mathcal{K}^{s - 1}$ is only needed for $s \in \mathbb{Z}$. 

\begin{defn}[Definition of $\PP{s}$]
\begin{enumerate}[label=(\roman*)]
\item Suppose $s \in \mathbb{N}^\ast$. Then the operator $\PP{s}: \mathcal{K}^{s + 1} \longrightarrow \mathcal{K}^{s - 1}$ is the differential operator $\mathsf{P}$ on $\mathcal{K}^{s + 1}$. It is a bounded operator. The operator $\PP{s}^\ast: \mathcal{K}_\ast^{s + 1} \longrightarrow \mathcal{K}_\ast^{s - 1}$ is defined similarly.

\item Suppose $s \in \mathbb{Z}$, $s \leq -1$. Define $\PP{s}: \mathcal{K}^{s + 1} \longrightarrow \mathcal{K}^{s - 1}$ as the adjoint of $\PP{- s}^\ast: \mathcal{K}_\ast^{- s + 1} \longrightarrow \mathcal{K}_\ast^{- s - 1}$, and $\PP{s}^\ast: \mathcal{K}_\ast^{s + 1} \longrightarrow \mathcal{K}_\ast^{s - 1}$ as the adjoint of $\PP{- s}: \mathcal{K}^{- s + 1} \longrightarrow \mathcal{K}^{- s - 1}$.

\item If $s = 0$, then $\mathcal{K}^{s + 1} = \mathcal{K}_\ast^{s + 1} = H_0^1(M, \mathbb{C}^N)$ and $\mathcal{K}^{s - 1} = \mathcal{K}_\ast^{s - 1} = H^{-1}(M, \mathbb{C}^N)$. For $u \in H_0^1(M, \mathbb{C}^N)$, define $\PP{0} u \in H^{-1}(M, \mathbb{C}^N)$ by 
\[\left\langle \PP{0} u, v \right\rangle_{H^{-1}, H_0^1} = - \sum_{k = 1}^N \left\langle \nabla u^k, \overline{\nabla v^k} \right\rangle_{L^2(M)} - \left\langle (X + q) u, \overline{v} \right\rangle_{L^2(M, \mathbb{C}^N)}, \quad v \in H_0^1(M, \mathbb{C}^N).\]
This gives an operator $\PP{0}: \mathcal{K}^1 \longrightarrow \mathcal{K}^{-1}$. The operator $\PP{0}^\ast: \mathcal{K}_\ast^1 \longrightarrow \mathcal{K}_\ast^{-1}$ is defined similarly.

\item For $r \in \mathbb{N}$ and $s \in \mathbb{R}$, also define $\PP{s}^0 = \mathrm{Id}_{\mathcal{K}^s}$, $\PP{s}^r: \mathcal{K}^{s + r} \rightarrow \mathcal{K}^{s - r}$ by
\[\PP{s}^r: \mathcal{K}^{s + r} \xrightarrow{\PP{s + r - 1}} \mathcal{K}^{s + r - 2} \xrightarrow{\PP{s + r - 3}} \cdots \xrightarrow{\PP{s - r + 3}} \mathcal{K}^{s - r + 2} \xrightarrow{\PP{s - r + 1}} \mathcal{K}^{s - r},\]
and $\PP{s}^{\ast r}: \mathcal{K}_\ast^{s + r} \rightarrow \mathcal{K}_\ast^{s - r}$ similarly.
\end{enumerate}
\end{defn}

Note that for all $s \in \mathbb{R}$, as $\mathcal{K}^{s + 1}$ and $\mathcal{K}^{s - 1}$ are Hilbert spaces, one has
\begin{equation}\label{eq_adjoint_adjoint_P}
\PP{s} = \left( \PP{- s}^\ast \right)^\ast.
\end{equation}
We check that our definitions make sense in the following lemma.

\begin{lem}
For $s \in \mathbb{R}$ and $r \in \mathbb{N}^\ast$, the operator $\PP{s}^r: \mathcal{K}^{s + r} \rightarrow \mathcal{K}^{s - r}$ is well-defined and continuous. The same is true for $\PP{s}^{\ast r}: \mathcal{K}_\ast^{s + r} \rightarrow \mathcal{K}_\ast^{s - r}$. If $s \in \mathbb{R}$ and $r \in \mathbb{N}^\ast$ are such that $s - r \geq -1$, then for $u \in \mathcal{K}^{s + r}$ and $v \in \mathcal{K}^{s - r}$ such that $v = \PP{s}^r u$, one has $v = \PP{\mathscr{D}^\prime}^r u$ in $\mathscr{D}^\prime(M, \mathbb{C}^N)$.
\end{lem}

\begin{proof}
By definition of $\PP{s}^r$, we can always assume that $r = 1$, and by interpolation, we may assume that $s \in \mathbb{Z}$. The connection between $\PP{s}$ and $\PP{\mathscr{D}^\prime}$ follows from our definition of $\PP{s}$ for $s \geq 0$. 

Fix $s \in \mathbb{N}^\ast$. For $u \in \mathcal{K}^{s + 1}$, one has $u \in H^{s + 1}(M, \mathbb{C}^N)$ and $\PP{s} u = \PP{\mathscr{D}^\prime} u$ in $\mathscr{D}^\prime(M, \mathbb{C}^N)$, implying $\PP{s} u \in H^{s - 1}(M, \mathbb{C}^N)$ and
\[\left\Vert \PP{s} u \right\Vert_{H^{s - 1}} \lesssim \Vert u \Vert_{H^{s + 1}}.\]
Thus, we only have to check that the boundary conditions of the definition of the spaces $\mathcal{K}^{s}$ are such that the operators $\PP{s}$ are well-defined. For $s = 0$, the result is true.

Assume that $s$ is even and write $s = 2 \sigma$. By definition, $u \in \mathcal{K}^{2 \sigma + 1}$ gives
\[\PP{\mathscr{D}^\prime}^k u \in H_0^1(M, \mathbb{C}^N) \text{ for } k \in \llbracket 0, \sigma \rrbracket.\]
As $\PP{\mathscr{D}^\prime} u = \PP{s} u$ in $\mathscr{D}^\prime(M, \mathbb{C}^N)$, one has 
\[\PP{\mathscr{D}^\prime}^k \left( \PP{s} u \right) \in H_0^1(M, \mathbb{C}^N) \text{ for } k \in \llbracket 0, \sigma - 1 \rrbracket,\]
that is, $\PP{s} u \in \mathcal{K}^{2 \sigma - 1}$.

Assume that $s$ is odd and write $s = 2 \sigma + 1$. By definition, $u \in \mathcal{K}^{2 \sigma + 2}$ gives
\[\PP{\mathscr{D}^\prime}^k u \in H_0^1(M, \mathbb{C}^N) \text{ for } k \in \llbracket 0, \sigma \rrbracket.\]
If $\sigma = 0$, one has $\PP{s} u \in \mathcal{K}^{2 \sigma}$. If $\sigma \geq 1$, then as $\PP{\mathscr{D}^\prime} u = \PP{s} u$ in $\mathscr{D}^\prime(M, \mathbb{C}^N)$, one has 
\[\PP{\mathscr{D}^\prime}^k \left( \PP{s} u \right) \in H_0^1(M, \mathbb{C}^N) \text{ for } k \in \llbracket 0, \sigma - 1 \rrbracket,\]
that is, $\PP{s} u \in \mathcal{K}^{2 \sigma}$.

Finally, the adjoint of a continuous linear operator between Hilbert spaces is a well-defined continuous operator, so the result is true for $s \in \mathbb{Z}$, $s \leq -1$.
\end{proof}

\begin{rem}\label{rem_Delta_contre_ex}
\begin{enumerate}[label=(\roman*)]
\item The fact that $v = \PP{s}^r u$ implies $v = \PP{\mathscr{D}^\prime}^r u$ does not hold for $s < -1$, because $\mathcal{K}^{s}$ is not included in $\mathscr{D}^\prime(M, \mathbb{C}^N)$ if $s < -1$. 
\item The previous definitions are very natural, but note that some non-intuitive phenomena can occur when dealing with the operator $\PP{s}: \mathcal{K}^{s + 1} \rightarrow \mathcal{K}^{s - 1}$. To illustrate this, take $N = 1$, $\mathsf{P} = \Delta$ and $s = -1$. In that case, one has $\PP{s}^\ast = \PP{-s}$ for $s \in \mathbb{R}$. Recall that by definition, $H^{- 2}(M)$ is the dual of $H_0^2(M)$. The constant function $u = 1$ belongs to $L^2(M)$, and is sent to zero by the differential operator $\Delta: L^2(M) \rightarrow H^{- 2}(M)$. However, by definition, the operator $\PP{-1}: L^2(M) \rightarrow \mathcal{K}^{- 2}$ is the adjoint of the operator
\[\PP{1}: \mathcal{K}^2 = H^2(M) \cap H_0^1(M) \longrightarrow L^2(M),\]
implying
\[\left\langle \PP{- 1} u, v \right\rangle_{\mathcal{K}^{-2}, \mathcal{K}^2} = \left\langle 1, \overline{\PP{1} v} \right\rangle_{L^2(M)} = \left\langle 1, \Delta \overline{v} \right\rangle_{L^2(M)} = \left\langle 1, \partial_\nu \overline{v} \right\rangle_{L^2(\partial M)}, \quad v \in \mathcal{K}^2.\]
Hence, the function $u = 1$ is not sent to zero by the operator $\PP{- 1}: L^2(M) \rightarrow \mathcal{K}^{- 2}$.
\end{enumerate}
\end{rem}

In the following proposition, we gather the properties of the spaces $\mathcal{K}^{s}$ that are needed for what follows.

\begin{prop}\label{prop_main_properties_K^s}
\begin{enumerate}[label=(\roman*)]
\item \emph{Embeddings properties.} For $s \in \mathbb{R}$ and $\delta > 0$, the map
\[\iota_{\mathcal{K}^{s + \delta} \rightarrow \mathcal{K}^{s}}: \mathcal{K}^{s + \delta} \hooklongrightarrow \mathcal{K}^{s}\]
is a well-defined, compact embedding with a dense range. If $s + \delta < 0$, the embedding corresponds to a restriction operator. If $s + \delta \geq 0 > s$, then the embedding is defined by using $L^2(M, \mathbb{C}^N)$ as a pivot space. The operator $\mathsf{P}$ commutes with the embeddings: more precisely, for $r \in \mathbb{N}$, $s \in \mathbb{R}$ and $\delta > 0$, one has 
\begin{equation}\label{eq_lem_P_comm_with_embeddings}
\PP{s}^r \circ \iota_{\mathcal{K}^{s + r + \delta} \rightarrow \mathcal{K}^{s + r}} = \iota_{\mathcal{K}^{s - r + \delta} \rightarrow \mathcal{K}^{s - r}} \circ \PP{s + \delta}^r: \mathcal{K}^{s + r + \delta} \rightarrow \mathcal{K}^{s - r}.
\end{equation}

\item \emph{Elliptic estimate of $\mathsf{P}$.} Consider $s \in \mathbb{R}$, $r \in \mathbb{N}^\ast$, and $w \in \mathcal{K}^{s + r - 1}$. We already know that $\PP{s - 1}^r w \in \mathcal{K}^{s - r - 1}$. Assume that there exists $v \in \mathcal{K}^{s - r}$ such that $\PP{s - 1}^r w = \iota_{\mathcal{K}^{s - r} \rightarrow \mathcal{K}^{s - r - 1}}(v)$. Then there exists $u \in \mathcal{K}^{s + r}$ such that 
\[\iota_{\mathcal{K}^{s + r} \rightarrow \mathcal{K}^{s + r - 1}}(u) = w \quad \text{and} \quad \PP{s}^r u = v.\]
Moreover, there exists $C > 0$ such that
\begin{equation}\label{eq_prop_elliptic_estimate_P}
\Vert u \Vert_{\mathcal{K}^{s + r}} \leq C \left( \left\Vert \PP{s}^r u \right\Vert_{\mathcal{K}^{s - r}} + \left\Vert \iota_{\mathcal{K}^{s + r} \rightarrow \mathcal{K}^{s + r - 1}} u \right\Vert_{\mathcal{K}^{s + r - 1}} \right), \quad u \in \mathcal{K}^{s + r}.
\end{equation} 

\item \emph{The shift operator.} 
For $s \in \mathbb{R}$ and $r \in \mathbb{N}^\ast$, there exists a continuous isomorphism $\Shift{s}^r: \mathcal{K}^{s + r} \longrightarrow \mathcal{K}^{s - r}$ such that the following property holds: for $r, r^\prime \in \mathbb{N}$, $s \in \mathbb{R}$ and $\delta > 0$,
\begin{equation}\label{eq_prop_shifting_op_1}
\Shift{s}^{r + r^\prime} = \Shift{s - r^\prime}^{r} \circ \Shift{s + r}^{r^\prime}: \mathcal{K}^{s + r + r^\prime} \rightarrow \mathcal{K}^{s - r - r^\prime},
\end{equation}
\[\Shift{s - 1}^r \circ \PP{s + r} = \PP{s - r} \circ \Shift{s + 1}^r: \mathcal{K}^{s + r + 1} \rightarrow \mathcal{K}^{s - r - 1},\]
and
\begin{equation}\label{eq_prop_shifting_op_3}
\Shift{s}^r \circ \iota_{\mathcal{K}^{s + r + \delta} \rightarrow \mathcal{K}^{s + r}} = \iota_{\mathcal{K}^{s - r + \delta} \rightarrow \mathcal{K}^{s - r}} \circ \Shift{s + \delta}^r: \mathcal{K}^{s + r + \delta} \rightarrow \mathcal{K}^{s - r}.
\end{equation}
In addition, for $r \in \mathbb{N}$ and $s \in \mathbb{R}$, one has
\begin{equation}\label{eq_prop_shifting_op_4}
\left\Vert \left(\PP{s}^{r} - \Shift{s}^{r} \right) u \right\Vert_{\mathcal{K}^{s - r}} \leq C_{r, s} \left\Vert \iota_{\mathcal{K}^{s + r} \rightarrow \mathcal{K}^{s + r - 1}} u \right\Vert_{\mathcal{K}^{s + r - 1}}, \quad u \in \mathcal{K}^{s + r},
\end{equation}
for some $C_{r, s} > 0$. The operator $\Shift{s}^1$ will be defined by $\PP{s} + i\mu \iota_{\mathcal{K}^{s + 1} \rightarrow \mathcal{K}^{s - 1}}$, for $\mu \in \mathbb{R}$ chosen sufficiently large.
\end{enumerate}
\end{prop}

\begin{rem}
If we omit the embedding notation, then \emph{(ii)} can be written as
\[u \in \mathcal{K}^{s + r - 1} \text{ and } \PP{s - 1}^r u \in \mathcal{K}^{s - r} \quad \Longrightarrow \quad u \in \mathcal{K}^{s + r}.\]
Note that we cannot replace $\PP{s - 1}^r u \in \mathcal{K}^{s - r}$ by $\PP{\mathscr{D}^\prime}^r u \in \mathcal{K}^{s - r}$. With the same example as in Remark \ref{rem_Delta_contre_ex}, take $N = 1$, $\mathsf{P} = \Delta$, $s = 0$, $r = 1$, and let $u$ be the constant function $u = 1 \in \mathcal{K}^0 = L^2(M)$. One has $\PP{\mathscr{D}^\prime} u = 0$, implying $\PP{\mathscr{D}^\prime} u \in \mathcal{K}^{-1} = H^{-1}(M)$. However, $u$ does not belong to the space $\mathcal{K}^1 = H_0^1(M)$.
\end{rem}

\begin{rem}
By definition, for $s \in \mathbb{R}$ and $\delta > 0$, one has
\begin{equation}\label{eq_adjoint_iota}
\iota_{\mathcal{K}^{s + \delta} \rightarrow \mathcal{K}^{s}} = \left( \iota_{\mathcal{K}_\ast^{- s} \rightarrow \mathcal{K}_\ast^{- s - \delta}} \right)^\ast.
\end{equation}
\end{rem}

The proof of Proposition \ref{prop_main_properties_K^s} is given in appendix. The proof of our main result will use the following interpolation lemma.

\begin{lem}\label{lem_iterpolation}
For $\eta \in [0, 1]$, $s \in \mathbb{R}$, one has $\left[ \mathcal{K}^{s + 2}, \mathcal{K}^{s} \right]_{\eta} = \mathcal{K}^{s + 2 - 2 \eta}$, with equivalent norms, where $\left[ \mathcal{K}^{s + 2}, \mathcal{K}^{s} \right]_{\eta}$ denotes the complex interpolation space between $\mathcal{K}^{s + 2}$ and $\mathcal{K}^{s}$.
\end{lem}

\begin{proof}
First, note that the result is standard for $s \in [-2, 0]$, as $\mathcal{K}^{s} = D(\Delta_\mathrm{dir}^\frac{s}{2})^N$ for $s \in [-2, 2]$. Second, we prove Lemma \ref{lem_iterpolation} for $s > 0$, using the shift operator and the definition of complex interpolation spaces (see, for example, \cite{Bergh1976}). If $A_0$ and $A_1$ are subspaces of a Banach space $\mathcal{X}$, we write $\mathscr{F}_{A_0, A_1}$ for the set of continuous functions $f: \left\{ z \in \mathbb{C}, 0 \leq \Re z \leq 1 \right\} \rightarrow A_0 + A_1$ satisfying the following two properties: $f$ is analytic on the open strip $\left\{ z \in \mathbb{C}, 0 < \Re z < 1 \right\}$, and for $j = 0, 1$, the function $t \mapsto f(j + it)$ maps continuously $\mathbb{R}$ to $A_j$, and tends to zero as $\vert t \vert$ tends to infinity. To ease notation, we omit embeddings and subscripts of the shift operator, identifying $\mathcal{K}^{s}$ and $\iota_{\mathcal{K}^{s} \rightarrow \mathcal{K}^{-2}} \left( \mathcal{K}^{s} \right)$, for $s \geq -2$, and writing $\mathcal{S}^k = \mathcal{S}_{k - 2}^k: \mathcal{K}^{2k - 2} \rightarrow \mathcal{K}^{-2}$. Consider $s > 0$, $k \in \mathbb{N}$ such that $s - 2k \in [-2, 0]$, and $\eta \in [0, 1]$. By definition, one has
\[\left[ \mathcal{K}^{s + 2}, \mathcal{K}^{s} \right]_{\eta} = \left\{ u \in \mathcal{K}^{s + 2} + \mathcal{K}^{s}, u = f(\eta) \text{ for some } f \in \mathscr{F}_{\mathcal{K}^{s + 2}, \mathcal{K}^{s}}\right\}.\]

As $\left\{ \mathcal{S}^k \circ f, f \in \mathscr{F}_{\mathcal{K}^{s + 2}, \mathcal{K}^{s}} \right\} = \mathscr{F}_{\mathcal{K}^{s + 2 - 2k}, \mathcal{K}^{s - 2k}}$, one has 
\begin{align*}
u \in \left[ \mathcal{K}^{s + 2}, \mathcal{K}^{s} \right]_{\eta} & \Longleftrightarrow \mathcal{S}^k u = v \text{ for some } v \in \left[\mathcal{K}^{s + 2 - 2k}, \mathcal{K}^{s - 2k} \right]_{\eta}, \\
& \Longleftrightarrow u = \left( \mathcal{S}^k \right)^{-1} v \text{ for some } v \in \mathcal{K}^{s + 2 - 2k - 2 \eta}, \\
& \Longleftrightarrow u \in \mathcal{K}^{s + 2 - 2 \eta},
\end{align*}
by the case $s \in [-2, 0]$, and Proposition \ref{prop_main_properties_K^s}. Third, for $s < -2$, using Corollary 4.5.2 and Theorem 4.2.1 of \cite{Bergh1976}, one obtains $\left( \left[ \mathcal{K}^{s + 2}, \mathcal{K}^{s} \right]_{\eta} \right)^\prime = \left[\mathcal{K}_{\ast}^{- s}, \mathcal{K}_{\ast}^{- s - 2} \right]_{1 - \eta} = \mathcal{K}_{\ast}^{- s - 2 + 2 \eta}$, as $\mathsf{P}$ and $\mathsf{P}^\ast$ are of the same form. This completes the proof.
\end{proof}

\subsection{Solutions of the wave equations}

Most of the ideas used here can be found in \cite{Las-Lio-Tri}. For wave equations with Dirichlet boundary condition, one has the following theorem.

\begin{thm}\label{thm_LLT_dir}
Consider $s \in \mathbb{R}$ and $\left( u^0, u^1 \right) \in \mathcal{K}^{s + 1} \times \mathcal{K}^{s}$. There exists a unique 
\[u \in \mathscr{C}^0(\mathbb{R}, \mathcal{K}^{s + 1}) \cap \mathscr{C}^1(\mathbb{R}, \mathcal{K}^{s}) \cap \mathscr{C}^2(\mathbb{R}, \mathcal{K}^{s - 1})\]
such that $\left( u(0), \partial_t u(0) \right) = \left( u^0, u^1 \right)$ and $\partial_t^2 u(t) = \PP{s} u(t)$ for all $t \in \mathbb{R}$. We will say that $u$ is the solution of the wave equation 
\[\left \{
\begin{array}{rcccl}
\partial_t^2 u - \mathsf{P} u & = & 0 & \quad & \text{in } \mathbb{R} \times M, \\
\left( u(0, \cdot), \partial_t u(0, \cdot) \right) & = & \left( u^0, u^1 \right) & \quad & \text{in } M, \\
u & = & 0 & \quad & \text{on } \mathbb{R} \times \partial M.
\end{array}
\right.\]
The following additional results hold.
\begin{enumerate}[label=(\roman*)]
\item One has
\[u \in \bigcap_{k \in \mathbb{N}} \mathscr{C}^k(\mathbb{R}, \mathcal{K}^{s + 1 - k}),\]
and $\partial_t^{2k} u(t) = \PP{s + 1 - k}^{k} u(t) \in \mathcal{K}^{s + 1 - 2 k}$ for $k \in \mathbb{N}$, and $t \in \mathbb{R}$. For all $k \in \mathbb{N}$ and $T > 0$, there exists $C > 0$ such that
\[\left\Vert \partial_t^k u \right\Vert_{L^\infty((0, T), \mathcal{K}^{s + 1 - k})} \leq C \left\Vert \left( u^0, u^1 \right) \right\Vert_{\mathcal{K}^{s + 1} \times \mathcal{K}^{s}}, \quad \left( u^0, u^1 \right) \in \mathcal{K}^{s + 1} \times \mathcal{K}^{s}.\]
In particular, if $s \geq -2$, then $u \in H^{s + 1}((0, T) \times M, \mathbb{C}^N)$ for all $T > 0$, with the corresponding inequality.

\item For $\delta > 0$, if $\tilde{u}$ is the solution with initial data
\[\left( \iota_{\mathcal{K}^{s + 1} \rightarrow \mathcal{K}^{s + 1 - \delta}} u^0, \iota_{\mathcal{K}^{s} \rightarrow \mathcal{K}^{s - \delta}} u^1 \right),\]
then for $t \in \mathbb{R}$, one has $\iota_{\mathcal{K}^{s + 1} \rightarrow \mathcal{K}^{s + 1 - \delta}} u(t) = \tilde{u}(t)$. In particular, a solution can be approximated by solutions with higher regularity.

\item Consider $T > 0$. A normal derivative $\partial_\nu u$ at the boundary, that lies in $H^s((0, T) \times \partial M, \mathbb{C}^N)$, can be defined extending the usual normal derivative if $u$ is sufficiently smooth. For $\delta > 0$, one has
\[\partial_\nu \left( \iota_{\mathcal{K}^{s + 1} \rightarrow \mathcal{K}^{s + 1 - \delta}} u \right) = \iota_{H^s \rightarrow H^{s - \delta}} \partial_\nu u,\]
where $\iota_{H^s \rightarrow H^{s - \delta}}$ denotes the embedding from $H^s((0, T) \times \partial M, \mathbb{C}^N)$ into $H^{s - \delta}((0, T) \times \partial M, \mathbb{C}^N)$. There exists $C > 0$ such that
\[\left\Vert \partial_\nu u \right\Vert_{H^s((0, T) \times \partial M, \mathbb{C}^N)} \leq C \left\Vert \left( u^0, u^1 \right) \right\Vert_{\mathcal{K}^{s + 1} \times \mathcal{K}^{s}}, \quad \left( u^0, u^1 \right) \in \mathcal{K}^{s + 1} \times \mathcal{K}^{s}.\]
For $k \in \mathbb{N}$, $\partial_t^{2k} u$ is the solution associated with $\left( \PP{s + 1 - k}^{k} u^0, \PP{s - k}^{k} u^1 \right) \in \mathcal{K}^{s + 1 - 2 k} \times \mathcal{K}^{s - 2 k}$, and one has
\[\partial_\nu \partial_t^{2k} u = \partial_t^{2k} \partial_\nu u \in H^{s - 2 k}((0, T) \times \partial M, \mathbb{C}^N).\]

\item Assume that $s \geq 0$. For $F \in L^1((0, T), H_0^s(M, \mathbb{C}^N))$, we define the solution of 
\[\left \{
\begin{array}{rcccl}
\partial_t^2 u - \mathsf{P} u & = & F & \quad & \text{in } (0, T) \times M, \\
\left( u(0, \cdot), \partial_t u(0, \cdot) \right) & = & 0 & \quad & \text{in } M, \\
u & = & 0 & \quad & \text{on } (0, T) \times \partial M,
\end{array}
\right.\]
using the Duhamel formula. One has $u \in \mathscr{C}^0((0, T), \mathcal{K}^{s + 1}) \cap \mathscr{C}^1((0, T), \mathcal{K}^{s})$, $\partial_\nu u \in H^s((0, T) \times \partial M, \mathbb{C}^N)$, and there exists $C > 0$ such that
\[\left\Vert \left( u, \partial_t u \right) \right\Vert_{L^\infty((0, T), \mathcal{K}^{s + 1} \times \mathcal{K}^{s})} + \left\Vert \partial_\nu u \right\Vert_{H^s((0, T) \times \partial M, \mathbb{C}^N)} \leq C \left\Vert F \right\Vert_{L^1((0, T), H^s)},\]
for all $F \in L^1((0, T), H_0^s(M, \mathbb{C}^N))$. If in addition $F \in \mathscr{C}^0((0, T), H^{s - 1}(M, \mathbb{C}^N))$, then $u \in \mathscr{C}^2((0, T), \mathcal{K}^{s - 1})$, with $\partial_t^2 u = \PP{s}u + F$ and
\[\left\Vert \partial_t^2 u \right\Vert_{L^\infty((0, T), \mathcal{K}^{s - 1})} \leq C \left( \left\Vert
F \right\Vert_{L^1((0, T), H_0^s(M, \mathbb{C}^N))} + \left\Vert F \right\Vert_{L^\infty((0, T), H^{s - 1}(M, \mathbb{C}^N))} \right),\]
for some $C > 0$ independent of $F$.
\end{enumerate}
\end{thm}

Using Theorem \ref{thm_LLT_dir}-\emph{(i)}, Theorem \ref{thm_LLT_dir}-\emph{(ii)} and Proposition \ref{prop_main_properties_K^s}-\emph{(iii)}, one obtains the following corollary.

\begin{cor}\label{cor_LLT_dir_shift}
Consider $s \in \mathbb{R}$, $r \in \mathbb{N}$, $\left( u^0, u^1 \right) \in \mathcal{K}^{s + 1} \times \mathcal{K}^{s}$, and denote by $u$ the solution with initial data $\left( u^0, u^1 \right)$, given by Theorem \ref{thm_LLT_dir}. Then, $w = \Shift{s - r + 1}^{r} u$ is the solution of
\[\left \{
\begin{array}{rcccl}
\partial_t^2 w - \mathsf{P} w & = & 0 & \quad & \text{in } (0, T) \times M, \\
\left( w(0, \cdot), \partial_t w(0, \cdot) \right) & = & \left( w^0, w^1 \right) & \quad & \text{in } M, \\
w & = & 0 & \quad & \text{on } (0, T) \times \partial M,
\end{array}
\right.\]
where $\left( w^0, w^1 \right) = \left( \Shift{s - r + 1}^{r} u^0, \Shift{s - r}^{r} u^1 \right) \in \mathcal{K}^{s - 2r + 1} \times \mathcal{K}^{s - 2r}$.
\end{cor}

For wave equations with inhomogeneous boundary condition, one has the following theorem. 

\begin{thm}\label{thm_LLT_inhomogeneous}
Consider $T > 0$, $\Theta \in \mathscr{C}^\infty_\mathrm{c}((0, T) \times \partial M, \mathbb{C}^N)$, $s \in \mathbb{R}$ and $f \in H^s((0, T) \times \partial M, \mathbb{C}^N)$. If $s \leq 0$, we define the solution of the wave equation 
\begin{equation}\label{eq_thm_LLT_inhomogeneous}
\left \{
\begin{array}{rcccl}
\partial_t^2 v - \mathsf{P}^\ast v & = & 0 & \quad & \text{in } \mathbb{R} \times M, \\
\left( v(0, \cdot), \partial_t v(0, \cdot) \right) & = & 0 & \quad & \text{in } M, \\
v & = & \diag(\Theta) f & \quad & \text{on } \mathbb{R} \times \partial M.
\end{array}
\right.
\end{equation}
by duality with Theorem \ref{thm_LLT_dir}-(iv): $v$ is the unique element of $L^\infty((0, T), H^s(M, \mathbb{C}^N))$ such that 
\[\left\langle v, F \right\rangle_{L^\infty(H^s), L^1(H_0^{- s})} = - \left\langle f, \diag(\Theta) \partial_\nu u \right\rangle_{H^s, H_0^{- s}},\]
for all $F \in L^1((0, T), H_0^{- s}(M, \mathbb{C}^N))$, where $u$ is the solution associated with $F$ defined in Theorem \ref{thm_LLT_dir}-(iv). If $s > 0$, we define the solution of the previous wave equation as in the case $s = 0$. In any case, one has
\[v \in \mathscr{C}^0((0, T), H^s(M, \mathbb{C}^N)) \cap \mathscr{C}^1((0, T), H^{s - 1}(M, \mathbb{C}^N)) \cap \mathscr{C}^2((0, T), H^{s - 2}(M, \mathbb{C}^N)),\]
$\partial_t^2 v = \PP{\mathscr{D}^\prime}^\ast v$ in $\mathscr{D}^\prime((0, T) \times M, \mathbb{C}^N)$, and there exists $C > 0$ such that
\[\sum_{j = 0}^2 \left\Vert \partial_t^j v \right\Vert_{L^\infty([0, T], H^{s - j})} \leq C \left\Vert f \right\Vert_{H^s((0, T) \times \partial M, \mathbb{C}^N)}, \quad f \in H^s((0, T) \times \partial M, \mathbb{C}^N).\]
If $s \geq 1$, then $v(t)_{\vert \partial M} = \left( \diag(\Theta) f \right)_{\vert \{t\} \times \partial M}$ in $H^{s - \frac{1}{2}}(\partial M, \mathbb{C}^N)$, in the sense of classical Sobolev trace operators. In addition, as $\Theta$ is compactly supported in $(0, T) \times \partial M$, one has $\left(v(T), \partial_t v(T)\right) \in \mathcal{K}_\ast^{s} \times \mathcal{K}_\ast^{s - 1}$, with the following duality equality: for $\left( u^0, u^1 \right) \in \mathcal{K}^{- s + 1} \times \mathcal{K}^{- s}$, if $u$ is the solution of
\[\left \{
\begin{array}{rcccl}
\partial_t^2 u - \mathsf{P} u & = & 0 & \quad & \text{in } (0, T) \times M, \\
\left( u(T, \cdot), \partial_t u(T, \cdot) \right) & = & \left( u^0, u^1 \right) & \quad & \text{in } M, \\
u & = & 0 & \quad & \text{on } (0, T) \times \partial M,
\end{array}
\right.\]
then
\begin{equation}\label{eq_thm_LLT_inhomogeneous_dual_eq}
\left\langle u^1, v(T) \right\rangle_{\mathcal{K}^{- s + 1}, \mathcal{K}_\ast^{s - 1}} - \left\langle u^0, \partial_t v(T) \right\rangle_{\mathcal{K}^{- s}, \mathcal{K}_\ast^{s}} =
\left\{ 
\begin{array}{cc}
\left\langle \partial_\nu u, \diag(\Theta) f \right\rangle_{H^{- s}, H_0^s} & \text{ if } s \geq 0 \\
\left\langle f, \diag(\Theta) \partial_\nu u \right\rangle_{H^s, H_0^{- s}} & \text{ if } s < 0
\end{array}
\right. .
\end{equation}
\end{thm}

The proof of Theorems \ref{thm_LLT_dir} and \ref{thm_LLT_inhomogeneous} is given in appendix.

\subsection{The duality argument}

Here, we prove Lemma \ref{lem_equivalence_obs_control}. The proof is based on the following classical result (see for example Corollary 11.20 of  \cite{LeRousseau}).

\begin{thm}\label{thm_banach}
Let $\mathcal{X}$ and $\mathcal{Y}$ be Hilbert spaces, and $K: \mathcal{X} \rightarrow \mathcal{Y}$ be a linear continuous operator. Then $K$ is surjective if and only if there exists $C > 0$ such that
\[\Vert y \Vert_{\mathcal{Y}} \leq C \Vert K^\ast y \Vert_{\mathcal{X}}, \quad y \in \mathcal{Y}.\]
\end{thm}

For $s \in \mathbb{R}$, if one denotes by $H_0^s((0, T) \times \partial M, \mathbb{C}^N)$ the closure of $\mathscr{C}^\infty_\mathrm{c}((0, T) \times \partial M, \mathbb{C}^N)$ in $H^s((0, T) \times \partial M, \mathbb{C}^N)$, then $H_0^s((0, T) \times \partial M, \mathbb{C}^N)$ is the dual of $H_0^{- s}((0, T) \times \partial M, \mathbb{C}^N)$ for all $s \in \mathbb{R}$. Consider $s \in \mathbb{R}$, $T > 0$ and $\Theta \in \mathscr{C}^\infty_\mathrm{c}((0, T) \times \partial M, \mathbb{C}^N)$. By Theorem \ref{thm_LLT_inhomogeneous}, one can define 
\[\begin{array}{cccc}
K: & H_0^s((0, T) \times \partial M, \mathbb{C}^N) & \longrightarrow & \mathcal{K}^{s} \times \mathcal{K}^{s - 1} \\
& f & \longmapsto & (v(T), \partial_t v(T))
\end{array}\]
where $v$ is the solution of
\[\left \{
\begin{array}{rcccl}
\partial_t^2 v - \mathsf{P}^\ast v & = & 0 & \quad & \text{in } (0, T) \times M, \\
\left( v(0, \cdot), \partial_t v(0, \cdot) \right) & = & 0 & \quad & \text{in } M, \\
v & = & \diag(\Theta) f & \quad & \text{on } (0, T) \times \partial M.
\end{array}
\right.\]
Note that in the definition of $H^s$-exact controllability, one can consider $f \in H_0^s((0, T) \times \partial M, \mathbb{C}^N)$ instead of $H^s(0, T) \times \partial M, \mathbb{C}^N)$. Hence, $H^s$-exact controllability for $\Theta$ holds if and only if the operator $K$ is surjective. By Theorem \ref{thm_LLT_inhomogeneous}, the adjoint of $K$ is given by
\[\begin{array}{cccc}
K^\ast: & \mathcal{K}^{- s + 1} \times \mathcal{K}^{- s} & \longrightarrow & H_0^{- s}((0, T) \times \partial M, \mathbb{C}^N) \\
& (u^1, u^0) & \longmapsto & \diag(\Theta) \partial_\nu u
\end{array}\]
where $u$ is the solution
\begin{equation}\label{eq_proof_thm_equivalence_obs_control_0}
\left \{
\begin{array}{rcccl}
\partial_t^2 u - \mathsf{P} u & = & 0 & \quad & \text{in } (0, T) \times M, \\
\left( u(T, \cdot), \partial_t u(T, \cdot) \right) & = & (-u^0, u^1) & \quad & \text{in } M, \\
u & = & 0 & \quad & \text{on } (0, T) \times \partial M.
\end{array}
\right.
\end{equation}

By Theorem \ref{thm_banach}, $H^s$-exact controllability for $\Theta$ is equivalent with the inequality
\begin{equation}\label{eq_proof_thm_equivalence_obs_control}
\left\Vert \left( u^0, u^1 \right) \right\Vert_{\mathcal{K}^{- s + 1} \times \mathcal{K}^{- s}} \lesssim \left\Vert \diag(\Theta) \partial_\nu u \right\Vert_{H^{- s}((0, T) \times \partial M, \mathbb{C}^{N})}, \quad \left( u^0, u^1 \right) \in \mathcal{K}^{- s + 1} \times \mathcal{K}^{- s},
\end{equation}
where $u$ is the solution of (\ref{eq_proof_thm_equivalence_obs_control_0}). One has
\[\left\Vert \left( u(0), \partial_t u(0) \right) \right\Vert_{\mathcal{K}^{- s + 1} \times \mathcal{K}^{- s}} \lesssim \left\Vert \left( u^0, u^1 \right) \right\Vert_{\mathcal{K}^{- s + 1} \times \mathcal{K}^{- s}} \lesssim \left\Vert \left( u(0), \partial_t u(0) \right) \right\Vert_{\mathcal{K}^{- s + 1} \times \mathcal{K}^{- s}},\]
for $\left( u^0, u^1 \right) \in \mathcal{K}^{- s + 1} \times \mathcal{K}^{- s}$, where $u$ is the solution of (\ref{eq_proof_thm_equivalence_obs_control_0}), implying that (\ref{eq_proof_thm_equivalence_obs_control}) and $H^{- s}$-observability are equivalent. This completes the proof of Lemma \ref{lem_equivalence_obs_control}.

\section{Ellipticity of the time-derivative on Neumann traces}

\subsection{Statement of the main estimate and beginning of the proof}

The main result of this section is the following theorem.

\begin{thm}[Ellipticity of the time-derivative on the Neumann trace]\label{thm_dt_ell_traces}
For $\Theta \in \mathscr{C}^\infty_\mathrm{c}((0, T) \times \partial M, \mathbb{C}^N)$, $s > - 1$, and $r \in \mathbb{N}^\ast$, there exists $C > 0$ such that for all $\left( u^0, u^1 \right) \in \mathcal{K}^{s + 1} \times \mathcal{K}^{s}$, one has
\[\left\Vert \diag(\Theta) \partial_\nu u \right\Vert_{H^s((0, T) \times \partial M, \mathbb{C}^N)} \leq C \left( \left\Vert \diag(\Theta) \partial_t^{2r} \partial_\nu u \right\Vert_{H^{s - 2r}((0, T) \times \partial M, \mathbb{C}^N)} + \left\Vert u^0 \right\Vert_{\mathcal{K}^{s + \frac{1}{2}}} + \left\Vert u^1 \right\Vert_{\mathcal{K}^{s - \frac{1}{2}}} \right)\]
where $u$ is the solution of
\[\left \{
\begin{array}{rcccl}
\partial_t^2 u - \mathsf{P}u & = & 0 & \quad & \text{in } \mathbb{R} \times M, \\
\left( u(0, \cdot), \partial_t u(0, \cdot) \right) & = & \left( u^0, u^1 \right) & \quad & \text{in } M, \\
u & = & 0 & \quad & \text{on } \mathbb{R} \times \partial M.
\end{array}
\right.\]
\end{thm}

\begin{rem}\label{rem_dt_ell_embeddings}
For clarity, embeddings have been omitted in the statement of Theorem \ref{thm_dt_ell_traces}. The notation $\left\Vert u^0 \right\Vert_{\mathcal{K}^{s + \frac{1}{2}}} + \left\Vert u^1 \right\Vert_{\mathcal{K}^{s - \frac{1}{2}}}$ stands for 
\[\left\Vert \iota_{\mathcal{K}^{s + 1} \rightarrow \mathcal{K}^{s + \frac{1}{2}}} u^0 \right\Vert_{\mathcal{K}^{s + \frac{1}{2}}} + \left\Vert \iota_{\mathcal{K}^{s} \rightarrow \mathcal{K}^{s - \frac{1}{2}}} u^1 \right\Vert_{\mathcal{K}^{s - \frac{1}{2}}}.\]
\end{rem}

\begin{proof}
Let $(O^j)_{j \in J}$ be a finite family of open subsets of $M$ satisfying the following properties:
\begin{enumerate}[label=(\roman*)]
\item One has 
\[\bigcup_{j \in J} \left( O^j \cap \partial M \right) = \partial M.\]
\item For each $j \in J$, there exists a smooth diffeomorphism $\kappa^j$ such that
\[\kappa^j: \tilde{O}^j \longrightarrow O^j \cap \partial M \]
where $\tilde{O}^j$ is a non-empty subset of $\mathbb{R}^{n-1}$.
\item We can use boundary normal coordinates on each $O^j$: more precisely, we assume that there exists $\delta > 0$ such that for all $j \in J$, the map
\[\begin{array}{cccc}
\tilde{\kappa}^j: & \tilde{O}^j \times [0, \delta) & \longrightarrow & O^j \\
& (x^\prime, x^n) & \longmapsto & \gamma(\nu_{\kappa^j(x^\prime)}, x^n)
\end{array}\]
is a smooth diffeomorphism, where for $y \in \partial M$, $\nu_y$ is the inward-pointing unit vector normal to the boundary at $y$, and $\gamma(\nu_{\kappa^j(x^\prime)}, \cdot)$ is the geodesic starting from $\kappa^j(x^\prime)$ and of initial velocity $\nu_{\kappa^j(x^\prime)}$.
\end{enumerate}

It is well-know that in the coordinates given by $\tilde{\kappa}^j$, the Laplace-Beltrami operator becomes an elliptic operator $\tilde{P}^j$ on $\mathbb{R}^n_+$ with principal part
\begin{equation}\label{eq_proof_dt_ell_main_1}
\tilde{P}^j = \partial_{x^n}^2 + \sum_{1 \leq p, q \leq n-1} \alpha_j^{pq}(x) \partial_{x^p} \partial_{x^q}.
\end{equation}
The coefficients $(\alpha_j^{pq})$ can be smoothly extended to $\mathbb{R}^n$ in such a way that $\tilde{P}^j$ is an elliptic operator on $\mathbb{R}^n$. 

We take a partition of the unity associated to the sets $(O^j)_{j \in J}$: there exists a family of functions $(\psi^j)_{j \in J}$ such that for each $j \in J$, $\psi^j \in \mathscr{C}^\infty_\mathrm{c}(O^j, [0, 1])$ and such that 
\[\sum_{j \in J} (\psi^j)^2 = 1\]
in a neighborhood of $\partial M$ in $M$. Also, take $\psi^0 \in \mathscr{C}^\infty_\mathrm{c}((0, T), [0, 1])$ such that $\psi^0 \Theta = \Theta$. 

Consider $\left( u^0, u^1 \right) \in \mathcal{K}^{s + 1} \times \mathcal{K}^{s}$. We start the proof by writing
\begin{align}
\left\Vert \diag(\Theta) \partial_\nu u \right\Vert_{H^s((0, T) \times \partial M, \mathbb{C}^N)} & = \left\Vert \diag(\Theta) \partial_\nu u \right\Vert_{H^s(\mathbb{R} \times \partial M, \mathbb{C}^N)} \nonumber \\
& = \left\Vert \psi^0 \diag(\Theta) \partial_\nu u \right\Vert_{H^s(\mathbb{R} \times \partial M, \mathbb{C}^N)} \nonumber \\
& \leq \sum_{j \in J} \left\Vert (\psi^j)^2 \psi^0 \diag(\Theta) \partial_\nu u \right\Vert_{H^s(\mathbb{R} \times (O^j \cap \partial M), \mathbb{C}^N)}. \label{eq_proof_dt_ell_main_2}
\end{align}
For $j \in J$, $t \in \mathbb{R}$ and $x \in \mathbb{R}^n$, we define
\[u^j(t, x) = \psi^0(t) \psi^j(\tilde{\kappa}^j(x)) u(t, \tilde{\kappa}^j(x)) \quad \text{ and } \quad \Theta^j(t, x^\prime) = \psi^j(\kappa^j(x^\prime)) \Theta(t, \kappa^j(x^\prime)).\]
Note that those functions are well-defined because $\psi^j$ is compactly supported in $O^j$. As $u_{\vert \partial M} = 0$, one has
\begin{equation}\label{eq_proof_dt_ell_main_3}
\partial_\nu u^j(t, x^\prime, 0) = \psi^0(t) \psi^j\left( \kappa^j(x^\prime) \right) \partial_\nu u\left(t, \kappa^j(x^\prime) \right).
\end{equation}

By definition of the $H^s-$norm on a Riemannian manifold, coming back to (\ref{eq_proof_dt_ell_main_2}), we thus have
\[\left\Vert \diag(\Theta) \partial_\nu u \right\Vert_{H^s((0, T) \times \partial M, \mathbb{C}^N)} \lesssim \sum_{j \in J} \left\Vert \diag(\Theta^j) \partial_\nu u^j \right\Vert_{H^s(\mathbb{R} \times \tilde{O}^j, \mathbb{C}^N)} = \sum_{j \in J} \left\Vert \diag(\Theta^j) \partial_\nu u^j \right\Vert_{H^s(\mathbb{R} \times \mathbb{R}^{n-1}, \mathbb{C}^N)}.\]
Recall that $(\pi_1, \cdots, \pi_N)$ denotes the projections associated with the canonical basis of $\mathbb{C}^N$. By definition of the $H^s(\mathbb{R} \times \mathbb{R}^{n-1}, \mathbb{C}^N)$-norm, one has
\begin{equation}\label{eq_proof_dt_ell_main_4}
\left\Vert \diag(\Theta) \partial_\nu u \right\Vert_{H^s((0, T) \times \partial M, \mathbb{C}^N)} \lesssim \sum_{k = 1}^N \sum_{j \in J} \left\Vert \left( \pi_k \Theta^j \right) \partial_\nu \left( \pi_k u^j \right) \right\Vert_{H^s(\mathbb{R} \times \mathbb{R}^{n-1})}.
\end{equation}

We see that we are reduced to the study of scalar functions defined on the half-space $\mathbb{R} \times \mathbb{R}^n_+$. We gather the properties satisfied by the functions $\pi_k u^j$. First, as $s \geq -2$, one has $\pi_k u^j \in H^{s + 1}(\mathbb{R} \times \mathbb{R}^n_+)$ by Theorem \ref{thm_LLT_dir}. Second, one has $\pi_k u^j(t, x^\prime, 0) = 0$ for all $(t, x^\prime) \in \mathbb{R} \times \mathbb{R}^{n - 1}$ by the Dirichlet boundary condition, and 
\[\partial_\nu \left( \pi_k u^j \right) \in H^s(\mathbb{R} \times \mathbb{R}^{n-1})\]
by Theorem \ref{thm_LLT_dir} and (\ref{eq_proof_dt_ell_main_3}). Third, we know that 
\[\partial_t^2 u - \Delta u = - Xu - qu\]
where $\Delta$ is the Laplace-Beltrami operator, so by the Leibniz formula, there exists a differential operator $R^j$ of order $1$, supported in $(0, T) \times O^j$ such that
\[\left( \partial_t^2 - \tilde{P}^j \right) u^j (t, x) = R^j u (t, \tilde{\kappa}^j(x))\]
where $\tilde{P}^j$ is defined by (\ref{eq_proof_dt_ell_main_1}). In particular, one has
\[\left( \partial_t^2 - \tilde{P}^j \right) (\pi_k u^j) \in H^s(\mathbb{R} \times \mathbb{R}^n_+).\]

\begin{prop}\label{prop_analysis_half_space}
Suppose 
\[P = \partial_t^2 - \partial_{x^n}^2 - \sum_{1 \leq i, j \leq n-1} \alpha^{ij}(x) \partial_{x^i} \partial_{x^j}\]
on $\mathbb{R}_t \times \mathbb{R}^n_x$, where the coefficients $(\alpha^{ij})$ are such that 
\[\xi_n^2 + \sum_{1 \leq i, j \leq n-1} \alpha^{ij}(x) \xi_i \xi_j\]
is uniformly elliptic on $\mathbb{R}^n$. Take $\theta \in \mathscr{C}^\infty_\mathrm{c}(\mathbb{R} \times \mathbb{R}^{n-1}, \mathbb{C})$. There exists $C > 0$ such that 
\begin{align*}
\left\Vert \theta \partial_\nu u \right\Vert_{H^s(\mathbb{R} \times \mathbb{R}^{n-1})} \leq \ & C \left( \left\Vert \theta \partial_t^{2r} \partial_\nu u \right\Vert_{H^{s - 2r}(\mathbb{R} \times \mathbb{R}^{n-1})} + \left\Vert Pu \right\Vert_{H^{s - \frac{1}{2}}(\mathbb{R} \times \mathbb{R}^n_+)} \right. \\
& \left. + \left\Vert \partial_\nu u \right\Vert_{H^{s - 1}(\mathbb{R} \times \mathbb{R}^{n-1})} + \Vert u \Vert_{H^{s + \frac{1}{2}}(\mathbb{R} \times \mathbb{R}^n_+)} \right).
\end{align*}
for all $u \in H^{s + 1}(\mathbb{R} \times \mathbb{R}^n_+)$ such that $Pu \in H^s(\mathbb{R} \times \mathbb{R}^n_+)$, $u_{\vert x_n = 0} = 0$ and $\partial_\nu u \in H^s(\mathbb{R} \times \mathbb{R}^{n-1})$.
\end{prop}

A proof is given in section 3.2. This proposition allows us to complete the proof of Theorem \ref{thm_dt_ell_traces}. One obtains
\begin{align*}
\left\Vert \left( \pi_k \Theta^j \right) \partial_\nu \pi_k u^j \right\Vert_{H^s(\mathbb{R} \times \mathbb{R}^{n-1})} \lesssim \ & \left\Vert \left( \pi_k \Theta^j \right) \partial_t^{2r} \partial_\nu \pi_k u^j \right\Vert_{H^{s - 2r}(\mathbb{R} \times \mathbb{R}^{n-1})} + \left\Vert \left( \partial_t^2 - \tilde{P}^j \right) (\pi_k u^j) \right\Vert_{H^{s - \frac{1}{2}}(\mathbb{R} \times \mathbb{R}^n_+)} \\
& + \left\Vert \partial_\nu \pi_k u^j \right\Vert_{H^{s - 1}(\mathbb{R} \times \mathbb{R}^{n-1})} + \left\Vert \pi_k u^j \right\Vert_{H^{s + \frac{1}{2}}(\mathbb{R} \times \mathbb{R}^n_+)}.
\end{align*}
Using (\ref{eq_proof_dt_ell_main_4}) and the definition of the $H^s$-norms of vectors, one has
\begin{align*}
& \left\Vert \diag(\Theta) \partial_\nu u \right\Vert_{H^s((0, T) \times \partial M, \mathbb{C}^N)} \lesssim \sum_{j \in J} \left\Vert \diag(\Theta^j) \partial_t^{2r} \partial_\nu u^j \right\Vert_{H^{s - 2r}(\mathbb{R} \times \mathbb{R}^{n-1}, \mathbb{C}^N)} \\
& \quad + \sum_{j \in J} \left( \left\Vert \left( \partial_t^2 - \tilde{P}^j \right) u^j \right\Vert_{H^{s - \frac{1}{2}}(\mathbb{R} \times \mathbb{R}^n_+, \mathbb{C}^N)} + \left\Vert \partial_\nu u^j \right\Vert_{H^{s - 1}(\mathbb{R} \times \mathbb{R}^{n-1}), \mathbb{C}^N)} + \left\Vert u^j \right\Vert_{H^{s + \frac{1}{2}}(\mathbb{R} \times \mathbb{R}^n_+, \mathbb{C}^N)} \right).
\end{align*}
We estimate the terms on the right-hand side one by one.

\paragraph{First term.}

We prove
\begin{align}
& \sum_{j \in J} \left\Vert \diag(\Theta^j) \partial_t^{2r} \partial_\nu u^j \right\Vert_{H^{s - 2r}(\mathbb{R} \times \mathbb{R}^{n-1}, \mathbb{C}^N)} \nonumber \\
\lesssim & \left\Vert \diag(\Theta) \partial_t^{2r} \partial_\nu u \right\Vert_{H^{s - 2r}((0, T) \times \partial M, \mathbb{C}^N)} + \left\Vert \partial_\nu u \right\Vert_{H^{s - 1}((0, T) \times \partial M, \mathbb{C}^N)}, \label{eq_proof_dt_ell_main_5}
\end{align}
meaning that the first term yields the main term of the estimate up to a remainder term.

Using (\ref{eq_proof_dt_ell_main_3}) and the Leibniz formula, one finds
\begin{align*}
& \left\Vert \diag(\Theta^j) \partial_t^{2r} \partial_\nu u^j \right\Vert_{H^{s - 2r}(\mathbb{R} \times \mathbb{R}^{n-1}, \mathbb{C}^N)} \\
\lesssim & \sum_{k = 0}^{2r} \left\Vert \diag(\Theta^j) \left( \psi^j \circ \kappa^j \right) \partial_t^k \psi^0 \partial_t^{2r - k} \partial_\nu u (t, \kappa^j(x^\prime)) \right\Vert_{H^{s - 2r}(\mathbb{R} \times \mathbb{R}^{n-1}, \mathbb{C}^N)}.
\end{align*}
As $\psi^j$ is supported in a coordinate chart of $\partial M$, one has
\[\left\Vert \diag(\Theta^j) \partial_t^{2r} \partial_\nu u^j \right\Vert_{H^{s - 2r}(\mathbb{R} \times \mathbb{R}^{n-1}, \mathbb{C}^N)} \lesssim \sum_{k = 0}^{2r} \left\Vert \diag(\Theta) \left(\psi^j\right)^2 \partial_t^k \psi^0 \partial_t^{2r - k} \partial_\nu u \right\Vert_{H^{s - 2r}((0, T) \times \partial M, \mathbb{C}^N)}.\]

One has
\[\left\Vert \diag(\Theta) \left(\psi^j\right)^2 \psi^0 \partial_t^{2r} \partial_\nu u \right\Vert_{H^{s - 2r}((0, T) \times \partial M, \mathbb{C}^N)} \lesssim \left\Vert \diag(\Theta) \partial_t^{2r} \partial_\nu u \right\Vert_{H^{s - 2r}((0, T) \times \partial M, \mathbb{C}^N)},\]
and for $k \in \llbracket 1, 2r \rrbracket$,
\[\left\Vert \diag(\Theta) \left(\psi^j\right)^2 \partial_t^k \psi^0 \partial_t^{2r - k} \partial_\nu u \right\Vert_{H^{s - 2r}((0, T) \times \partial M, \mathbb{C}^N)} \lesssim \left\Vert \partial_\nu u \right\Vert_{H^{s - 1}((0, T) \times \partial M, \mathbb{C}^N)}.\]
This gives (\ref{eq_proof_dt_ell_main_5}).

\paragraph{Second term.}

For the second term, one has
\[\sum_{j \in J} \left\Vert \left( \partial_t^2 - \tilde{P}^j \right) u^j \right\Vert_{H^{s - \frac{1}{2}}(\mathbb{R} \times \mathbb{R}^n_+, \mathbb{C}^N)} \lesssim \Vert u \Vert_{H^{s + \frac{1}{2}}((0, T) \times M, \mathbb{C}^N)}.\]
This holds since for all $j$, there exists a differential operator $R^j$ of order $1$, supported in $(0, T) \times O^j$ such that
\[\left( \partial_t^2 - \tilde{P}^j \right) u^j (t, x) = R^j u (t, \tilde{\kappa}^j(x)).\]

\paragraph{Third and forth term.}

Arguing as above, one finds 
\begin{align*}
& \sum_{j \in J} \left( \left\Vert \partial_\nu u^j \right\Vert_{H^{s - 1}(\mathbb{R} \times \mathbb{R}^{n-1}, \mathbb{C}^N)} + \left\Vert u^j \right\Vert_{H^{s + \frac{1}{2}}(\mathbb{R} \times \mathbb{R}^n_+, \mathbb{C}^N)} \right) \\
\lesssim \ & \left\Vert \partial_\nu u \right\Vert_{H^{s - 1}((0, T) \times \partial M, \mathbb{C}^N)} + \left\Vert u \right\Vert_{H^{s + \frac{1}{2}}((0, T) \times M, \mathbb{C}^N)}.
\end{align*}

\paragraph{Conclusion.}

Gathering all our estimates, one finds
\begin{align*}
\left\Vert \diag(\Theta) \partial_\nu u \right\Vert_{H^{s - 2r}((0, T) \times \partial M, \mathbb{C}^N)} \lesssim \ & \left\Vert \diag(\Theta) \partial_t^{2r} \partial_\nu u \right\Vert_{H^{s - 2r}((0, T) \times \partial M, \mathbb{C}^N)} \\
& + \left\Vert \partial_\nu u \right\Vert_{H^{s - 1}((0, T) \times \partial M, \mathbb{C}^N)} + \left\Vert u \right\Vert_{H^{s + \frac{1}{2}}((0, T) \times M, \mathbb{C}^N)}.
\end{align*}
By Theorem \ref{thm_LLT_dir}, one has
\[\left\Vert \partial_\nu u \right\Vert_{H^{s - 1}((0, T) \times \partial M, \mathbb{C}^N)} \lesssim \Vert u^0 \Vert_{\mathcal{K}^{s}} + \Vert u^1 \Vert_{\mathcal{K}^{s - 1}}\]
and
\[\left\Vert u \right\Vert_{H^{\left( s - \frac{1}{2} \right) + 1}((0, T) \times M, \mathbb{C}^N)} \lesssim \Vert u^0 \Vert_{\mathcal{K}^{s + \frac{1}{2}}} + \Vert u^1 \Vert_{\mathcal{K}^{s - \frac{1}{2}}},\]
as $s - \frac{1}{2} \geq -2$. This completes the proof.
\end{proof}

\subsection{Analysis in a half-space.}

Here, we prove Proposition \ref{prop_analysis_half_space}. Write $\mathsf{S}^m(\mathbb{R}_t \times \mathbb{R}^n_x)$ for the set of symbols of order $m$, $\mathsf{S}_{\intercal}^m(\mathbb{R}_t \times \mathbb{R}^n_x)$ for the set of tangential symbols of order $m$, and $\Psi^m(\mathbb{R}_t \times \mathbb{R}^n_x)$ and $\Psi^m_{\intercal}(\mathbb{R}_t \times \mathbb{R}^n_x)$ for the associated sets of pseudo-differential operators. Let $p$ be the principal symbol of $P$, that is 
\[p(x, \tau, \xi) = - \tau^2 + \xi_n^2 + \sum_{1 \leq i, j \leq n-1} \alpha^{ij}(x) \xi_i \xi_j.\]
Write
\[\vert \xi^\prime \vert_x^2 = \sum_{1 \leq i, j \leq n-1} \alpha^{ij}(x) \xi_i \xi_j \quad \text{ and } \quad \rho(x, \tau, \xi^\prime) = - \tau^2 + \vert \xi^\prime \vert_x^2\]
so that $p(x, \tau, \xi) = \xi_n^2 + \rho(x, \tau, \xi^\prime)$. On the boundary, we sometimes use the notation $\vert \xi^\prime \vert_{x^\prime} = \vert \xi^\prime \vert_{(x^\prime, 0)}$.

Consider $u \in H^{s + 1}(\mathbb{R} \times \mathbb{R}^n_+)$ satisfying the assumptions of Proposition \ref{prop_analysis_half_space}. The idea of the proof is to split miscellaneously $\theta \partial_\nu u$ into two terms: one on which $\partial_t$ is elliptic, and one on which the wave operator is elliptic. More precisely, we fix $\chi_0 \in \mathscr{C}^\infty(\mathbb{R}_+, [0, 1])$ such that $\chi_0 = 1$ on $\left[ \frac{1}{4}, + \infty \right)$ and $\chi_0 = 0$ on $\left[ 0, \frac{1}{5} \right]$, and we define
\[\chi(x, \tau, \xi^\prime) = \chi_0 \left( \frac{\rho(x, \tau, \xi^\prime)}{1 + \tau^2 + \vert \xi^\prime \vert_x^2}\right).\]
Then $\chi \in \mathsf{S}_{\intercal}^0(\mathbb{R}_t \times \mathbb{R}^n_x)$ by Lemma 18.1.10 of \cite{HormanderIII}. One has $\chi(x, \tau, \xi^\prime) = 1$ if $1 + \tau^2 + \vert \xi^\prime \vert_x^2 \leq 4 \rho(x, \tau, \xi^\prime)$, and $\chi$ is supported in the set
\[\left\{(x, \tau, \xi^\prime) \in \mathbb{R}^n \times \mathbb{R} \times \mathbb{R}^n, 1 + \tau^2 + \vert \xi^\prime \vert_x^2 \leq 5 \rho(x, \tau, \xi^\prime) \right\}.\]
One can write
\[\left\Vert \theta \partial_\nu u \right\Vert^2_{H^s(\mathbb{R} \times \mathbb{R}^{n-1})} \lesssim \left\Vert \theta \OpT(1 - \chi_{\vert x^n = 0})(\partial_\nu u) \right\Vert^2_{H^s(\mathbb{R} \times \mathbb{R}^{n-1})} + \left\Vert \OpT(\chi_{\vert x^n = 0})(\partial_\nu u) \right\Vert^2_{H^s(\mathbb{R} \times \mathbb{R}^{n-1})}.\]
We will study the two terms on the right-hand side separately: for the first one, $\partial_t$ turns out to be elliptic, and for the second one, $P$ is elliptic.

\begin{rem}
We use the notation $\OpT$ both for tangential pseudo-differential operators on $\mathbb{R}_t \times \mathbb{R}^n_x$ and for pseudo-differential operators on the boundary $\mathbb{R}_t \times \mathbb{R}^{n-1}_{x^\prime}$. They coincide at $x_n = 0$ up to a $\frac{1}{2\pi}$ factor.
\end{rem}

\subsubsection{Ellipticity of the time-derivative}

We prove the following estimate.

\begin{lem}
There exists $C > 0$ such that 
\[\left\Vert \theta \OpT(1 - \chi_{\vert x^n = 0})( \partial_\nu u) \right\Vert^2_{H^s(\mathbb{R} \times \mathbb{R}^{n-1})} \leq C \left( \left\Vert \theta \partial_t^{2r} \partial_\nu u \right\Vert^2_{H^{s - 2r}(\mathbb{R} \times \mathbb{R}^{n-1})} + \left\Vert \partial_\nu u \right\Vert^2_{H^{s - 1}(\mathbb{R} \times \mathbb{R}^{n-1})} \right).\]
\end{lem}

\begin{proof}
Fix $\chi_1 = \chi_1(x^\prime, \tau, \xi^\prime)$ a smooth compactly supported function such that $\chi_1(x^\prime, \tau, \xi^\prime) = 1$ if $\tau^2 + \vert \xi^\prime \vert_{x^\prime}^2 \leq 1$. Write 
\begin{align*}
\left\Vert \theta \OpT(1 - \chi_{\vert x^n = 0})(\partial_\nu u) \right\Vert^2_{H^s(\mathbb{R} \times \mathbb{R}^{n-1})} \leq & \ \left\Vert \theta \OpT\left(\chi_1 (1 - \chi_{\vert x^n = 0})\right)(\partial_\nu u) \right\Vert^2_{H^s(\mathbb{R} \times \mathbb{R}^{n-1})} \\
& + \left\Vert \theta \OpT\left((1 - \chi_1) (1 - \chi_{\vert x^n = 0})\right)(\partial_\nu u) \right\Vert^2_{H^s(\mathbb{R} \times \mathbb{R}^{n-1})}.
\end{align*}
As $\chi_1 (1 - \chi_{\vert x^n = 0}) \in \mathsf{S}^{- \infty}(\mathbb{R}_t \times \mathbb{R}^{n - 1}_{x^\prime})$, one has
\[\left\Vert \theta \OpT\left(\chi_1 (1 - \chi_{\vert x^n = 0})\right)(\partial_\nu u) \right\Vert^2_{H^s(\mathbb{R} \times \mathbb{R}^{n-1})} \lesssim \left\Vert \partial_\nu u \right\Vert^2_{H^{s - 1}(\mathbb{R} \times \mathbb{R}^{n-1})}.\]
Thus, to complete the proof of the lemma, it suffices to show that
\begin{align}
& \left\Vert \theta \OpT\left((1 - \chi_1) (1 - \chi_{\vert x^n = 0})\right)(\partial_\nu u) \right\Vert^2_{H^s(\mathbb{R} \times \mathbb{R}^{n-1})} \nonumber \\ 
\lesssim \ & \left\Vert \theta \partial_t^{2r} \partial_\nu u \right\Vert^2_{H^{s - 2r}(\mathbb{R} \times \mathbb{R}^{n-1})} + \left\Vert \partial_\nu u \right\Vert^2_{H^{s - 1}(\mathbb{R} \times \mathbb{R}^{n-1})}. \label{eq_dt_ell_main}
\end{align}

Set $\chi_2 = (1 - \chi_1) (1 - \chi_{\vert x^n = 0}) \in \mathsf{S}^0(\mathbb{R}_t \times \mathbb{R}^{n - 1}_{x^\prime})$. If $(x^\prime, \tau, \xi^\prime) \in \supp \chi_2$, then $\tau^2 + \vert \xi^\prime \vert_{x^\prime}^2 > 1$ and
\[1 + \tau^2 + \vert \xi^\prime \vert_{x^\prime}^2 > 4 \left( \vert \xi^\prime \vert_{x^\prime}^2 - \tau^2 \right).\]
Combining those two inequalities, one finds 
\[2 \left(\tau^2 + \vert \xi^\prime \vert_{x^\prime}^2 \right) + 1 + \tau^2 + \vert \xi^\prime \vert_{x^\prime}^2 > 2 + 4 \left( \vert \xi^\prime \vert_{x^\prime}^2 - \tau^2 \right)\]
that is
\begin{equation}\label{eq_dt_ell_def_chi_tilde_3}
7 \tau^2 > 1 + \vert \xi^\prime \vert_{x^\prime}^2.
\end{equation}
In particular, $\tilde{\chi}_2(x^\prime, \tau, \xi^\prime) = \tau^{-2r} \chi_2 (x^\prime, \tau, \xi^\prime)$ is well-defined. Using (\ref{eq_dt_ell_def_chi_tilde_3}), one finds $\tilde{\chi}_2 \in \mathsf{S}^{- 2 r}(\mathbb{R}_t \times \mathbb{R}^{n - 1}_{x^\prime})$. Since $\OpT(\chi_2) = \OpT( \tilde{\chi}_2) \partial_t^{2r}$, one has
\begin{align*}
\left\Vert \theta \OpT(\chi_2)(\partial_\nu u) \right\Vert^2_{H^s(\mathbb{R} \times \mathbb{R}^{n-1})} \leq \ & \left\Vert \OpT( \tilde{\chi}_2) \left( \theta \partial_t^{2r} \partial_\nu u \right) \right\Vert^2_{H^s(\mathbb{R} \times \mathbb{R}^{n-1})} \\
& + \left\Vert \left[ \theta, \OpT( \tilde{\chi}_2) \right] \partial_t^{2r} \partial_\nu u \right\Vert^2_{H^s(\mathbb{R} \times \mathbb{R}^{n-1})}
\end{align*}
and this gives (\ref{eq_dt_ell_main}), as $\left[ \theta, \OpT( \tilde{\chi}_2) \right] \partial_t^{2r} \in \Psi^{- 1}(\mathbb{R}_t \times \mathbb{R}^{n - 1}_{x^\prime})$.
\end{proof}

\subsubsection{Ellipticity of the wave operator}

We denote by $\underline{u}$ and $\underline{f}$ the extensions by $0$ of $u$ and $f$ on the whole space. In the sense of distributions on $\mathbb{R}^{n+1}$, since $u_{\vert x_n = 0} = 0$, one has
\[P \underline{u} = \underline{f} + \delta_{x^n = 0} \otimes \partial_\nu u,\]
and this holds in fact in $\mathscr{S}^\prime(\mathbb{R}^{n+1})$. As $\OpT(\chi)$ sends $\mathscr{S}^\prime(\mathbb{R}^{n+1})$ to $\mathscr{S}^\prime(\mathbb{R}^{n+1})$, one has 
\begin{equation}\label{eq_dt_ell_trace_1}
P \OpT(\chi) \underline{u} + \left[ \OpT(\chi), P \right] \underline{u} = \OpT(\chi) \underline{f} + \delta_{x^n = 0} \otimes \left( \OpT(\chi_{\vert x^n = 0}) \partial_\nu u \right)
\end{equation}
in $\mathscr{S}^\prime(\mathbb{R}^{n+1})$.

To get an estimate on $\OpT(\chi_{\vert x^n = 0}) \partial_\nu u$, we apply a parametrix of $P$. Thus, one has to find a non-tangential symbol $\tilde{\chi}$ of order $0$ supported where $P$ is elliptic, and such that $\tilde{\chi}(x, \tau, \xi) = 1$ if $(x, \tau, \xi^\prime) \in \supp \chi$. If $(x, \tau, \xi)$ is such that $(x, \tau, \xi^\prime) \in \supp \chi$, then 
\[1 + \tau^2 + \vert \xi^\prime \vert_x^2 \leq 5 \left( \vert \xi^\prime \vert_x^2 - \tau^2 \right)\]
and this implies
\[1 + \tau^2 + \xi_n^2 +\vert \xi^\prime \vert_x^2 \leq 5 \left( \vert \xi^\prime \vert_x^2 + \xi_n^2 - \tau^2 \right) = 5 p(x, \tau, \xi).\]
Set
\[\tilde{\chi}(x, \tau, \xi) = \eta\left( \frac{p(x, \tau, \xi)}{1 + \xi_n^2 + \vert \xi^\prime \vert_x^2 + \tau^2} \right)\]
where $\eta \in \mathscr{C}^\infty(\mathbb{R}_+, [0, 1])$ is such that $\eta(\sigma) = 1$ if $\sigma \geq \frac{1}{5}$, and $\eta(\sigma) = 0$ if $\sigma \leq \frac{1}{10}$. Then $\tilde{\chi}$ is supported where $P$ is elliptic, and $\tilde{\chi}(x, \tau, \xi) = 1$ if $(x, \tau, \xi^\prime) \in \supp \chi$. The function $\tilde{\chi}$ is a symbol of order 0 by Lemma 18.1.10 of \cite{HormanderIII}.

Set $Q = \Op(q) \in \Psi^{- 2}(\mathbb{R}_t \times \mathbb{R}^n_x)$, with
\[q(x, \tau, \xi) = \frac{\tilde{\chi}(x, \tau, \xi)}{p(x, \tau, \xi)}.\]
Pseudo-differential calculus gives $QP = \Op(\tilde{\chi}) + R_1$, with $R_1 \in \Psi^{- 1}(\mathbb{R}_t \times \mathbb{R}^n_x)$. Note that one can construct $\tilde{Q} \in \Psi^{- 2}(\mathbb{R}_t \times \mathbb{R}^n_x)$ such that 
\[\tilde{Q} P - \Op(\tilde{\chi}) \in \Psi^{- \infty}(\mathbb{R}_t \times \mathbb{R}^n_x)\]
as in Theorem 18.1.9 of \cite{HormanderIII}, but such a refinement is not needed here.

Applying $Q$ to Equation (\ref{eq_dt_ell_trace_1}), one finds
\begin{equation}\label{eq_dt_ell_trace_2}
\Op(\tilde{\chi}) \OpT(\chi) \underline{u} + R_1 \OpT(\chi) \underline{u} + Q \left[ \OpT(\chi), P \right] \underline{u} = Q \OpT(\chi) \underline{f} + Q \left( \delta_{x^n = 0} \otimes \left( \OpT(\chi_{\vert x^n = 0}) \partial_\nu u \right) \right).
\end{equation}
Since $u_{\vert x_n = 0} = 0$, one has
\[\Op(\tilde{\chi}) \OpT(\chi) \underline{u}_{\vert x_n = 0} = \OpT(\chi) \underline{u}_{\vert x_n = 0} + \Op(\tilde{\chi} - 1) \OpT(\chi) \underline{u}_{\vert x_n = 0} = \Op(\tilde{\chi} - 1) \OpT(\chi) \underline{u}_{\vert x_n = 0}.\]
Thus, computing the trace of (\ref{eq_dt_ell_trace_2}) at $x^n = 0$ gives
\begin{equation}\label{eq_dt_ell_trace_3}
Q \left( \delta_{x^n = 0} \otimes \left( \OpT(\chi_{\vert x^n = 0}) \partial_\nu u \right) \right)_{\vert x_n = 0} = - Q \OpT(\chi) \underline{f}_{\vert x_n = 0} + R_3 u_{\vert x_n = 0}
\end{equation}
where the rest $R_3 u$ is
\[R_3 u = \Op(\tilde{\chi} - 1) \OpT(\chi) \underline{u} + R_1 \OpT(\chi) \underline{u} + Q \left[ \OpT(\chi), P \right] \underline{u}.\]

\begin{lem}\label{lem_right_side_proof_dt_ell}
There exists $C > 0$ such that
\begin{equation}\label{eq_dt_ell_trace_4}
\left\Vert Q \OpT(\chi) \underline{f}_{\vert x^n = 0} \right\Vert_{H^{s + 1}(\mathbb{R} \times \mathbb{R}^{n-1})} \leq C \left( \Vert f \Vert_{H^{s - \frac{1}{2}}(\mathbb{R} \times \mathbb{R}^n_+)} \right)
\end{equation}
and
\begin{equation}\label{eq_dt_ell_trace_5}
\left\Vert R_3 u_{\vert x^n = 0} \right\Vert_{H^{s + 1}(\mathbb{R} \times \mathbb{R}^{n-1})} \leq C \left( \left\Vert \partial_\nu u \right\Vert_{H^{s - 1}(\mathbb{R} \times \mathbb{R}^{n-1})} + \Vert u \Vert_{H^{s + \frac{1}{2}}(\mathbb{R} \times \mathbb{R}^n_+)} \right).
\end{equation}
\end{lem}

\begin{proof}
As $s > - 1$ and as $Q$ is of order $-2$, one has
\[\left\Vert Q \OpT(\chi) \underline{f}_{\vert x^n = 0} \right\Vert_{H^{s + 1}(\mathbb{R} \times \mathbb{R}^{n-1})} \lesssim \left\Vert Q \OpT(\chi) \underline{f} \right\Vert_{H^{s + \frac{3}{2}}(\mathbb{R} \times \mathbb{R}^n_+)} \lesssim \left\Vert \OpT(\chi) \underline{f} \right\Vert_{H^{s - \frac{1}{2}}(\mathbb{R} \times \mathbb{R}^n_+)}.\]
As $\chi \in \Psi_\intercal^{0}(\mathbb{R}_t \times \mathbb{R}^n_x)$, one obtains (\ref{eq_dt_ell_trace_4}). Next, we prove (\ref{eq_dt_ell_trace_5}).

\paragraph{Term 1.} For the term $\Op(\tilde{\chi} - 1) \OpT(\chi) \underline{u}$, we use Theorem 18.1.35 of \cite{HormanderIII}.

\begin{lem}
The symbol $1 - \tilde{\chi}$ satisfies the assumption of Theorem 18.1.35 of \cite{HormanderIII}: there exists $\epsilon > 0$ such that 
\[1 - \tilde{\chi}(x, \tau, \xi) = 0\]
if $\epsilon \vert \xi_n\vert > 1$ and $\vert (\tau, \xi^\prime)\vert \leq \epsilon \vert \xi_n \vert$.
\end{lem}

\begin{proof}
There exists $C > 0$ such that
\[\vert \xi^\prime \vert_x^2 \leq c \vert \xi^\prime\vert ^2\]
for all $(x, \xi^\prime)$. Hence, if $\epsilon \vert \xi_n\vert > 1$ and $\vert (\tau, \xi^\prime)\vert \leq \epsilon \vert \xi_n \vert$, then 
\begin{align*}
\frac{\xi_n^2 + \vert \xi^\prime \vert_x^2 - \tau^2}{1 + \xi_n^2 + \vert \xi^\prime \vert_x^2 + \tau^2} & \geq \frac{\xi_n^2 - \tau^2}{1 + \xi_n^2 + \vert \xi^\prime \vert_x^2 + \tau^2} \geq \frac{\xi_n^2 - \epsilon^2 \xi_n^2 }{1 + \xi_n^2 + c \epsilon^2 \xi_n^2 + \epsilon^2 \xi_n^2} \geq \frac{\xi_n^2 - \epsilon^2 \xi_n^2 }{\epsilon^2 \xi_n^2 + \xi_n^2 + c \epsilon^2 \xi_n^2 + \epsilon^2 \xi_n^2} \\
& \geq \frac{1 - \epsilon^2}{1 + c \epsilon^2 + 2 \epsilon^2}.
\end{align*}
Thus, if $\epsilon$ is sufficiently small, one has
\[\frac{\xi_n^2 + \vert \xi^\prime \vert_x^2 - \tau^2}{1 + \xi_n^2 + \vert \xi^\prime \vert_x^2 + \tau^2} \geq \frac{1}{5}\]
implying $\tilde{\chi}(x, \tau, \xi) = 1$.
\end{proof}

Thus, one has $\Op(\tilde{\chi} - 1) \OpT(\chi) \in \Psi^0(\mathbb{R}_t \times \mathbb{R}^n_x)$, with vanishing symbol. Hence, $\Op(\tilde{\chi} - 1) \OpT(\chi) \in \Psi^{- \infty}(\mathbb{R}_t \times \mathbb{R}^n_x)$, yielding
\[\left\Vert \Op(\tilde{\chi} - 1) \OpT(\chi) \underline{u}_{\vert x^n = 0} \right\Vert_{H^{s + 1}(\mathbb{R} \times \mathbb{R}^{n-1})} \lesssim \left\Vert \Op(\tilde{\chi} - 1) \OpT(\chi) \underline{u} \right\Vert_{H^{s + \frac{3}{2}}(\mathbb{R} \times \mathbb{R}^n_+)} \lesssim \left\Vert u \right\Vert_{H^{s + \frac{3}{2} - N}(\mathbb{R} \times \mathbb{R}^n_+)}\]
for any $N \geq 0$.

\paragraph{Term 2.} For the term $R_1 \OpT(\chi) \underline{u}$, as $R_1 \in \Psi^{- 1}(\mathbb{R}_t \times \mathbb{R}^n_x)$, one has
\[\left\Vert R_1 \OpT(\chi) \underline{u}_{\vert x^n = 0} \right\Vert_{H^{s + 1}(\mathbb{R} \times \mathbb{R}^{n-1})} \lesssim \left\Vert R_1 \OpT(\chi) \underline{u} \right\Vert_{H^{s + \frac{3}{2}}(\mathbb{R} \times \mathbb{R}^n_+)} \lesssim \left\Vert u \right\Vert_{H^{s + \frac{1}{2}}(\mathbb{R} \times \mathbb{R}^n_+)}.\]

\paragraph{Term 3.} Pseudo-differential calculus gives $Q \left[ \OpT(\chi), P \right] \in \Psi^{- 1}(\mathbb{R}_t \times \mathbb{R}^n_x)$, yielding
\begin{align*}
\left\Vert Q \left[ \OpT(\chi), P \right] \OpT(\chi) \underline{u}_{\vert x^n = 0} \right\Vert_{H^{s + 1}(\mathbb{R} \times \mathbb{R}^{n-1})} & \lesssim \left\Vert Q \left[ \OpT(\chi), P \right] \OpT(\chi) \underline{u} \right\Vert_{H^{s + \frac{3}{2}}(\mathbb{R} \times \mathbb{R}^n_+)} \\
& \lesssim \left\Vert u \right\Vert_{H^{s + \frac{1}{2}}(\mathbb{R} \times \mathbb{R}^n_+)}.
\end{align*}
Gathering all those estimates, one finds (\ref{eq_dt_ell_trace_5}). This completes the proof of Lemma \ref{lem_right_side_proof_dt_ell}.
\end{proof}

We now turn to the study of the left-hand side of (\ref{eq_dt_ell_trace_3}).

\begin{lem}\label{lem_left_side_proof_dt_ell}
There exists $C > 0$ such that
\begin{align*}
& \left\Vert \OpT(\chi_{\vert x^n = 0}) \partial_\nu u \right\Vert_{H^s(\mathbb{R} \times \mathbb{R}^{n-1})} \\
\leq \ & C \left( \left\Vert Q \left( \delta_{x^n = 0} \otimes \left( \OpT(\chi_{\vert x^n = 0}) \partial_\nu u \right) \right)_{\vert x^n = 0} \right\Vert_{H^{s + 1}(\mathbb{R} \times \mathbb{R}^{n-1})} + \left\Vert \partial_\nu u \right\Vert_{H^{s - 1}(\mathbb{R} \times \mathbb{R}^{n-1})} \right).
\end{align*}
\end{lem}

\begin{proof}
The idea is to find a pseudo-differential expression of 
\[\Op(q) \left( \delta_{x^n = 0} \otimes \left( \OpT(\chi_{\vert x^n = 0}) \partial_\nu u \right) \right)_{\vert x^n = 0}.\]
By definition, one has
\begin{equation}\label{proof_lem_main_term_OpT_chi_eq1}
\Op(q) \left( \delta_{x^n = 0} \otimes \left( \OpT(\chi_{\vert x^n = 0}) \partial_\nu u \right) \right)_{\vert x_n = 0} (t, x^\prime) = \OpT(q_\intercal) \OpT(\chi_{\vert x^n = 0}) \partial_\nu u(t, x) 
\end{equation}
where $q_\intercal$ is the symbol 
\[q_\intercal(x^\prime, \tau, \xi^\prime) = \int_\mathbb{R} \frac{\tilde{\chi}(x^\prime, 0, \tau, \xi)}{\xi_n^2 + \rho(x^\prime, 0, \tau, \xi^\prime)} \mathrm{d}\xi_n.\]

\begin{lem}\label{lem_q_T_symbol}
One has $q_\intercal \in \mathsf{S}^{-1}(\mathbb{R}_t \times \mathbb{R}^{n - 1}_{x^\prime})$.
\end{lem}

\begin{proof}
Note that an explicit formula for $q_\intercal$ is not needed. The idea of the proof is to write $q_\intercal = a \times \left( F \circ b \right)$ where $a$ is a symbol of order $-1$, $F$ is a smooth function, and $b$ is a symbol of order $0$, so that the conclusion will be a consequence of Lemma 18.1.10 of \cite{HormanderIII}. Recall that 
\[\tilde{\chi}(x, \tau, \xi) = \eta\left( \frac{p(x, \tau, \xi)}{1 + \xi_n^2 + \vert \xi^\prime \vert_x^2 + \tau^2} \right)\]
where $\eta$ is a real nonnegative smooth function such that $\eta(\sigma) = 1$ if $\sigma \geq \frac{1}{5}$, and $\eta(\sigma) = 0$ if $\sigma \leq \frac{1}{10}$. Write $\vert \xi^\prime \vert_{x^\prime}$ instead of $\vert \xi^\prime \vert_{(x^\prime, 0)}$, and set
\[b(x^\prime, \tau, \xi^\prime) = \frac{\vert \xi^\prime \vert_{x^\prime}^2 - \tau^2}{1 + \vert \xi^\prime \vert_{x^\prime}^2 + \tau^2}.\]
Then $b \in \mathsf{S}^0(\mathbb{R}_t \times \mathbb{R}^{n - 1}_{x^\prime})$ and a change of variable gives
\[q_\intercal(x^\prime, \tau, \xi^\prime) = \frac{1}{\sqrt{1 + \vert \xi^\prime \vert_{x^\prime}^2 + \tau^2}} \int_\mathbb{R} \frac{1}{\sigma^2 + b(x^\prime, \tau, \xi^\prime)} \eta\left( \frac{\sigma^2 + b(x^\prime, \tau, \xi^\prime)}{\sigma^2 + 1} \right) \mathrm{d}\sigma.\]
Set 
\[F(\sigma^\prime) = \int_\mathbb{R} \frac{1}{\sigma^2 + \sigma^\prime} \eta\left( \frac{\sigma^2 + \sigma^\prime}{\sigma^2 + 1} \right) \mathrm{d}\sigma.\]

\begin{lem}
The function $F$ is smooth.
\end{lem}

\begin{proof}
Consider $\sigma^\prime \in \mathbb{R}$ and write 
\[g(\sigma) = \frac{\sigma^2 + \sigma^\prime}{\sigma^2 + 1}\]
for $\sigma \in \mathbb{R}$. Note that if $\sigma^\prime \geq \frac{1}{5}$, then $g(\sigma) \geq \frac{1}{5}$ for all $\sigma \in \mathbb{R}$, so that 
\[F(\sigma^\prime) = \int_\mathbb{R} \frac{1}{\sigma^2 + \sigma^\prime} \mathrm{d}\sigma.\]
Hence, we may assume that $\sigma^\prime < \frac{1}{5}$. In particular, as $\sigma^\prime < 1$, the function $g$ is a bijection from $\mathbb{R}_+$ to $[\sigma^\prime, 1]$. A change of variables gives
\[F(\sigma^\prime) = \int_{\sigma^\prime}^1 \frac{\eta(\sigma)}{\sigma \sqrt{1 - \sigma} \sqrt{\sigma - \sigma^\prime} }\mathrm{d}\sigma.\]
As $\eta(\sigma) = 0$ for $\sigma \leq \frac{1}{10}$, one finds that if $\sigma^\prime \leq \frac{1}{10}$, then
\[F(\sigma^\prime) = \int_{\frac{1}{10}}^1 \frac{\eta(\sigma)}{\sigma \sqrt{1 - \sigma} \sqrt{\sigma - \sigma^\prime} }\mathrm{d}\sigma,\]
implying that $F$ is smooth on $\left(-\infty, \frac{1}{10}\right]$. Finally, note that
\[F(\sigma^\prime) = \int_0^1 \frac{\eta\left(\sigma^\prime + (1 - \sigma^\prime) \sigma\right)}{ (\sigma^\prime + (1 - \sigma^\prime)\sigma) \sqrt{1 - \sigma} \sqrt{\sigma}} \mathrm{d}\sigma,\]
by a last change of variable. As $\sigma^\prime + (1 - \sigma^\prime) \sigma > \frac{1}{10}$ for $\sigma \in (0, 1)$, this proves that $F$ is smooth on $[\frac{1}{10}, \frac{1}{5}]$.
\end{proof}

Lemma 18.1.10 of \cite{HormanderIII} gives $q_\intercal \in \mathsf{S}^{-1}(\mathbb{R}_t \times \mathbb{R}^{n - 1}_{x^\prime})$, completing the proof of Lemma \ref{lem_q_T_symbol}.
\end{proof}

With the same construction as for $\chi$, consider $\chi_3 \in \mathsf{S}^0(\mathbb{R}_t \times \mathbb{R}^{n - 1}_{x^\prime})$ such that $\chi_3(x^\prime, \tau, \xi^\prime) = 1$ if $(x^\prime, 0, \tau, \xi^\prime) \in \supp \chi$ and $1 + \tau^2 + \vert \xi^\prime \vert_x^2 \lesssim \rho(x^\prime, 0, \tau, \xi^\prime)$ on $\supp \chi_3$. The function $\tilde{\chi}_3(x^\prime, \tau, \xi^\prime) = \sqrt{\rho(x^\prime, 0, \tau, \xi^\prime)} \chi_3(x, \tau, \xi^\prime)$ is well-defined, and one has $\tilde{\chi}_3 \in \mathsf{S}^1(\mathbb{R}_t \times \mathbb{R}^{n - 1}_{x^\prime})$. By Lemma \ref{lem_q_T_symbol}, one obtains
\[\OpT\left( \tilde{\chi}_3 \right) \OpT(q_\intercal) \OpT(\chi_{\vert x_n = 0}) = \OpT\left(\sqrt{\rho} q_\intercal \chi_{\vert x_n = 0} \right) + R_4\]
where $R_4$ is a tangential pseudo-differential operator of order $-1$. As $\tilde{\chi}(x^\prime, \tau, \xi) = 1$ if $(x, \tau, \xi^\prime) \in \supp \chi$, one has
\[q_\intercal(x^\prime, 0, \tau, \xi^\prime) \chi(x^\prime, 0, \tau, \xi^\prime) = \chi(x^\prime, 0, \tau, \xi^\prime) \int_\mathbb{R} \frac{1}{\xi_n^2 + \rho(x^\prime, 0, \tau, \xi^\prime)} \mathrm{d}\xi_n = \frac{\pi \chi(x^\prime, 0, \tau, \xi^\prime)}{\sqrt{\rho(x^\prime, 0, \tau, \xi^\prime)}},\]
and this gives
\[\OpT\left( \tilde{\chi}_3 \right) \OpT(q_\intercal) \OpT(\chi_{\vert x_n = 0})\partial_\nu u = \pi \OpT(\chi_{\vert x^n = 0}) \partial_\nu u + R_4 \partial_\nu u.\]

This yields
\[\left\Vert \OpT(\chi_{\vert x^n = 0}) \partial_\nu u \right\Vert_{H^s(\mathbb{R} \times \mathbb{R}^{n-1})} \lesssim \left\Vert \OpT\left( \tilde{\chi}_3 \right) \OpT(q_\intercal) \OpT(\chi_{\vert x_n = 0})\partial_\nu u \right\Vert_{H^s(\mathbb{R} \times \mathbb{R}^{n-1})} + \left\Vert R_4 \partial_\nu u \right\Vert_{H^s(\mathbb{R} \times \mathbb{R}^{n-1})}.\]
As $\tilde{\chi}_3$ is of order $1$ and $R_4$ of order $-1$, this gives
\[\left\Vert \OpT(\chi_{\vert x^n = 0}) \partial_\nu u \right\Vert_{H^s(\mathbb{R} \times \mathbb{R}^{n-1})} \lesssim \left\Vert \OpT(q_\intercal) \OpT(\chi_{\vert x_n = 0}) \partial_\nu u \right\Vert_{H^{s + 1}(\mathbb{R} \times \mathbb{R}^{n-1})} + \left\Vert \partial_\nu u \right\Vert_{H^{s - 1}(\mathbb{R} \times \mathbb{R}^{n-1})}.\]
Using (\ref{proof_lem_main_term_OpT_chi_eq1}), one obtains
\begin{align*}
\left\Vert \OpT(\chi_{\vert x^n = 0}) \partial_\nu u \right\Vert_{H^s(\mathbb{R} \times \mathbb{R}^{n-1})} \lesssim & \left\Vert \Op(q) \left( \delta_{x^n = 0} \otimes \left( \OpT(\chi_{\vert x^n = 0}) \partial_\nu u \right) \right)_{\vert x^n = 0} \right\Vert_{H^{s + 1}(\mathbb{R} \times \mathbb{R}^{n-1})} \\
& + \left\Vert \partial_\nu u \right\Vert_{H^{s - 1}(\mathbb{R} \times \mathbb{R}^{n-1})}.
\end{align*}
This completes the proof of Lemma \ref{lem_left_side_proof_dt_ell}.
\end{proof}

Using Lemma \ref{lem_right_side_proof_dt_ell} and Lemma \ref{lem_left_side_proof_dt_ell} in (\ref{eq_dt_ell_trace_3}), one finds the estimate of Proposition \ref{prop_analysis_half_space}.

\section{Change of regularity in observability inequalities }

We prove our main result, Theorem \ref{thm_main_change_reg}, that is, the equivalence between observability at different levels of regularity. First, we show that for all $r \in \mathbb{N}$ and $s \in \mathbb{R}$, $H^s$-observability for $\Theta$ implies $H^{s + 2r}$-observability for $\Theta$. Second, we show that for all $r \in \mathbb{N}$ and $s > -1$, $H^s$-observability for $\Theta$ implies $H^{s - 2r}$-observability for $\tilde{\Theta}$, for all $\tilde{\Theta}$ such that $\pi_k \tilde{\Theta} \neq 0$ on $\supp \pi_k \Theta$, for all $k \in \llbracket 1, N \rrbracket$. This is enough to give the full conclusion, by Lemma \ref{lem_iterpolation}. 

\subsection{Increasing the level of regularity}

Consider $\Theta \in \mathscr{C}^\infty_\mathrm{c}((0, T) \times \partial M, \mathbb{C}^N)$, $s \in \mathbb{R}$ and assume that $H^s$-observability for $\Theta$ holds. We prove that $H^{s + 2}$-observability for $\Theta$ holds, implying by induction that $H^{s + 2r}$-observability for $\Theta$ holds for all $r \geq 1$. Consider also $\left( u^0, u^1 \right) \in \mathcal{K}^{s + 3} \times \mathcal{K}^{s + 2}$, and write $u$ for the solution with initial data $\left( u^0, u^1 \right)$. We show that 
\[\left\Vert \left( u^0, u^1 \right) \right\Vert_{\mathcal{K}^{s + 3} \times \mathcal{K}^{s + 2}} \lesssim \left\Vert \diag(\Theta) \partial_\nu u \right\Vert_{H^{s + 2}((0, T) \times \partial M, \mathbb{C}^{N})}.\]

For $t \in (0, T)$, set $\tilde{u}(t) = \PP{s + 2} u(t)$. Then, $\tilde{u}$ is the solution of
\[\left \{
\begin{array}{rcccl}
\partial_t^2 \tilde{u} - \mathsf{P} \tilde{u} & = & 0 & \quad & \text{in } (0, T) \times M, \\
\left( \tilde{u}(0, \cdot), \partial_t \tilde{u}(0, \cdot) \right) & = & \left( \PP{s + 2} u^0, \PP{s + 1} u^1 \right) & \quad & \text{in } M, \\
\tilde{u} & = & 0 & \quad & \text{on } (0, T) \times \partial M.
\end{array}
\right.\]
Since $\left( \PP{s + 2} u^0, \PP{s + 1} u^1 \right) \in \mathcal{K}^{s + 1} \times \mathcal{K}^{s}$, $H^s$-observability for $\Theta$ gives
\[\left\Vert \PP{s + 2} u^0 \right\Vert_{\mathcal{K}^{s + 1}} + \left\Vert \PP{s + 1} u^1 \right\Vert_{\mathcal{K}^{s}} \lesssim \left\Vert \diag(\Theta) \partial_\nu \tilde{u} \right\Vert_{H^s((0, T) \times \partial M, \mathbb{C}^N)}.\]
By Theorem \ref{thm_LLT_dir}, one has $\tilde{u}(t) = \PP{s + 2} u(t) = \partial_t^2 u(t)$ in $\mathcal{K}^{s + 1}$ for all $t \in [0, T]$, and $\partial_\nu \tilde{u} = \partial_\nu \partial_t^2 u = \partial_t^2 \partial_\nu u$. Hence, the previous estimate reads
\[\left\Vert \PP{s + 2} u^0 \right\Vert_{\mathcal{K}^{s + 1}} + \left\Vert \PP{s + 1} u^1 \right\Vert_{\mathcal{K}^{s}} \lesssim \left\Vert \diag(\Theta) \partial_t^2 \partial_\nu u \right\Vert_{H^s((0, T) \times \partial M, \mathbb{C}^N)}.\]
Using the ellipticity estimate for $\mathsf{P}$ (Proposition \ref{prop_main_properties_K^s}-\emph{(ii)}), one finds
\begin{align*}
\left\Vert u^0 \right\Vert_{\mathcal{K}^{s + 3}} + \left\Vert u^1 \right\Vert_{\mathcal{K}^{s + 2}} \lesssim \ & \left\Vert \diag(\Theta) \partial_t^2 \partial_\nu u \right\Vert_{H^s((0, T) \times \partial M, \mathbb{C}^N)} \\
+ & \left\Vert \iota_{\mathcal{K}^{s + 3} \rightarrow \mathcal{K}^{s + 2}} u^0 \right\Vert_{\mathcal{K}^{s + 2}} + \left\Vert \iota_{\mathcal{K}^{s + 2} \rightarrow \mathcal{K}^{s + 1}} u^1 \right\Vert_{\mathcal{K}^{s + 1}}.
\end{align*}

To estimate the term with the normal derivative, note that $\diag(\Theta) \partial_t^2 \partial_\nu u = \partial_t^2 \left( \diag(\Theta) \partial_\nu u \right) - 2 \partial_t \diag(\Theta) \partial_t \partial_\nu u - \partial_t^2 \diag(\Theta) \partial_\nu u$, implying
\[\left\Vert \diag(\Theta) \partial_t^2 \partial_\nu u \right\Vert_{H^s((0, T) \times \partial M, \mathbb{C}^N)} \lesssim \left\Vert \diag(\Theta) \partial_\nu u \right\Vert_{H^{s + 2}((0, T) \times \partial M, \mathbb{C}^N)} + \left\Vert \partial_\nu u \right\Vert_{H^{s + 1}((0, T) \times \partial M, \mathbb{C}^N)},\]
where the embedding $\iota_{H^{s + 2} \rightarrow H^{s + 1}}$ has been omitted. By Theorem \ref{thm_LLT_dir}, if $v$ is the solution of 
\[\left \{
\begin{array}{rcccl}
\partial_t^2 v - \mathsf{P} v & = & 0 & \quad & \text{in } \mathbb{R} \times M, \\
\left( v(0, \cdot), \partial_t v(0, \cdot) \right) & = & \left(\iota_{\mathcal{K}^{s + 3} \rightarrow \mathcal{K}^{s + 2}} u^0, \iota_{\mathcal{K}^{s + 2} \rightarrow \mathcal{K}^{s + 1}} u^1 \right) & \quad & \text{in } M, \\
v & = & 0 & \quad & \text{on } \mathbb{R} \times \partial M,
\end{array}
\right.\]
then $\iota_{\mathcal{K}^{s + 3} \rightarrow \mathcal{K}^{s + 2}} u = v$ and $\iota_{H^{s + 2} \rightarrow H^{s + 1}} \partial_\nu u = \partial_\nu v$. Hence, using Theorem \ref{thm_LLT_dir} again, one obtains
\begin{align*}
\left\Vert \partial_\nu u \right\Vert_{H^{s + 1}((0, T) \times \partial M, \mathbb{C}^N)} 
& = \left\Vert \partial_\nu v \right\Vert_{H^{s + 1}((0, T) \times \partial M, \mathbb{C}^N)} \\
& \lesssim \left\Vert v(0) \right\Vert_{\mathcal{K}^{s + 2}} + \left\Vert \partial_t v(0) \right\Vert_{\mathcal{K}^{s + 1}} \\
& = \left\Vert \iota_{\mathcal{K}^{s + 3} \rightarrow \mathcal{K}^{s + 2}} u^0 \right\Vert_{\mathcal{K}^{s + 2}} + \left\Vert \iota_{\mathcal{K}^{s + 2} \rightarrow \mathcal{K}^{s + 1}} u^1 \right\Vert_{\mathcal{K}^{s + 1}},
\end{align*}
yielding
\[\left\Vert u^0 \right\Vert_{\mathcal{K}^{s + 3}} + \left\Vert u^1 \right\Vert_{\mathcal{K}^{s + 2}}
\lesssim \left\Vert \diag(\Theta) \partial_\nu u \right\Vert_{H^{s + 2}((0, T) \times \partial M, \mathbb{C}^N)} + \left\Vert \iota_{\mathcal{K}^{s + 3} \rightarrow \mathcal{K}^{s + 2}} u^0 \right\Vert_{\mathcal{K}^{s + 2}} + \left\Vert \iota_{\mathcal{K}^{s + 2} \rightarrow \mathcal{K}^{s + 1}} u^1 \right\Vert_{\mathcal{K}^{s + 1}}.\]

To complete the proof, we prove that the remainder terms on the right-hand side can be removed. The embedding 
\[\begin{array}{cccc}
K: & \mathcal{K}^{s + 3} \times \mathcal{K}^{s + 2} & \longrightarrow & \mathcal{K}^{s + 2} \times \mathcal{K}^{s + 1} \\
& \left( u^0, u^1 \right) & \longmapsto & \left( \iota_{\mathcal{K}^{s + 3} \rightarrow \mathcal{K}^{s + 2}} u^0, \iota_{\mathcal{K}^{s + 2} \rightarrow \mathcal{K}^{s + 1}} u^1 \right)
\end{array}\]
is compact, by Proposition \ref{prop_main_properties_K^s}-\emph{(i)}. We show that the operator 
\[\begin{array}{cccc}
A: & \mathcal{K}^{s + 3} \times \mathcal{K}^{s + 2} & \longrightarrow & H^{s + 2}((0, T) \times \partial M, \mathbb{C}^N) \\
& \left( u^0, u^1 \right) & \longmapsto & \diag(\Theta) \partial_\nu u
\end{array}\]
is one-to-one, using the fact that the operator 
\[\begin{array}{ccc}
\mathcal{K}^{s + 1} \times \mathcal{K}^{s} & \longrightarrow & H^s((0, T) \times \partial M, \mathbb{C}^N) \\
\left( u^0, u^1 \right) & \longmapsto & \diag(\Theta) \partial_\nu u
\end{array}\]
is one-to-one, by $H^s$-observability. Assume that $\left( u^0, u^1 \right) \in \mathcal{K}^{s + 3} \times \mathcal{K}^{s + 2}$ is such that $\diag(\Theta) \partial_\nu u = 0$. As above, one has $\iota_{\mathcal{K}^{s + 3} \rightarrow \mathcal{K}^{s + 1}} u = v$ and $\iota_{H^{s + 2} \rightarrow H^s} \partial_\nu u = \partial_\nu v$, where $v$ is the solution of 
\[\left \{
\begin{array}{rcccl}
\partial_t^2 v - \mathsf{P} v & = & 0 & \quad & \text{in } \mathbb{R} \times M, \\
\left( v(0, \cdot), \partial_t v(0, \cdot) \right) & = & \left(\iota_{\mathcal{K}^{s + 3} \rightarrow \mathcal{K}^{s + 1}} u^0, \iota_{\mathcal{K}^{s + 2} \rightarrow \mathcal{K}^{s}} u^1 \right) & \quad & \text{in } M, \\
v & = & 0 & \quad & \text{on } \mathbb{R} \times \partial M.
\end{array}
\right.\]
Since $\diag(\Theta) \partial_\nu v = 0$ in $H^s((0, T) \times \partial M, \mathbb{C}^N)$, $H^s$-observability gives $\left(\iota_{\mathcal{K}^{s + 3} \rightarrow \mathcal{K}^{s + 1}} u^0, \iota_{\mathcal{K}^{s + 2} \rightarrow \mathcal{K}^{s}} u^1 \right) = 0$. Thus, one finds $\left( u^0, u^1 \right) = 0$, and $H^{s + 2}$-observability is a consequence of the following lemma.

\begin{lem}\label{lem_general_compact_term_1}
Let $\mathcal{X}$, $\mathcal{Y}$ and $\mathcal{Z}$ be Hilbert spaces. Consider two continuous linear operators $A: \mathcal{X} \rightarrow \mathcal{Y}$ and $K: \mathcal{X} \rightarrow \mathcal{Z}$. Assume that $K$ is compact and that there exists $C > 0$ such that
\[\Vert x \Vert_\mathcal{X} \leq \left( \left\Vert A x \right\Vert_\mathcal{Y} + \left\Vert K x \right\Vert_\mathcal{Z} \right), \quad x \in \mathcal{X}.\]
Then the kernel of $A$ is finite-dimensional. If moreover $A$ is one-to-one, there exists $C^\prime > 0$ such that for all $x \in \mathcal{X}$, one has
\[\Vert x \Vert_\mathcal{X} \leq C^\prime \left\Vert A x \right\Vert_\mathcal{Y}, \quad x \in \mathcal{X}.\]
\end{lem}

The proof is straightforward, and is omitted. Here, $A$ is one-to-one: the information from Lemma \ref{lem_general_compact_term_1} about the kernel is used below.

\subsection{Decreasing the level of regularity}

Consider $\Theta^1 \in \mathscr{C}^\infty_\mathrm{c}((0, T) \times \partial M, \mathbb{C}^N)$ and assume that $H^s$-observability for $\Theta^1$ holds. As the level of regularity can be increased, by the part of the proof of Section 3.1, we may assume that $s \geq 1$, without loss of generality. We prove that for $r \in \mathbb{N}^\ast$, $H^{s - 2r}$-observability for $\Theta^2$ holds, for all $\Theta^2 \in \mathscr{C}^\infty_\mathrm{c}((0, T) \times \partial M, \mathbb{C}^N)$ such that $\pi_k \Theta^2 \neq 0$ on $\supp \pi_k \Theta^1$, for all $k \in \llbracket 1, N \rrbracket$. Consider $\left( u^0, u^1 \right) \in \mathcal{K}^{s - 2r + 1} \times \mathcal{K}^{s - 2r}$, and denote by $u$ the associated solution.

Following the proof of Section 3.1, one might be inclined to define $\tilde{u}(t) = \mathsf{P}^{- r} u(t)$. However, this is not always possible, for example if $\mathsf{P} = \Delta + \lambda$, with $\lambda$ in the spectrum of the Dirichlet Laplacian. To overcome this difficulty, we use the shift operator of Proposition \ref{prop_main_properties_K^s}-\emph{(iii)}. We introduce $\left( \tilde{u}^0, \tilde{u}^1 \right)$ as the unique element of $\mathcal{K}^{s + 1} \times \mathcal{K}^{s}$ such that $\left( u^0, u^1 \right) = \left( \Shift{s - r + 1}^{r} \tilde{u}^0, \Shift{s - r}^{r} \tilde{u}^1 \right)$, and set $\tilde{u}$ as the solution associated with $\left( \tilde{u}^0, \tilde{u}^1 \right)$.

By $H^s$-observability for $\Theta^1$, one has
\[\left\Vert \tilde{u}^0 \right\Vert_{\mathcal{K}^{s + 1}} + \left\Vert \tilde{u}^1 \right\Vert_{\mathcal{K}^{s}} \lesssim \left\Vert \diag(\Theta^1) \partial_\nu \tilde{u} \right\Vert_{H^s((0, T) \times \partial M, \mathbb{C}^N)}.\]
Since $\Shift{s - r + 1}^{r}$ and $\Shift{s - r}^{r}$ are continuous, one has $\left\Vert u^0 \right\Vert_{\mathcal{K}^{s - 2 r + 1}} + \left\Vert u^1 \right\Vert_{\mathcal{K}^{s - 2 r}} \lesssim \left\Vert \tilde{u}^0 \right\Vert_{\mathcal{K}^{s + 1}} + \left\Vert \tilde{u}^1 \right\Vert_{\mathcal{K}^{s}}$, implying
\[\left\Vert u^0 \right\Vert_{\mathcal{K}^{s - 2 r + 1}} + \left\Vert u^1 \right\Vert_{\mathcal{K}^{s - 2 r}} \lesssim \left\Vert \diag(\Theta^1) \partial_\nu \tilde{u} \right\Vert_{H^s((0, T) \times \partial M, \mathbb{C}^N)}.\]
Now, Theorem \ref{thm_dt_ell_traces} gives
\[\left\Vert u^0 \right\Vert_{\mathcal{K}^{s - 2 r + 1}} + \left\Vert u^1 \right\Vert_{\mathcal{K}^{s - 2 r}} \lesssim \left\Vert \diag(\Theta^1) \partial_t^{2 r} \partial_\nu \tilde{u} \right\Vert_{H^{s - 2 r}((0, T) \times \partial M, \mathbb{C}^N)} + \left\Vert \tilde{u}^0 \right\Vert_{\mathcal{K}^{s + \frac{1}{2}}} + \left\Vert \tilde{u}^1 \right\Vert_{\mathcal{K}^{s - \frac{1}{2}}}.\]

Note that embeddings in the remainder terms are omitted here, as $s \geq 1$ (see Remark \ref{rem_dt_ell_embeddings}). Using Proposition \ref{prop_main_properties_K^s}-\emph{(iii)}, one finds
\begin{align*}
\left\Vert \tilde{u}^0 \right\Vert_{\mathcal{K}^{s + \frac{1}{2}}} + \left\Vert \tilde{u}^1 \right\Vert_{\mathcal{K}^{s - \frac{1}{2}}} = \ & \left\Vert \left( \Shift{s - r + \frac{1}{2}}^{r} \right)^{-1} \circ \iota_{\mathcal{K}^{s - 2r + 1} \rightarrow \mathcal{K}^{s - 2r + \frac{1}{2}}} u^0 \right\Vert_{\mathcal{K}^{s + \frac{1}{2}}} \\
+ & \left\Vert  \left( \Shift{s - r - \frac{1}{2}}^{r} \right)^{-1} \circ \iota_{\mathcal{K}^{s - 2r} \rightarrow \mathcal{K}^{s - 2r - \frac{1}{2}}} u^1 \right\Vert_{\mathcal{K}^{s - \frac{1}{2}}}.
\end{align*}
With the continuity of $\left( \Shift{s - r + \frac{1}{2}}^{r} \right)^{-1}$ and $\left( \Shift{s - r - \frac{1}{2}}^{r} \right)^{-1}$, one obtains
\[\left\Vert \tilde{u}^0 \right\Vert_{\mathcal{K}^{s + \frac{1}{2}}} + \left\Vert \tilde{u}^1 \right\Vert_{\mathcal{K}^{s - \frac{1}{2}}} \lesssim \left\Vert \iota_{\mathcal{K}^{s - 2r + 1} \rightarrow \mathcal{K}^{s - 2r + \frac{1}{2}}} u^0 \right\Vert_{\mathcal{K}^{s - 2r + \frac{1}{2}}} + \left\Vert \iota_{\mathcal{K}^{s - 2r} \rightarrow \mathcal{K}^{s - 2r - \frac{1}{2}}} u^1 \right\Vert_{\mathcal{K}^{s - 2r - \frac{1}{2}}},\]
implying
\begin{align*}
\left\Vert u^0 \right\Vert_{\mathcal{K}^{s - 2 r + 1}} + \left\Vert u^1 \right\Vert_{\mathcal{K}^{s - 2 r}} \lesssim \ & \left\Vert \diag(\Theta^1) \partial_t^{2 r} \partial_\nu \tilde{u} \right\Vert_{H^{s - 2 r}((0, T) \times \partial M, \mathbb{C}^N)} \\
+ & \left\Vert \iota_{\mathcal{K}^{s - 2r + 1} \rightarrow \mathcal{K}^{s - 2r + \frac{1}{2}}} u^0 \right\Vert_{\mathcal{K}^{s - 2r + \frac{1}{2}}} + \left\Vert \iota_{\mathcal{K}^{s - 2r} \rightarrow \mathcal{K}^{s - 2r - \frac{1}{2}}} u^1 \right\Vert_{\mathcal{K}^{s - 2r - \frac{1}{2}}}.
\end{align*}

Next, we want to replace $\left\Vert \diag(\Theta^1) \partial_t^{2 r} \partial_\nu \tilde{u} \right\Vert_{H^{s - 2 r}((0, T) \times \partial M, \mathbb{C}^N)}$ by $\left\Vert \diag(\Theta^1) \partial_\nu u \right\Vert_{H^{s - 2 r}((0, T) \times \partial M, \mathbb{C}^N)}$, up to a remainder term. The idea is the following: if $\Shift{s - r}^{r} = \PP{s - r}^r$, as in the case $\mathsf{P} = \Delta$ for example, then we can prove that $\partial_t^{2 r} \tilde{u} = u$. In the general case, one has the following lemma.

\begin{lem}\label{lem_switching_derivatives}
For $s \in \mathbb{R}$, $r \in \mathbb{N}^\ast$, $\Theta \in \mathscr{C}^\infty((0, T) \times \partial M, \mathbb{C}^N)$, $\left( v^0, v^1 \right) \in \mathcal{K}^{s + 1} \times \mathcal{K}^{s}$, and $v$ the solution associated with $\left( v^0, v^1 \right)$, one has 
\begin{align*}
\left\Vert \diag(\Theta) \partial_t^{2 r} \partial_\nu v \right\Vert_{H^{s - 2 r}((0, T) \times \partial M, \mathbb{C}^N)} \lesssim \ & \left\Vert \diag(\Theta) \partial_\nu \left( \Shift{s - r + 1}^{r} v \right) \right\Vert_{H^{s - 2 r}((0, T) \times \partial M, \mathbb{C}^N)} \\
+ & \left\Vert \iota_{\mathcal{K}^{s + 1} \rightarrow \mathcal{K}^{s}} v^0 \right\Vert_{\mathcal{K}^{s}} + \left\Vert \iota_{\mathcal{K}^{s} \rightarrow \mathcal{K}^{s - 1}} v^1 \right\Vert_{\mathcal{K}^{s - 1}}.
\end{align*}
\end{lem}

A proof of Lemma \ref{lem_switching_derivatives} is given below. Arguing as above, one finds
\begin{align*}
& \left\Vert \iota_{\mathcal{K}^{s + 1} \rightarrow \mathcal{K}^{s}} \tilde{u}^0 \right\Vert_{\mathcal{K}^{s}} + \left\Vert \iota_{\mathcal{K}^{s} \rightarrow \mathcal{K}^{s - 1}} \tilde{u}^1 \right\Vert_{\mathcal{K}^{s - 1}} \\
\lesssim & \left\Vert \iota_{\mathcal{K}^{s - 2r + 1} \rightarrow \mathcal{K}^{s - 2r}} u^0 \right\Vert_{\mathcal{K}^{s - 2r}} + \left\Vert \iota_{\mathcal{K}^{s - 2r} \rightarrow \mathcal{K}^{s - 2r - 1}} u^1 \right\Vert_{\mathcal{K}^{s - 2r - 1}} \\
\lesssim & \left\Vert \iota_{\mathcal{K}^{s - 2r + 1} \rightarrow \mathcal{K}^{s - 2r + \frac{1}{2}}} u^0 \right\Vert_{\mathcal{K}^{s - 2r + \frac{1}{2}}} + \left\Vert \iota_{\mathcal{K}^{s - 2r} \rightarrow \mathcal{K}^{s - 2r - \frac{1}{2}}} u^1 \right\Vert_{\mathcal{K}^{s - 2r - \frac{1}{2}}}.
\end{align*}
As $\Shift{s - r + 1}^{r} \tilde{u} = u$, by Corollary \ref{cor_LLT_dir_shift}, Lemma \ref{lem_switching_derivatives} gives
\begin{align}
\left\Vert u^0 \right\Vert_{\mathcal{K}^{s - 2 r + 1}} + & \left\Vert u^1 \right\Vert_{\mathcal{K}^{s - 2 r}} \lesssim \left\Vert \diag(\Theta^1) \partial_\nu u \right\Vert_{H^{s - 2 r}((0, T) \times \partial M, \mathbb{C}^N)} \nonumber \\
& \quad + \left\Vert \iota_{\mathcal{K}^{s - 2r + 1} \rightarrow \mathcal{K}^{s - 2r + \frac{1}{2}}} u^0 \right\Vert_{\mathcal{K}^{s - 2r + \frac{1}{2}}} + \left\Vert \iota_{\mathcal{K}^{s - 2r} \rightarrow \mathcal{K}^{s - 2r - \frac{1}{2}}} u^1 \right\Vert_{\mathcal{K}^{s - 2r - \frac{1}{2}}}. \label{eq_proof_decreasing_1}
\end{align}
Note that (\ref{eq_proof_decreasing_1}) holds true if $\Theta^1$ is replaced by some $\Theta^2 \in \mathscr{C}^\infty_\mathrm{c}((0, T) \times \partial M, \mathbb{C}^N)$ such that $\pi_k \Theta^2 \neq 0$ on $\supp \pi_k \Theta^1$, for all $k \in \llbracket 1, N \rrbracket$. To complete the proof, we show that the remainder terms on the right-hand side of (\ref{eq_proof_decreasing_1}) can be removed, when $\Theta^1$ is replaced by such $\Theta^2$. The embedding
\[\begin{array}{cccc}
K: & \mathcal{K}^{s - 2r + 1} \times \mathcal{K}^{s - 2r} & \longrightarrow & \mathcal{K}^{s - 2r + \frac{1}{2}} \times \mathcal{K}^{s - 2r - \frac{1}{2}} \\
& \left( u^0, u^1 \right) & \longmapsto & \left( \iota_{\mathcal{K}^{s - 2r + 1} \rightarrow \mathcal{K}^{s - 2r + \frac{1}{2}}} u^0, \iota_{\mathcal{K}^{s - 2 r} \rightarrow \mathcal{K}^{s - 2r - \frac{1}{2}}} u^1 \right)
\end{array}\]
is compact, by Proposition \ref{prop_main_properties_K^s}-\emph{(i)}. For $s^\prime \in \mathbb{R}$ and $\Theta \in \mathscr{C}^\infty_\mathrm{c}((0, T) \times \partial M, \mathbb{C}^N)$, introduce 
\[\begin{array}{cccc}
A_{\Theta, s^\prime}: & \mathcal{K}^{s^\prime + 1} \times \mathcal{K}^{s^\prime} & \longrightarrow & H^{s^\prime}((0, T) \times \partial M, \mathbb{C}^N) \\
& \left( u^0, u^1 \right) & \longmapsto & \diag(\Theta) \partial_\nu u
\end{array}.\]
By (\ref{eq_proof_decreasing_1}) and Lemma \ref{lem_general_compact_term_1}, the kernel of $A_{\Theta^1, s - 2r}$ is finite-dimensional. Note also that $H^{s}$-observability for $\Theta^1$ implies that $A_{\Theta^1, s}$ is one-to-one.

\begin{lem}\label{lem_inj_obs_theta1_theta2}
Consider $s \in \mathbb{R}$ and $\Theta^1 \in \mathscr{C}^\infty_\mathrm{c}((0, T) \times \partial M, \mathbb{C}^N)$ such that $A_{\Theta^1, s}$ is one-to-one, and such that for all $r \in \mathbb{N}^\ast$, the kernel of $A_{\Theta^1, s - 2r}$ is finite-dimensional. Then, for $r \in \mathbb{N}^\ast$ and $\Theta^2 \in \mathscr{C}^\infty_\mathrm{c}((0, T) \times \partial M, \mathbb{C}^N)$ such that $\pi_k \Theta^2 \neq 0$ on $\supp \pi_k \Theta^1$, for all $k \in \llbracket 1, N \rrbracket$, $A_{\Theta^2, s - 2r}$ is one-to-one.
\end{lem}

A proof of Lemma \ref{lem_inj_obs_theta1_theta2} is given below. Now, using (\ref{eq_proof_decreasing_1}) with $\Theta^2$ instead of $\Theta^1$, and Lemma \ref{lem_general_compact_term_1} again, one concludes that $H^{s - 2r}$-observability for $\Theta^2$ holds. As explained in the beginning of Section 3, this completes the proof of Theorem \ref{thm_main_change_reg}.

Now, we prove Lemma \ref{lem_switching_derivatives} and Lemma \ref{lem_inj_obs_theta1_theta2}.

\begin{proof}[Proof of Lemma \ref{lem_switching_derivatives}.]
By interpolation, one may assume that $s \in \mathbb{Z}$. By Theorem \ref{thm_LLT_dir}, one has $\partial_t^{2 r} \partial_\nu v = \partial_\nu \partial_t^{2 r} v = \partial_\nu \PP{s - r + 1}^{r} v$. The triangular inequality gives
\begin{align*}
& \left\Vert \diag(\Theta) \partial_t^{2 r} \partial_\nu v \right\Vert_{H^{s - 2 r}((0, T) \times \partial M, \mathbb{C}^N)} \\ 
\lesssim & \left\Vert \diag(\Theta) \partial_\nu \left( \Shift{s - r + 1}^{r} v \right) \right\Vert_{H^{s - 2 r}((0, T) \times \partial M, \mathbb{C}^N)} + \left\Vert \partial_\nu \left(\PP{s - r + 1}^{r} v - \Shift{s - r + 1}^{r} v \right) \right\Vert_{H^{s - 2 r}((0, T) \times \partial M, \mathbb{C}^N)}.
\end{align*}
Set $w = \left( \PP{s - r + 1}^{r} - \Shift{s - r + 1}^{r} \right) v$. By Theorem \ref{thm_LLT_dir} and Corollary \ref{cor_LLT_dir_shift}, $w$ is the solution of 
\[\left \{
\begin{array}{rcccl}
\partial_t^2 w - \mathsf{P} w & = & 0 & \quad & \text{in } (0, T) \times M, \\
\left( w(0, \cdot), \partial_t w(0, \cdot) \right) & = & \left( w^0, w^1 \right) & \quad & \text{in } M, \\
w & = & 0 & \quad & \text{on } (0, T) \times \partial M,
\end{array}
\right.\]
where $\left( w^0, w^1 \right) = \left( \left(\PP{s - r + 1}^{r} - \Shift{s - r + 1}^{r} \right) v^0, \left(\PP{s - r}^{r} - \Shift{s - r}^{r} \right) v^1 \right)$. Hence, using Theorem \ref{thm_LLT_dir} and Proposition \ref{prop_main_properties_K^s}-\emph{(iii)}, one obtains 
\begin{align*}
\left\Vert \partial_\nu w \right\Vert_{H^{s - 2 r}((0, T) \times \partial M, \mathbb{C}^N)} 
& \lesssim \left\Vert w^0 \right\Vert_{\mathcal{K}^{s - 2 r + 1}} + \left\Vert w^1 \right\Vert_{\mathcal{K}^{s - 2 r}} \\
& \lesssim \left\Vert \iota_{\mathcal{K}^{s + 1} \rightarrow \mathcal{K}^{s}} v^0 \right\Vert_{\mathcal{K}^{s}} + \left\Vert \iota_{\mathcal{K}^{s} \rightarrow \mathcal{K}^{s - 1}} v^1 \right\Vert_{\mathcal{K}^{s - 1}},
\end{align*}
and this gives the desired result.
\end{proof}

\begin{proof}[Proof of Lemma \ref{lem_inj_obs_theta1_theta2}.]
For $s^\prime \in \mathbb{R}$ and $\Theta \in \mathscr{C}^\infty_\mathrm{c}((0, T) \times \partial M, \mathbb{C}^N)$, denote by $N_{\Theta, s^\prime}$ the kernel of $A_{\Theta, s^\prime}$, that is,
\[N_{\Theta, s^\prime} = \left\{ \left( u^0, u^1 \right) \in \mathcal{K}^{s^\prime + 1} \times \mathcal{K}^{s^\prime}, \diag(\Theta) \partial_\nu u = 0 \right\}.\]
Note that by definition, one has
\begin{equation}\label{eq_proof_decreasing_reg_N_Theta_2}
N_{\Theta, s^\prime} \subset N_{\tilde{\Theta}, s^\prime}, \quad s^\prime \in \mathbb{R},
\end{equation}
if $\pi_k \Theta \neq 0$ on $\supp \pi_k \tilde{\Theta}$, for all $k \in \llbracket 1, N \rrbracket$, and for $s_1 > s_2$, the map 
\[\begin{array}{cccc}
\Phi_{s_1, s_2}(\Theta): & N_{\Theta, s_1} & \longrightarrow & N_{\Theta, s_2} \\
& \left( u^0, u^1 \right) & \longmapsto & \left( \iota_{\mathcal{K}^{s_1 + 1} \rightarrow \mathcal{K}^{s_2 + 1}} u^0, \iota_{\mathcal{K}^{s_1} \rightarrow \mathcal{K}^{s_2}} u^1 \right)
\end{array}\]
is well-defined, injective, and compact. 

Consider $\Theta^2 \in \mathscr{C}^\infty_\mathrm{c}((0, T) \times \partial M, \mathbb{C}^N)$ such that $\pi_k \Theta^2 \neq 0$ on $\supp \pi_k \Theta^1$, for all $k \in \llbracket 1, N \rrbracket$. We claim that $N_{\Theta^2, s - 2r}$ is finite-dimensional for all $r \geq 0$. Indeed, for $r \in \mathbb{N}$, it follows from the assumptions of Lemma \ref{lem_inj_obs_theta1_theta2} and (\ref{eq_proof_decreasing_reg_N_Theta_2}). For $r \geq 0$, as $\Phi_{s - 2r, s - \lfloor 2r \rfloor - 1}(\Theta^2)$ is one-to-one, $N_{\Theta^2, s - 2r}$ is isomorphic to a subspace of $N_{\Theta^2, s - \lfloor 2r \rfloor - 1}$, and hence, is finite-dimensional.

Consider $r \geq 0$. We prove that 
\begin{equation}\label{eq_proof_decreasing_reg_N_Theta_3_bis}
\Phi_{s - 2r, s - 2r - 1}(\Theta^2) \text{ is an isomorphism.}
\end{equation}
It suffices to show that $\Phi_{s - 2r, s - 2r - 1}(\Theta^2)$ is onto. Consider $\left( u^0, u^1 \right) \in N_{\Theta^2, s - 2r - 1}$, and write $U(t) = (u(t), \partial_t u(t))$ for $t \in \mathbb{R}$, where $u$ is the solution associated with $\left( u^0, u^1 \right)$. As the distance between $\supp \pi_k \Theta^1$ and $(\supp \pi_k \Theta^2)^\complement$ is positive for all $k \in \llbracket 1, N \rrbracket$, there exists $\epsilon > 0$ such that for $t \in [0, \epsilon)$, $U(t) \in N_{\Theta^1, s - 2r - 1}$.

For $t \in (0, \epsilon)$, set $V_t = \frac{1}{t}(U(t) - U(0)) \in N_{\Theta^1, s - 2r - 1}$. One has
\begin{equation}\label{eq_proof_decreasing_reg_N_Theta_4}
\iota_{\mathcal{K}^{s - 2 r + 1} \rightarrow \mathcal{K}^{s - 2 r}} \left( \frac{1}{t}(u(t) - u^0) \right) \underset{t \rightarrow 0^+}{\longrightarrow} \partial_t u(0) = u^1 \in \mathcal{K}^{s - 2 r},
\end{equation}
and
\begin{equation}\label{eq_proof_decreasing_reg_N_Theta_5}
\iota_{\mathcal{K}^{s - 2 r} \rightarrow \mathcal{K}^{s - 2 r - 1}} \left( \frac{1}{t}(\partial_t u(t) - u^1) \right) \underset{t \rightarrow 0^+}{\longrightarrow} \partial_t^2 u(0) = \PP{s - 2 r} u^0 \in \mathcal{K}^{s - 2 r - 1}.
\end{equation}
As $N_{\Theta^1, s - 2r - 1}$ is finite-dimensional, the norm of $\mathcal{K}^{s - 2 r + 1} \times \mathcal{K}^{s - 2 r}$ is equivalent to the norm
\[\mathcal{N}\left( u^0, u^1 \right) = \left\Vert \left( \iota_{\mathcal{K}^{s - 2 r + 1} \rightarrow \mathcal{K}^{s - 2 r}} u^0, \iota_{\mathcal{K}^{s - 2 r} \rightarrow \mathcal{K}^{s - 2 r - 1}} u^1 \right) \right\Vert_{\mathcal{K}^{s - 2r} \times \mathcal{K}^{s - 2 r - 1}}.\]
By (\ref{eq_proof_decreasing_reg_N_Theta_4}) and (\ref{eq_proof_decreasing_reg_N_Theta_5}), $(V_t)_{t > 0}$ is a Cauchy sequence for the norm $\mathcal{N}$, and thus, it converges in $N_{\Theta^1, s - 2r}$. Write $\left( v^0, v^1 \right)$ for its limit. Using (\ref{eq_proof_decreasing_reg_N_Theta_4}) and (\ref{eq_proof_decreasing_reg_N_Theta_5}) again, one finds
\[\left( \iota_{\mathcal{K}^{s - 2 r + 1} \rightarrow \mathcal{K}^{s - 2 r}} v^0, \iota_{\mathcal{K}^{s - 2 r} \rightarrow \mathcal{K}^{s - 2 r - 1}} v^1 \right) = \left( u^1, \PP{s - 2 r} u^0 \right).\]
By Proposition \ref{prop_main_properties_K^s}-\emph{(ii)}, there exists $\tilde{u}^0 \in \mathcal{K}^{s - 2 r + 2}$ such that $u^0 = \iota_{\mathcal{K}^{s - 2 r + 2} \rightarrow \mathcal{K}^{s - 2 r + 1}} \tilde{u}^0$ and $v^1 = \PP{s - 2 r + 1} \tilde{u}^0$. One has
\[\left( \iota_{\mathcal{K}^{s - 2 r + 2} \rightarrow \mathcal{K}^{s - 2 r + 1}} \tilde{u}^0, \iota_{\mathcal{K}^{s - 2 r + 1} \rightarrow \mathcal{K}^{s - 2 r}} v^0 \right) = \left( u^0, u^1 \right).\]
This gives $\Phi_{s - 2r, s - 2r - 1}(\Theta^2) \left( \tilde{u}^0, v^0 \right) = \left( u^0, u^1 \right)$, if we show that $\left( \tilde{u}^0, v^0 \right) \in N_{\Theta^2, s - 2r - 1}$. If $\tilde{u}$ is the solution associated with $\left( \tilde{u}^0, v^0 \right)$, then by Theorem \ref{thm_LLT_dir}, one has $\iota_{\mathcal{K}^{s - 2 r + 2} \rightarrow \mathcal{K}^{s - 2 r + 1}} \tilde{u} = u$ and 
\[\iota_{H^{s - 2 r + 1} \rightarrow H^{s - 2 r}} \partial_\nu \tilde{u} = \partial_\nu u.\]
As $\diag(\Theta^2) \partial_\nu u = 0 \in \mathscr{D}^\prime((0, T) \times \partial M, \mathbb{C}^N)$, this implies $\left( \tilde{u}^0, v^0 \right) \in N_{\Theta^2, s - 2r - 1}$, completing the proof of (\ref{eq_proof_decreasing_reg_N_Theta_3_bis}).

By iteration, one obtains an isomorphism between $N_{\Theta^2, s}$ and $N_{\Theta^2, s - 2r}$ for $r \in \mathbb{N}^\ast$. As $N_{\Theta^1, s} = \{ 0 \}$, (\ref{eq_proof_decreasing_reg_N_Theta_2}) gives $N_{\Theta^2, s} = \{0\}$. This completes the proof of Lemma \ref{lem_inj_obs_theta1_theta2}.
\end{proof}

\appendix

\section{The case of internal observability}

Here, we explain how to adapt the methods of this article to the case of internal observability. Consider $\chi \in \mathscr{C}^\infty(M, \mathbb{C}^N)$.

\begin{defn}[$\mathcal{K}^s$-observability for $\left( \chi, T \right)$]
We say that $\mathcal{K}^s$-observability for $\left( \chi, T \right)$ holds if there exists $C > 0$ such that for all $\left( u^0, u^1 \right) \in \mathcal{K}^{s} \times \mathcal{K}^{s - 1}$, 
\[\left\Vert \left( u^0, u^1 \right) \right\Vert_{\mathcal{K}^{s} \times \mathcal{K}^{s - 1}} \leq C \left\Vert \diag(\chi) u \right\Vert_{L^2((0, T), \mathcal{K}^s)},\]
where $u$ is the solution of (\ref{eq_intro_dirichlet_hom}) with initial data $\left( u^0, u^1 \right)$.
\end{defn}

Note that the multiplication operator $u \in \mathcal{K}^s \mapsto \diag(\chi) u \in \mathcal{K}^s$ is well-defined, and commutes with the embeddings of Proposition \ref{prop_main_properties_K^s}. As in the boundary case, one can prove that $\mathcal{K}^s$-observability for $\left( \chi, T \right)$ is equivalent with a controllability property, for the equation (\ref{eq_intro_dirichlet_hom}) with a source term of the form $\diag(\chi) F$, with $F \in L^2((0, T), \mathcal{K}_\ast^{- s})$. One can check that the solution of such a system is well-defined, by adapting the proof of Theorem \ref{thm_LLT_dir}. The analogue of Theorem \ref{thm_main_change_reg} is the following result.

\begin{thm}\label{thm_main_internal_change_reg}
Consider $s_1, s_2 \in \mathbb{R}$, and $\chi \in \mathscr{C}^\infty(M, \mathbb{C}^N)$. If $s_1 < s_2$, then for all $T > 0$, $\mathcal{K}^{s_1}$-observability for $\left( \chi, T \right)$ implies $\mathcal{K}^{s_2}$-observability for $\left( \chi, T \right)$. If $s_1 > s_2$, then for all $0 < T_1 < T_2$, $\mathcal{K}^{s_1}$-observability for $\left( \chi, T_1 \right)$ implies $\mathcal{K}^{s_2}$-observability for $\left( \chi, T_2 \right)$.
\end{thm}

The proof of Theorem \ref{thm_main_internal_change_reg} is simpler than that of Theorem \ref{thm_main_change_reg}, so we only sketch it. To increase the regularity level, one uses the following lemma.

\begin{lem}\label{lem_technical_increasing_reg_internal}
Consider $s \in \mathbb{R}$ and $\chi \in \mathscr{C}^\infty(M, \mathbb{C}^N)$. There exists $C > 0$ such that
\[\left\Vert \left[ \diag(\chi), \PP{s + 1} \right] u \right\Vert_{\mathcal{K}^{s}} \leq C \left\Vert \iota_{\mathcal{K}^{s + 2} \rightarrow \mathcal{K}^{s + 1}} u \right\Vert_{\mathcal{K}^{s + 1}}, \quad u \in \mathcal{K}^{s + 2}.\]
\end{lem}

\begin{proof}[Sketch of proof of Lemma \ref{lem_technical_increasing_reg_internal}.]
By interpolation, it suffices to prove Lemma \ref{lem_technical_increasing_reg_internal} for $s \in \mathbb{Z}$. For $s \in \mathbb{N}$ and $u \in \mathcal{K}^{s + 2}$, one has
\[\left\Vert \left[ \diag(\chi), \PP{s + 1} \right] u \right\Vert_{\mathcal{K}^{s}} = \left\Vert \left[ \diag(\chi), \PP{\mathscr{D}^\prime} \right] u \right\Vert_{H^s(M, \mathbb{C}^N)} \lesssim \left\Vert u \right\Vert_{H^{s + 1}(M, \mathbb{C}^N)} = \left\Vert \iota_{\mathcal{K}^{s + 2} \rightarrow \mathcal{K}^{s + 1}} u \right\Vert_{\mathcal{K}^{s + 1}},\]
and the same holds for $\mathsf{P}^\ast$. Now, consider $s \in \mathbb{Z}$, $s \leq -1$, and $u \in \mathcal{K}^{s + 2}$. Note that 
\begin{align*}
& \left\Vert \left[ \diag(\chi), \PP{s + 1} \right] u \right\Vert_{\mathcal{K}^{s}} \\
= & \sup \left\{ \left\vert \left\langle \left[ \diag(\chi), \PP{s + 1} \right] u, \iota_{\mathcal{K}_\ast^{- s + 1} \rightarrow \mathcal{K}_\ast^{-s}} v \right\rangle_{\mathcal{K}^{s}, \mathcal{K}_\ast^{- s}} \right\vert, v \in \mathcal{K}_\ast^{- s + 1}, \left\Vert \iota_{\mathcal{K}_\ast^{- s + 1} \rightarrow \mathcal{K}_\ast^{- s}} v \right\Vert_{\mathcal{K}_\ast^{-s}} \leq 1 \right\},
\end{align*}
as $\iota_{\mathcal{K}_\ast^{- s + 1} \rightarrow \mathcal{K}_\ast^{- s}}$ has a dense range. For $v \in \mathcal{K}_\ast^{- s + 1}$, using the case $s \in \mathbb{N}$, one finds
\begin{align*}
\left\vert \left\langle \left[ \diag(\chi), \PP{s + 1} \right] u, \iota_{\mathcal{K}_\ast^{- s + 1} \rightarrow \mathcal{K}_\ast^{-s}} v \right\rangle_{\mathcal{K}^{s}, \mathcal{K}_\ast^{- s}} \right\vert 
& = \left\vert \left\langle \iota_{\mathcal{K}^{s + 2} \rightarrow \mathcal{K}^{s + 1}} u, \left[ \diag(\chi), \PP{-s}^\ast \right] v \right\rangle_{\mathcal{K}^{s + 1}, \mathcal{K}_\ast^{- s - 1}} \right\vert \\
& \lesssim \left\Vert \iota_{\mathcal{K}^{s + 2} \rightarrow \mathcal{K}^{s + 1}} u \right\Vert_{\mathcal{K}^{s + 1}} \left\Vert \iota_{\mathcal{K}_\ast^{- s + 1} \rightarrow \mathcal{K}_\ast^{- s}} v \right\Vert_{\mathcal{K}_\ast^{-s}},
\end{align*}
and that completes the proof of Lemma \ref{lem_technical_increasing_reg_internal}.
\end{proof}

Lemma \ref{lem_technical_increasing_reg_internal} and $\mathcal{K}^{s}$-observability for $\left( \chi, T \right)$ yield 
\[\left\Vert \left( u^0, u^1 \right) \right\Vert_{\mathcal{K}^{s + 2} \times \mathcal{K}^{s + 1}} \lesssim \left\Vert \diag(\chi) u \right\Vert_{L^2((0, T), \mathcal{K}^{s + 2})} + \left\Vert \left( \iota_{\mathcal{K}^{s + 2} \rightarrow \mathcal{K}^{s + 1}} u^0, \iota_{\mathcal{K}^{s + 1} \rightarrow \mathcal{K}^{s}} u^1 \right) \right\Vert_{\mathcal{K}^{s + 1} \times \mathcal{K}^{s}},\]
for $\left( u^0, u^1 \right) \in \mathcal{K}^{s + 2} \times \mathcal{K}^{s + 1}$. The remainder term is compact, and $\mathcal{K}^{s}$-observability for $\left( \chi, T \right)$ implies that the operator $\left( u^0, u^1 \right) \in \mathcal{K}^{s + 2} \times \mathcal{K}^{s + 1} \mapsto \diag(\chi) u$ is one-to-one. This proves that $\mathcal{K}^{s}$-observability implies $\mathcal{K}^{s + 2}$-observability.

To decrease the regularity level, one relies on the following result about the shift operator of Proposition \ref{prop_main_properties_K^s}. We use the notation $\mathcal{S}_{s}^{- 1} = \left( \mathcal{S}_{s}^1 \right)^{- 1}$, for $s \in \mathbb{R}$.

\begin{lem}\label{lem_technical_decreasing_reg_internal}
Consider $s \in \mathbb{R}$ and $\chi \in \mathscr{C}^\infty(M, \mathbb{C}^N)$. There exists $C > 0$ such that
\[\left\Vert \left[ \diag(\chi), \mathcal{S}_{s - 1}^{- 1} \right] u \right\Vert_{\mathcal{K}^{s}} \leq C \left\Vert \iota_{\mathcal{K}^{s - 2} \rightarrow \mathcal{K}^{s - 3}} u \right\Vert_{\mathcal{K}^{s - 3}}, \quad u \in \mathcal{K}^{s - 2}.\]
\end{lem}

\begin{proof}[Sketch of proof of Lemma \ref{lem_technical_decreasing_reg_internal}.]
Using Proposition \ref{prop_main_properties_K^s} and Lemma \ref{lem_technical_increasing_reg_internal}, one finds
\begin{align*}
\left\Vert \left[ \diag(\chi), \mathcal{S}_{s - 1}^{- 1} \right] u \right\Vert_{\mathcal{K}^{s}} 
& = \left\Vert \mathcal{S}_{s - 1}^{- 1} \left( \mathcal{S}_{s - 1}^1 \left( \diag(\chi) \mathcal{S}_{s - 1}^{- 1} u \right) - \diag(\chi) \mathcal{S}_{s - 1}^1 \mathcal{S}_{s - 1}^{- 1} u \right) \right\Vert_{\mathcal{K}^{s}} \\
& \lesssim \left\Vert \mathcal{S}_{s - 1}^1 \left( \diag(\chi) \mathcal{S}_{s - 1}^{- 1} u \right) - \diag(\chi) \mathcal{S}_{s - 1}^1 \mathcal{S}_{s - 1}^{- 1} u \right\Vert_{\mathcal{K}^{s - 2}} \\
& = \left\Vert \PP{s - 1} \left( \diag(\chi) \mathcal{S}_{s - 1}^{- 1} u \right) - \diag(\chi) \PP{s - 1} \mathcal{S}_{s - 1}^{- 1} u \right\Vert_{\mathcal{K}^{s - 2}} \\
& \lesssim \left\Vert \iota_{\mathcal{K}^{s} \rightarrow \mathcal{K}^{s - 1}} \mathcal{S}_{s - 1}^{- 1} u \right\Vert_{\mathcal{K}^{s - 1}} \\
& \lesssim \left\Vert \iota_{\mathcal{K}^{s - 2} \rightarrow \mathcal{K}^{s - 3}} u \right\Vert_{\mathcal{K}^{s - 3}},
\end{align*}
for all $u \in \mathcal{K}^{s - 2}$.
\end{proof}

Assume that $\mathcal{K}^{s}$-observability for $\left( \chi, T_1 \right)$ holds. Then, Lemma \ref{lem_technical_decreasing_reg_internal} gives 
\[\left\Vert \left( u^0, u^1 \right) \right\Vert_{\mathcal{K}^{s - 2} \times \mathcal{K}^{s - 3}} \lesssim \left\Vert \diag(\chi) u \right\Vert_{L^2((0, T_1), \mathcal{K}^{s - 2})} + \left\Vert \left( \iota_{\mathcal{K}^{s - 2} \rightarrow \mathcal{K}^{s - 3}} u^0, \iota_{\mathcal{K}^{s - 3} \rightarrow \mathcal{K}^{s - 4}} u^1 \right) \right\Vert_{\mathcal{K}^{s - 3} \times \mathcal{K}^{s - 4}},\]
for $\left( u^0, u^1 \right) \in \mathcal{K}^{s - 2} \times \mathcal{K}^{s - 3}$. For $T > 0$ and $s^\prime \in \mathbb{R}$, set
\[\begin{array}{cccc}
A_{T, s^\prime}: & \mathcal{K}^{s^\prime} \times \mathcal{K}^{s^\prime - 1} & \longrightarrow & L^2((0, T), \mathcal{K}^{s^\prime}) \\
& \left( u^0, u^1 \right) & \longmapsto & \diag(\chi) u
\end{array},\]
and write $\Ker A_{T, s^\prime}$ for the kernel of that operator. Then $A_{T_1, s}$ is one-to-one, $\Ker A_{T_1, s - 2}$ is finite-dimensional, and to complete the proof of Theorem \ref{thm_main_internal_change_reg}, it suffices to prove that $A_{T_2, s - 2}$ is one-to-one, for all $T_2 > T_1$. For $s_1 > s_2$ and $T > 0$, introduce the embedding
\[\begin{array}{cccc}
\Phi_{T, s_1, s_2}: & \Ker A_{T, s_1} & \longrightarrow & \Ker A_{T, s_2} \\
& \left( u^0, u^1 \right) & \longmapsto & \left( \iota_{\mathcal{K}^{s_1} \rightarrow \mathcal{K}^{s_2}} u^0, \iota_{\mathcal{K}^{s_1 - 1} \rightarrow \mathcal{K}^{s_2 - 1}} u^1 \right)
\end{array}.\]
Consider $T_2 > T_1$ and $\sigma \in \{ 0, 1 \}$. We prove that $\Phi_{T_2, s - \sigma, s - \sigma - 1}$ is an isomorphism. Take $\left( u^0, u^1 \right) \in \Ker A_{T_2, s - \sigma - 1}$. For $t \in (0, T_2]$, set
\[V_t = \frac{1}{t} \left( \left( u(t), \partial_t u(t) \right) - \left( u^0, u^1 \right) \right).\]
As $T_2 > T_1$, one has $V_t \in \Ker A_{T_1, s - \sigma - 1}$ for all $t > 0$ sufficiently small. In addition, $V_t$ converges to a limit as $t \rightarrow 0^+$, for one particular norm on $\Ker A_{T_1, s - \sigma - 1}$, and hence for any norm on $\Ker A_{T_1, s - \sigma - 1}$, as $\Ker A_{T_1, s - \sigma - 1}$ is finite-dimensional. This gives $\left( u^0, u^1 \right) \in \Phi_{T_2, s - \sigma, s - \sigma - 1} \left( \Ker A_{T_2, s - \sigma} \right)$. Hence, $\Phi_{T_2, s - \sigma, s - \sigma - 1}$ is an isomorphism. In particular, $\Ker A_{T_2, s - 2}$ is isomorphic to $\Ker A_{T_2, s} = \{ 0 \}$, and this completes the proof of Theorem \ref{thm_main_internal_change_reg}.

\section{Proof of the results of Section 1}

\subsection{Proof of Proposition \ref{prop_main_properties_K^s}}

We start by giving some details about \emph{(i)}. If $s \geq 0$ then the map $\iota_{\mathcal{K}^{s + \delta} \rightarrow \mathcal{K}^{s}}: \mathcal{K}^{s + \delta} \hooklongrightarrow \mathcal{K}^{s}$ is just a natural inclusion, and is thus one-to-one. It is an embedding, and it will often be omitted. If $s + \delta < 0$, then by definition $\iota_{\mathcal{K}^{s + \delta} \rightarrow \mathcal{K}^{s}}$ is the restriction operator
\[\begin{array}{ccc}
\mathcal{K}^{s + \delta} & \longrightarrow & \mathcal{K}^{s} \\
u & \longmapsto & u_{\vert \mathcal{K}_\ast^{- s}}
\end{array}.\]
In Step 5 below, we prove that $\iota_{\mathcal{K}_\ast^{- s} \rightarrow \mathcal{K}_\ast^{- s - \delta}}$ has dense range, implying that $\iota_{\mathcal{K}^{s + \delta} \rightarrow \mathcal{K}^{s}}$ is one-to-one if $s + \delta < 0$. By definition, if $s + \delta \geq 0 > s$, one has 
\[\left\langle \iota_{\mathcal{K}^{s + \delta} \rightarrow \mathcal{K}^{s}}(u), v \right\rangle_{\mathcal{K}^{s}, \mathcal{K}_\ast^{- s}} = \left\langle u, \overline{v} \right\rangle_{L^2(M, \mathbb{C}^N)},\]
for $u \in \mathcal{K}^{s + \delta}$ and $v \in \mathcal{K}_\ast^{- s} = \mathcal{K}_\ast^{\vert s \vert}$. As $\mathscr{D}(M, \mathbb{C}^N) \subset \mathcal{K}_\ast^{- s}$, one sees that $\mathcal{K}_\ast^{- s}$ is dense in $L^2(M, \mathbb{C}^N)$. This implies that $\iota_{\mathcal{K}^{s + \delta} \rightarrow \mathcal{K}^{s}}$ is one-to-one in the case $s + \delta \geq 0 > s$.

To prove \emph{(iii)}, one can always assume that $r = 1$: the operator $\Shift{s}^r$ is then defined by 
\[\Shift{s}^r: \mathcal{K}^{s + r} \xrightarrow{\Shift{s + r - 1}^1} \mathcal{K}^{s + r - 2} \xrightarrow{\Shift{s + r - 3}^1} \cdots \xrightarrow{\Shift{s - r + 3}^1} \mathcal{K}^{s - r + 2} \xrightarrow{\Shift{s - r + 1}^1} \mathcal{K}^{s - r}.\]
Note that in particular, (\ref{eq_prop_shifting_op_1}) holds true by definition.

The proof is organized as follows. First, we prove \emph{(ii)} for $s \in \mathbb{N}$, $s \geq r$, except the inequality. Second, we construct the shift operator in three steps. Third, using the shift operator, we prove \emph{(i)}. Finally, we complete the proof of \emph{(ii)} and \emph{(iii)}. We often assume that $s \in \mathbb{Z}$: the case $s \in \mathbb{R}$ follows by interpolation.

\paragraph{Step 1: Proof of \emph{(ii)} for $s \in \mathbb{N}$, $s \geq r$ (except the inequality).} As $s - r \geq 0$ here, we know that $\PP{s - 1}^{r} u = \PP{\mathscr{D}^\prime}^{r} u$ for all $u \in \mathcal{K}^{s + r - 1}$. Hence, all equations of this step can be understood in $\mathscr{D}^\prime(M, \mathbb{C}^N)$ (or in $H^{-1}(M, \mathbb{C}^N)$), and we omit embeddings.

We prove by induction on $r \in \mathbb{N}$ that for all $s \in \mathbb{N}$, $s \geq r$, and all $u \in \mathcal{K}^{s + r - 1}$, if $\PP{s - 1}^r u \in \mathcal{K}^{s - r}$ then $u \in \mathcal{K}^{s + r}$. It is true for $r = 0$, as $\PP{s - 1}^0 = \mathrm{Id}_{\mathcal{K}^{s - 1}}$. Take $r \in \mathbb{N}$ such that the result holds. Fix $s \in \mathbb{N}$, $s \geq r + 1$, and $u \in \mathcal{K}^{s + r}$ such that $\PP{s - 1}^{r + 1} u \in \mathcal{K}^{s - r - 1}$. We want to show that $u \in \mathcal{K}^{s + r + 1}$. By induction, we only need to prove that $v = \PP{s}^r u \in \mathcal{K}^{s + 1 - r}$.

We start by proving that $v \in H^{s + 1 - r}(M, \mathbb{C}^N)$. By definition of $\PP{s}^r$, we know that $v \in \mathcal{K}^{s - r}$. One has $\PP{s - r - 1} v = \PP{s - 1}^{r + 1} u$, so that $\PP{s - r - 1} v \in \mathcal{K}^{s - r - 1}$ by assumption. In particular, one has $v \in H^{s - r}(M, \mathbb{C}^N)$ and $\PP{s - r - 1} v \in H^{s - r - 1}(M, \mathbb{C}^N)$: this gives
\[\Delta v = \PP{s - r - 1} v + \left( X + q \right) v \in H^{s - r - 1}(M, \mathbb{C}^N),\]
with equality in $\mathscr{D}^\prime(M, \mathbb{C}^N)$. As $s - r \geq 1$, one also has $v \in H^1_0(M, \mathbb{C}^N)$. Thus, by a standard elliptic regularity result, applied componentwise, one finds $v \in H^{s + 1 - r}(M, \mathbb{C}^N)$.

Now, we prove that $v = \PP{s}^r u \in \mathcal{K}^{s + 1 - r}$. Assume that $s - r$ is odd and write $s - r = 2 \sigma + 1$. By definition, the fact that $u \in \mathcal{K}^{s + r}$ gives $\PP{\mathscr{D}^\prime}^k u \in H_0^1(M, \mathbb{C}^N)$ for $k \in \llbracket 0, \sigma + r \rrbracket$. As $v = \PP{\mathscr{D}^\prime}^r u$, this implies
\[\PP{\mathscr{D}^\prime}^k v \in H_0^1(M, \mathbb{C}^N) \text{ for } k \in \llbracket 0, \sigma \rrbracket,\]
yielding $v \in \mathcal{K}^{s + 1 - r}$. Now, assume that $s - r$ is even and write $s - r = 2 \sigma$. By definition, $u \in \mathcal{K}^{s + r}$ gives $\PP{\mathscr{D}^\prime}^k u \in H_0^1(M, \mathbb{C}^N)$ for $k \in \llbracket 0, \sigma + r - 1\rrbracket$, implying
\[\PP{\mathscr{D}^\prime}^k v \in H_0^1(M, \mathbb{C}^N) \text{ for } k \in \llbracket 0, \sigma - 1 \rrbracket.\]
As $\PP{s - r - 1} v = \PP{s - 1}^{r + 1} u \in \mathcal{K}^{s - r - 1}$, one also has $\PP{\mathscr{D}^\prime}^{\sigma} v \in H_0^1(M, \mathbb{C}^N)$, so that $v \in \mathcal{K}^{s + 1 - r}$.

\paragraph{Step 2: Injectivity of the shift operator for $s \in \mathbb{N}$.}
Consider $\mu \in \mathbb{R}$. We show that for $\vert \mu \vert$ sufficiently large and for $s \in \mathbb{N}$, the operator
\[\begin{array}{cccc}
\PP{s} + i\mu: & \mathcal{K}^{s + 1} & \longrightarrow & \mathcal{K}^{s - 1} \\
& u & \longmapsto & \left( \PP{s} + i\mu \right) u
\end{array}\]
is one-to-one, where the embedding $\mathcal{K}^{s + 1} \hooklongrightarrow \mathcal{K}^{s - 1}$ has been omitted (as explained in the beginning of the proof, this embedding is indeed an embedding even if $s = 0$). Fix $s \in \mathbb{N}$ and $u = (u^1, \cdots, u^N) \in \mathcal{K}^{s + 1}$ such that $\left( \PP{s} + i\mu \right) u = 0$. Write $(\pi_1, \cdots, \pi_N)$ for the projections associated with the canonical basis of $\mathbb{C}^N$. If $s = 0$, one has 
\[\PP{0} u = - i\mu u \in \mathcal{K}^1 \subset L^2(M, \mathbb{C}^N),\]
so that Step 1 gives $u \in \mathcal{K}^2$. Hence, $u \in \mathcal{K}^2$ for all $s \in \mathbb{N}$. 

For $k \in \llbracket 1, N \rrbracket$, one has
\[- \Delta u^k + \pi_k \left( X u + qu \right) + i\mu u^k = 0.\]
Multiplication by $\overline{u^k}$, integration on $M$ and an integration by parts give
\begin{equation}\label{eq_proof_prop_main_properties_domains_1}
\int_M \left( \left\vert \nabla u^k \right\vert ^2 + i\mu \vert u^k\vert ^2 \right) \mathrm{d}x = - \int_M \overline{u^k} \pi_k \left( X u + qu \right) \mathrm{d}x.
\end{equation}

Computing the real part and using the Cauchy-Schwarz inequality yields
\[\int_M \left\vert \nabla u^k \right\vert ^2 \mathrm{d}x \leq \left\Vert Xu + qu \right\Vert_{L^2(M, \mathbb{C}^N)} \left\Vert u \right\Vert_{L^2(M, \mathbb{C}^N)} \lesssim \left\Vert u \right\Vert_{H^1(M, \mathbb{C}^N)} \left\Vert u \right\Vert_{L^2(M, \mathbb{C}^N)}.\]
By the Poincaré inequality, one obtains
\[\left\Vert u \right\Vert_{H^1(M, \mathbb{C}^N)}^2 \lesssim \left\Vert u \right\Vert_{H^1(M, \mathbb{C}^N)} \left\Vert u \right\Vert_{L^2(M, \mathbb{C}^N)}.\]
Thus, there exists $C > 0$ depending only on $X$ and $q$ such that
\begin{equation}\label{eq_proof_prop_main_properties_domains_2}
\left\Vert u \right\Vert_{H^1(M, \mathbb{C}^N)} \leq C \left\Vert u \right\Vert_{L^2(M, \mathbb{C}^N)}.
\end{equation}

Now, compute the imaginary part of (\ref{eq_proof_prop_main_properties_domains_1}) and use the Cauchy-Schwarz inequality to find 
\[\vert \mu\vert \left\Vert u \right\Vert_{L^2(M, \mathbb{C}^N)}^2 \leq \left\Vert Xu + qu \right\Vert_{L^2(M, \mathbb{C}^N)} \left\Vert u \right\Vert_{L^2(M, \mathbb{C}^N)} \lesssim \left\Vert u \right\Vert_{H^1(M, \mathbb{C}^N)} \left\Vert u \right\Vert_{L^2(M, \mathbb{C}^N)}.\]
Together with (\ref{eq_proof_prop_main_properties_domains_2}), this gives
\[\vert \mu\vert \left\Vert u \right\Vert_{L^2(M, \mathbb{C}^N)} \leq C \left\Vert u \right\Vert_{L^2(M, \mathbb{C}^N)},\]
with $C > 0$ depending only on $X$ and $q$. Thus, for $\vert \mu \vert$ chosen sufficiently large, one has $u = 0$. 

\paragraph{Step 3: Surjectivity of the shift operator for $s = 0$.}
We prove that
\[\begin{array}{cccc}
\PP{0} + i\mu: & H_0^1(M, \mathbb{C}^N) & \longrightarrow & H^{-1}(M, \mathbb{C}^N) \\
& u & \longmapsto & \left( \PP{0} + i\mu \right) u
\end{array}\]
is onto. By definition of the operator $\PP{0}: H_0^1(M, \mathbb{C}^N) \rightarrow H^{-1}(M, \mathbb{C}^N)$, we need to show that for $v \in H^{-1}(M, \mathbb{C}^N)$, there exists $u \in H_0^1(M, \mathbb{C}^N)$ such that
\[- \left\langle \nabla u, \nabla \phi \right\rangle_{L^2(M, \mathbb{C}^N)} + \left\langle i\mu u - (X + q) u, \phi \right\rangle_{L^2(M, \mathbb{C}^N)} = \left\langle v, \overline{\phi} \right\rangle_{H^{-1}, H_0^1}, \quad \phi \in H_0^1(M, \mathbb{C}^N),\]
where one has used the notation
\[\left\langle \nabla u, \nabla \phi \right\rangle_{L^2(M, \mathbb{C}^N)} = \sum_{k = 1}^N \left\langle \nabla u^k, \nabla \phi^k \right\rangle_{L^2(M)}.\]

Using the Lax-Milgram theorem, it suffices to prove a coercivity inequality of the form
\begin{equation}\label{eq_proof_prop_main_properties_domains_3}
\left\vert - \left\Vert \nabla u \right\Vert_{L^2(M, \mathbb{C}^N)}^2 + \left\langle i\mu u - (X + q) u, u \right\rangle_{L^2(M, \mathbb{C}^N)} \right\vert \gtrsim \Vert u \Vert_{H^1(M, \mathbb{C}^N)}^2, \quad u \in H_0^1(M, \mathbb{C}^N).
\end{equation}
Take $u \in H_0^1(M, \mathbb{C}^N)$. Using the triangular inequality, one has
\begin{equation}\label{eq_proof_prop_main_properties_domains_4}
\begin{split}
& \left\vert - \left\Vert \nabla u \right\Vert_{L^2(M, \mathbb{C}^N)}^2 + \left\langle i\mu u - (X + q) u, u \right\rangle_{L^2(M, \mathbb{C}^N)} \right\vert \\
\geq & \left\vert - \left\Vert \nabla u \right\Vert_{L^2(M, \mathbb{C}^N)}^2 + i\mu \Vert u \Vert_{L^2(M, \mathbb{C}^N)}^2 \right\vert - \left\vert \left\langle (X + q) u, u \right\rangle_{L^2(M, \mathbb{C}^N)} \right\vert . \end{split}
\end{equation}

Using the Cauchy-Schwarz and Poincaré inequalities, one obtains
\[\left\vert \left\langle (X + q) u, u \right\rangle_{L^2(M, \mathbb{C}^N)} \right\vert \lesssim \left\Vert u \right\Vert_{H^1(M, \mathbb{C}^N)} \left\Vert u \right\Vert_{L^2(M, \mathbb{C}^N)} \lesssim \left\Vert \nabla u \right\Vert_{L^2(M, \mathbb{C}^N)} \left\Vert u \right\Vert_{L^2(M, \mathbb{C}^N)}.\]
For $\epsilon > 0$, one writes
\[\left\vert \left\langle (X + q) u, u \right\rangle_{L^2(M, \mathbb{C}^N)} \right\vert \leq c_1 \left( \epsilon \left\Vert \nabla u \right\Vert_{L^2(M, \mathbb{C}^N)}^2 + \frac{1}{\epsilon} \left\Vert u \right\Vert_{L^2(M, \mathbb{C}^N)}^2 \right).\]
For the other term of (\ref{eq_proof_prop_main_properties_domains_4}), simply write
\[\left\vert - \left\Vert \nabla u \right\Vert_{L^2(M, \mathbb{C}^N)}^2 + i\mu \Vert u \Vert_{L^2(M, \mathbb{C}^N)}^2 \right\vert \geq c_2 \left( \left\Vert \nabla u \right\Vert_{L^2(M, \mathbb{C}^N)}^2 + \vert \mu\vert \Vert u \Vert_{L^2(M, \mathbb{C}^N)}^2 \right).\]
Thus, coming back to (\ref{eq_proof_prop_main_properties_domains_4}), one obtains
\begin{align*}
& \left\vert - \left\Vert \nabla u \right\Vert_{L^2(M, \mathbb{C}^N)}^2 + \left\langle i\mu u - (X + q) u, u \right\rangle_{L^2(M, \mathbb{C}^N)} \right\vert \\
\geq \ & \left(c_2 - c_1 \epsilon\right) \left\Vert \nabla u \right\Vert_{L^2(M, \mathbb{C}^N)}^2 + \left(c_2 \vert \mu\vert - \frac{c_1}{\epsilon} \right) \Vert u \Vert_{L^2(M, \mathbb{C}^N)}^2.
\end{align*}
We choose $\epsilon$ so that $c_2 - c_1 \epsilon > 0$. Then, for $\vert \mu \vert$ sufficiently large, the coercivity inequality (\ref{eq_proof_prop_main_properties_domains_3}) holds.

\paragraph{Step 4: Construction of the shift operator.}
We start by the case $s \in \mathbb{N}$. As above, we omit the embeddings. We proceed by induction on $s \in \mathbb{N}$ and prove that $\PP{s} + i\mu: \mathcal{K}^{s + 1} \longrightarrow \mathcal{K}^{s - 1}$ is onto. By Step 3, it holds for $s = 0$. Assume that it holds for some fixed $s \in \mathbb{N}$. Take $v \in \mathcal{K}^{s}$. We show that there exists $u \in \mathcal{K}^{s + 2}$ such that $v = \left(\PP{s + 1} + i\mu \right) u$.

As $v \in \mathcal{K}^{s} \subset \mathcal{K}^{s - 1}$, there exists $u \in \mathcal{K}^{s + 1}$ such that $v = \left(\PP{s} + i\mu \right) u$. We apply Step 1. One has $u \in \mathcal{K}^{s + 1}$, $v \in \mathcal{K}^{s}$, and $\PP{s} u = v - i\mu u \in \mathcal{K}^{s}$. As $s \geq 0 $, this gives $u \in \mathcal{K}^{s + 2}$. 

Together with Step 2, this shows that $\PP{s} + i\mu: \mathcal{K}^{s + 1} \longrightarrow \mathcal{K}^{s - 1}$ is an isomorphism for all $s \in \mathbb{N}$. Note that as $\mathsf{P}$ and $\mathsf{P}^\ast$ are of the same form, the operator $\mathsf{P}^\ast - i\mu: \mathcal{K}_\ast^{s + 2} \longrightarrow \mathcal{K}_\ast^{s}$ is also an isomorphism for all $s \in \mathbb{N}$. Now, for $s \leq - 1$, we define $\Shift{s}^1: \mathcal{K}^{s + 1} \longrightarrow \mathcal{K}^{s - 1}$ as the adjoint of $\mathsf{P}^\ast - i\mu: \mathcal{K}_\ast^{- s + 1} \longrightarrow \mathcal{K}_\ast^{- s - 1}$: it is an isomorphism. Note that for $s \in \mathbb{Z}$, $\Shift{s}^1: \mathcal{K}^{s + 1} \longrightarrow \mathcal{K}^{s - 1}$ is a continuous isomorphism between Hilbert spaces, so its inverse is continuous. As $\Shift{s}^1$ is now defined for $s \in \mathbb{Z}$, we can define $\Shift{s}^1: \mathcal{K}^{s + 1} \longrightarrow \mathcal{K}^{s - 1}$ for $s \in \mathbb{R}$ by interpolation.

\paragraph{Step 5: Proof of \emph{(i)}.}
Note that at this stage, the injectivity of the operator $\iota_{\mathcal{K}^{s + \delta} \rightarrow \mathcal{K}^{s}}$ is yet to be proven if $s + \delta < 0$. To that purpose, the commutativity property (\ref{eq_prop_shifting_op_3}) is needed.

We start by proving (\ref{eq_lem_P_comm_with_embeddings}). It suffices to prove it for $r = 1$, that is
\begin{equation}\label{eq_proof_prop_main_properties_domains_7}
\PP{s} \circ \iota_{\mathcal{K}^{s + 1 + \delta} \rightarrow \mathcal{K}^{s + 1}} = \iota_{\mathcal{K}^{s - 1 + \delta} \rightarrow \mathcal{K}^{s - 1}} \circ \PP{s + \delta}: \mathcal{K}^{s + 1 + \delta} \rightarrow \mathcal{K}^{s - 1}.
\end{equation}
By interpolation, it suffices to prove it for $s \in \mathbb{Z}$ and $\delta \in \mathbb{N}^\ast$. If $s \geq 0$, then (\ref{eq_proof_prop_main_properties_domains_7}) is true as it holds in $\mathscr{D}^\prime(M, \mathbb{C}^N)$. Similarly, one has
\begin{equation}\label{eq_proof_prop_main_properties_domains_8}
\PP{s}^\ast \circ \iota_{\mathcal{K}_\ast^{s + 1 + \delta} \rightarrow \mathcal{K}_\ast^{s + 1}} = \iota_{\mathcal{K}_\ast^{s - 1 + \delta} \rightarrow \mathcal{K}_\ast^{s - 1}} \circ \PP{s + \delta}^\ast: \mathcal{K}_\ast^{s + 1 + \delta} \rightarrow \mathcal{K}_\ast^{s - 1}, \quad s \geq 0.
\end{equation}
Computing the adjoint of (\ref{eq_proof_prop_main_properties_domains_8}) and using (\ref{eq_adjoint_adjoint_P}) and (\ref{eq_adjoint_iota}), one finds
\[\iota_{\mathcal{K}^{- \tilde{s} - 1} \rightarrow \mathcal{K}^{- \tilde{s} - 1 - \delta}} \circ \PP{- \tilde{s}} = \PP{- \tilde{s} - \delta} \circ \iota_{\mathcal{K}^{- \tilde{s} + 1} \rightarrow \mathcal{K}^{- \tilde{s} + 1 - \delta}}, \quad \tilde{s} \geq 0,\]
and this gives (\ref{eq_proof_prop_main_properties_domains_7}) for $s = - \tilde{s} - \delta$, that is, in the case $s \leq -1$ and $s + \delta \leq 0$. Thus, (\ref{eq_proof_prop_main_properties_domains_7}) only remains to be proven for $s \leq - 1$ and $s + \delta \geq 1$. In this case, for $u \in \mathcal{K}^{s + 1 + \delta} \subset \mathcal{K}^2$ and $v \in \mathcal{K}_\ast^{1 - s} \subset \mathcal{K}_\ast^2$, one writes
\[\left\langle \PP{s} \circ \iota_{\mathcal{K}^{s + 1 + \delta} \rightarrow \mathcal{K}^{s + 1}} u, v \right\rangle_{\mathcal{K}^{s - 1}, \mathcal{K}_\ast^{1 - s}} = \left\langle \iota_{\mathcal{K}^{s + 1 + \delta} \rightarrow \mathcal{K}^{s + 1}} u, \PP{- s}^\ast v \right\rangle_{\mathcal{K}^{s + 1}, \mathcal{K}_\ast^{- s - 1}} = \left\langle u, \overline{\PP{- s}^\ast v} \right\rangle_{L^2(M, \mathbb{C}^N)},\]
as $s + 1 + \delta \geq 0 \geq s + 1$. Using $- s \geq 0$ and $s + \delta \geq 0$, one has
\[\left\langle \PP{s} \circ \iota_{\mathcal{K}^{s + 1 + \delta} \rightarrow \mathcal{K}^{s + 1}} u, v \right\rangle_{\mathcal{K}^{s - 1}, \mathcal{K}_\ast^{1 - s}} = \left\langle u, \overline{\PP{\mathscr{D}^\prime}^\ast v} \right\rangle_{L^2(M, \mathbb{C}^N)} = \left\langle \PP{s + \delta} u, \overline{v} \right\rangle_{L^2(M, \mathbb{C}^N)}.\]
As $s - 1 + \delta \geq 0 \geq s - 1$, one finds
\[\left\langle \PP{s} \circ \iota_{\mathcal{K}^{s + 1 + \delta} \rightarrow \mathcal{K}^{s + 1}} u, v \right\rangle_{\mathcal{K}^{s - 1}, \mathcal{K}_\ast^{1 - s}} = \left\langle \iota_{\mathcal{K}^{s - 1 + \delta} \rightarrow \mathcal{K}^{s - 1}} \circ \PP{s + \delta} u, v \right\rangle_{\mathcal{K}^{s - 1}, \mathcal{K}_\ast^{1 - s}},\]
and this gives (\ref{eq_proof_prop_main_properties_domains_7}).

Now, we prove that it implies the commutativity property (\ref{eq_prop_shifting_op_3}). The case $r = 1$, $s \in \mathbb{Z}$ and $\delta \in \mathbb{N}^\ast$ suffices, that is,
\begin{equation}\label{eq_proof_prop_main_properties_domains_9}
\Shift{s}^1 \circ \iota_{\mathcal{K}^{s + 1 + \delta} \rightarrow \mathcal{K}^{s + 1}} = \iota_{\mathcal{K}^{s - 1 + \delta} \rightarrow \mathcal{K}^{s - 1}} \circ \Shift{s + \delta}^1.
\end{equation}
For $s \geq 0$, one has $\Shift{s}^1 = \PP{s} + i \mu \iota_{\mathcal{K}^{s + 1} \rightarrow \mathcal{K}^{s - 1}}$ by definition, yielding (\ref{eq_proof_prop_main_properties_domains_9}) as a direct consequence of (\ref{eq_proof_prop_main_properties_domains_7}). For $s \leq - 1$, one has
\[\Shift{s}^1 = \left(\PP{- s}^\ast - i \mu \iota_{\mathcal{K}^{- s + 1} \rightarrow \mathcal{K}^{- s - 1}}\right)^\ast,\]
implying $\Shift{s}^1 = \PP{s} + i \mu \iota_{\mathcal{K}^{s + 1} \rightarrow \mathcal{K}^{s - 1}}$ by (\ref{eq_adjoint_adjoint_P}) and (\ref{eq_adjoint_iota}). Hence, (\ref{eq_proof_prop_main_properties_domains_9}) is also a consequence of (\ref{eq_proof_prop_main_properties_domains_7}) in that case. 

Now, we complete the proof of \emph{(i)}. As $\mathscr{D}(M, \mathbb{C}^N) \subset \mathcal{K}^{1 + \delta}$, $\mathcal{K}^{1 + \delta}$ is dense in $\mathcal{K}^1$. Take $s \in \mathbb{N}$. If $s = 2 \sigma$, then using (\ref{eq_prop_shifting_op_3}), one can factor the map $\iota_{\mathcal{K}^{s + \delta} \rightarrow \mathcal{K}^{s}}$ into
\begin{equation}\label{eq_proof_prop_main_properties_domains_10}
\iota_{\mathcal{K}^{s + \delta} \rightarrow \mathcal{K}^{s}}: \mathcal{K}^{s + \delta} \xrightarrow{\Shift{\delta + \sigma}^\sigma} \mathcal{K}^{\delta} \xhookrightarrow{\iota_{\mathcal{K}^{\delta} \rightarrow \mathcal{K}^{0}}} \mathcal{K}^0 \xrightarrow{\left( \Shift{\sigma}^\sigma \right)^{-1}} \mathcal{K}^{s}
\end{equation}
and this proves the density of $\mathcal{K}^{s + \delta}$ in $\mathcal{K}^{s}$. Similarly, if $s = 2 \sigma + 1$ then one writes
\begin{equation}\label{eq_proof_prop_main_properties_domains_11}
\iota_{\mathcal{K}^{s + \delta} \rightarrow \mathcal{K}^{s}}: \mathcal{K}^{s + \delta} \xrightarrow{\Shift{1 + \delta + \sigma}^\sigma} \mathcal{K}^{1 + \delta} \xhookrightarrow{\iota_{\mathcal{K}^{1 + \delta} \rightarrow \mathcal{K}^{1}}} \mathcal{K}^1 \xrightarrow{\left( \Shift{1 + \sigma}^\sigma \right)^{-1}} \mathcal{K}^{s}.
\end{equation}

This proves that all the maps of \emph{(i)} are embeddings. The rest of the proof of \emph{(i)} is similar, using (\ref{eq_proof_prop_main_properties_domains_10}), (\ref{eq_proof_prop_main_properties_domains_11}). For the compactness property, note that for $0 < \delta < \frac{1}{2}$, one has $\mathcal{K}^{\delta} = H^\delta(M, \mathbb{C}^N)$, implying the compactness of  $\iota_{\mathcal{K}^{\delta} \rightarrow \mathcal{K}^{0}}$ by the Rellich theorem.

\paragraph{Step 6: End of the proof of \emph{(ii)}.}
Fix $s \in \mathbb{Z}$ and $r \in \mathbb{N}^\ast$ such that $s \leq r - 1$, the case $s \geq r$ having been carried out in Step 1. Take $w \in \mathcal{K}^{s + r - 1}$ and $v \in \mathcal{K}^{s - r}$ such that
\begin{equation}\label{eq_proof_prop_main_properties_domains_12}
\PP{s - 1}^{r} w = \iota_{\mathcal{K}^{s - r} \rightarrow \mathcal{K}^{s - r - 1}} v.
\end{equation}
We seek $u \in \mathcal{K}^{s + r}$ such that 
\begin{equation}\label{eq_proof_prop_main_properties_domains_13}
\iota_{\mathcal{K}^{s + r} \rightarrow \mathcal{K}^{s + r - 1}} u = w
\end{equation}
and
\begin{equation}\label{eq_proof_prop_main_properties_domains_14}
\PP{s}^r u = v.
\end{equation}
Note that (\ref{eq_proof_prop_main_properties_domains_13}) implies
\begin{align*}
\iota_{\mathcal{K}^{s - r} \rightarrow \mathcal{K}^{s - r - 1}} \left( \PP{s}^r u - v \right) 
& = \PP{s - 1}^r \circ \iota_{\mathcal{K}^{s + r} \rightarrow \mathcal{K}^{s + r - 1}} u - \iota_{\mathcal{K}^{s - r} \rightarrow \mathcal{K}^{s - r - 1}} v \\
& = \PP{s - 1}^r w - \iota_{\mathcal{K}^{s - r} \rightarrow \mathcal{K}^{s - r - 1}} v = 0
\end{align*}
and this gives (\ref{eq_proof_prop_main_properties_domains_14}) since $\iota_{\mathcal{K}^{s - r} \rightarrow \mathcal{K}^{s - r - 1}}$ is one-to-one. Hence, it suffices to find $u \in \mathcal{K}^{s + 2r}$ such that (\ref{eq_proof_prop_main_properties_domains_13}) holds. Note that such a $u$ is unique since $\iota_{\mathcal{K}^{s + 2r} \rightarrow \mathcal{K}^{s + 2r - 1}}$ is one-to-one. If $s + r \geq 1$, the embedding could be omitted. If however $s + r \leq 0$, (\ref{eq_proof_prop_main_properties_domains_13}) means that $u$ is an extension of $w$ as a continuous linear form.

Take $\sigma \in \mathbb{N}$ such that $s - r - 1 + 2 \sigma \geq 0$. Applying $\left( \Shift{s - r - 1 + \sigma}^\sigma \right)^{- 1}$ to (\ref{eq_proof_prop_main_properties_domains_12}), one finds
\[\PP{s + 2 \sigma - 1}^{r} \circ \left( \Shift{s + r + \sigma - 1}^\sigma \right)^{- 1} w = \iota_{\mathcal{K}^{s - r + 2 \sigma} \rightarrow \mathcal{K}^{s - r + 2 \sigma - 1}} \circ \left( \Shift{s - r + \sigma}^\sigma \right)^{- 1} v.\]
Apply Step 1 with $W = \left( \Shift{s + r + \sigma - 1}^\sigma \right)^{- 1} w \in \mathcal{K}^{s + r + 2 \sigma - 1}$ and $\tilde{s} = s + 2 \sigma \geq r$: there exists $U \in \mathcal{K}^{s + r + 2 \sigma}$ such that
\[W = \iota_{\mathcal{K}^{s + r + 2 \sigma} \rightarrow \mathcal{K}^{s + r + 2 \sigma - 1}} U.\]
Hence, one finds
\[w = \Shift{s + r + \sigma - 1}^\sigma \circ \iota_{\mathcal{K}^{s + r + 2 \sigma} \rightarrow \mathcal{K}^{s + r + 2 \sigma - 1}} U = \iota_{\mathcal{K}^{s + r} \rightarrow \mathcal{K}^{s + r - 1}} \circ \Shift{s + r + \sigma}^\sigma U,\]
and this is (\ref{eq_proof_prop_main_properties_domains_13}) with $u = \Shift{s + r + \sigma}^\sigma U$.

We now prove (\ref{eq_prop_elliptic_estimate_P}) by induction on $r \in \mathbb{N}^\ast$. We start with the case $r = 1$. Take $s \in \mathbb{R}$ and $u \in \mathcal{K}^{s + 1}$, and write
\[\Vert u \Vert_{\mathcal{K}^{s + 1}} = \left\Vert \left( \Shift{s}^1 \right)^{-1} \circ \Shift{s}^1 u \right\Vert_{\mathcal{K}^{s + 1}} \lesssim \left\Vert \Shift{s}^1 u \right\Vert_{\mathcal{K}^{s - 1}}.\]
Using the definition of $\Shift{s}^1$ and the triangular inequality, one obtains
\[\Vert u \Vert_{\mathcal{K}^{s + 1}} \lesssim \left\Vert \PP{s} u \right\Vert_{\mathcal{K}^{s - 1}} + \left\Vert \iota_{\mathcal{K}^{s + 1} \rightarrow \mathcal{K}^{s - 1}} u \right\Vert_{\mathcal{K}^{s - 1}},\]
a better estimate than (\ref{eq_prop_elliptic_estimate_P}) in the case $r = 1$, since one has
\[\left\Vert \iota_{\mathcal{K}^{s + 1} \rightarrow \mathcal{K}^{s - 1}} u \right\Vert_{\mathcal{K}^{s - 1}} = \left\Vert \iota_{\mathcal{K}^{s} \rightarrow \mathcal{K}^{s - 1}} \circ \iota_{\mathcal{K}^{s + 1} \rightarrow \mathcal{K}^{s}} u \right\Vert_{\mathcal{K}^{s - 1}} \lesssim \left\Vert \iota_{\mathcal{K}^{s + 1} \rightarrow \mathcal{K}^{s}} u \right\Vert_{\mathcal{K}^{s}}.\]

Now, we take $r \in \mathbb{N}^\ast$ such that (\ref{eq_prop_elliptic_estimate_P}) holds, and we prove that (\ref{eq_prop_elliptic_estimate_P}) also holds for $r + 1$. Take $s \in \mathbb{R}$ and $u \in \mathcal{K}^{s + r + 1}$. We want to show that
\[\Vert u \Vert_{\mathcal{K}^{s + r + 1}} \lesssim \left\Vert \PP{s}^{r + 1} u \right\Vert_{\mathcal{K}^{s - r - 1}} + \left\Vert \iota_{\mathcal{K}^{s + r + 1} \rightarrow \mathcal{K}^{s + r}} u \right\Vert_{\mathcal{K}^{s + r}}.\] 
By induction, one has
\[\Vert u \Vert_{\mathcal{K}^{s + r + 1}} \lesssim \left\Vert \PP{s + 1}^{r} u \right\Vert_{\mathcal{K}^{s + 1 - r}} + \left\Vert \iota_{\mathcal{K}^{s + r + 1} \rightarrow \mathcal{K}^{s + r}} u \right\Vert_{\mathcal{K}^{s + r}}.\] 
The case $r = 1$ gives
\[\left\Vert \PP{s + 1}^{r} u \right\Vert_{\mathcal{K}^{s + 1 - r}} \lesssim \left\Vert \PP{s - r} \circ \PP{s + 1}^{r} u \right\Vert_{\mathcal{K}^{s - r - 1}} + \left\Vert \iota_{\mathcal{K}^{s - r + 1} \rightarrow \mathcal{K}^{s - r}} \circ \PP{s + 1}^{r} u \right\Vert_{\mathcal{K}^{s - r}}.\] 
Using the definition of $\PP{s}^{r + 1}$ and the commutativity property (\ref{eq_lem_P_comm_with_embeddings}), one finds
\begin{align*}
\left\Vert \PP{s + 1}^{r} u \right\Vert_{\mathcal{K}^{s + 1 - r}} & \lesssim \left\Vert \PP{s}^{r + 1} u \right\Vert_{\mathcal{K}^{s - r - 1}} + \left\Vert \PP{s}^{r} \circ \iota_{\mathcal{K}^{s + r + 1} \rightarrow \mathcal{K}^{s + r}} u \right\Vert_{\mathcal{K}^{s - r}} \\
& \lesssim \left\Vert \PP{s}^{r + 1} u \right\Vert_{\mathcal{K}^{s - r - 1}} + \left\Vert \iota_{\mathcal{K}^{s + r + 1} \rightarrow \mathcal{K}^{s + r}} u \right\Vert_{\mathcal{K}^{s + r}}.
\end{align*}
This completes the proof of \emph{(ii)}.

\paragraph{Step 7: End of the proof of \emph{(iii)}.}
Fix $s \in \mathbb{R}$. We prove (\ref{eq_prop_shifting_op_4}), that is,
\[\left\Vert \left(\PP{s}^{r} - \Shift{s}^{r} \right) u \right\Vert_{\mathcal{K}^{s - r}} \lesssim \left\Vert \iota_{\mathcal{K}^{s + r} \rightarrow \mathcal{K}^{s + r - 1}} u \right\Vert_{\mathcal{K}^{s + r - 1}},\]
for $r \in \mathbb{N}^\ast$ and $u \in \mathcal{K}^{s + r}$. Note that this is true for $r = 1$: one has
\[\left\Vert \left(\PP{s} - \Shift{s}^{1} \right) u \right\Vert_{\mathcal{K}^{s - 1}} = \left\Vert \iota_{\mathcal{K}^{s + 1} \rightarrow \mathcal{K}^{s - 1}} u \right\Vert_{\mathcal{K}^{s - 1}} \lesssim \left\Vert \iota_{\mathcal{K}^{s + 1} \rightarrow \mathcal{K}^{s}} u \right\Vert_{\mathcal{K}^{s}},\]
for all $u \in \mathcal{K}^{s + 1}$. Now, assume that the result is true for some $r \in \mathbb{N}^\ast$, and take $u \in \mathcal{K}^{s + r + 1}$. We write
\begin{align*}
\PP{s}^{r + 1} - \Shift{s}^{r + 1} & = \PP{s - r} \circ \PP{s + 1}^{r} - \Shift{s - r}^{1} \circ \Shift{s + 1}^{r} \\
& = \PP{s - r} \circ \PP{s + 1}^{r} - \left( \PP{s - r} + i \mu \iota_{\mathcal{K}^{s - r + 1} \rightarrow \mathcal{K}^{s - r - 1}} \right) \circ \Shift{s + 1}^{r} \\
& = \PP{s - r} \circ \left( \PP{s + 1}^{r} - \Shift{s + 1}^{r} \right) - i \mu \iota_{\mathcal{K}^{s - r + 1} \rightarrow \mathcal{K}^{s - r - 1}} \circ \Shift{s + 1}^{r} \\
& = \PP{s - r} \circ \left( \PP{s + 1}^{r} - \Shift{s + 1}^{r} \right) - i \mu \Shift{s - 1}^{r} \circ \iota_{\mathcal{K}^{s + r + 1} \rightarrow \mathcal{K}^{s + r - 1}}
\end{align*}
and we use the continuity of $\PP{s - r}$ and $\Shift{s - 1}^{r}$ to find
\[\left\Vert \left( \PP{s}^{r + 1} - \Shift{s}^{r + 1} \right) u \right\Vert_{\mathcal{K}^{s - r - 1}} \lesssim \left\Vert \left( \PP{s + 1}^{r} - \Shift{s + 1}^{r} \right) u \right\Vert_{\mathcal{K}^{s - r + 1}} + \left\Vert \iota_{\mathcal{K}^{s + r + 1} \rightarrow \mathcal{K}^{s + r - 1}} u \right\Vert_{\mathcal{K}^{s + r - 1}}.\]
By induction, we get
\[\left\Vert \left( \PP{s}^{r + 1} - \Shift{s}^{r + 1} \right) u \right\Vert_{\mathcal{K}^{s - r - 1}} \lesssim \left\Vert \iota_{\mathcal{K}^{s + r + 1} \rightarrow \mathcal{K}^{s + r}} u \right\Vert_{\mathcal{K}^{s + r}},\]
and this completes the proof.

\subsection{Solutions of the wave equations}

\subsubsection{Proof of Theorem \ref{thm_LLT_dir}}

The proof of Theorem \ref{thm_LLT_dir} is organized as follows. First, we check that the assumptions of the Hille-Yosida theorem are fulfilled, to construct the solution for $s \geq 0$, with the regularity result 
\[u \in \bigcap_{k = 0}^{s + 1} \mathscr{C}^k(\mathbb{R}, \mathcal{K}^{s + 1 - k}).\]
Second, we construct the solution for $s < 0$ by using the shift operator of Proposition \ref{prop_main_properties_K^s}-\emph{(iii)}. Third, we prove Theorem \ref{thm_LLT_dir}-\emph{(ii)} and the regularity result
\[u \in \bigcap_{k = 0}^{\infty} \mathscr{C}^k(\mathbb{R}, \mathcal{K}^{s + 1 - k}), \quad s \in \mathbb{R}.\]
Fourth, we construct the solution of the wave equation with a source term as is Theorem \ref{thm_LLT_dir}-\emph{(iv)}. Finally, we prove the results about the normal derivative. Note that by interpolation, we can always assume that $s \in \mathbb{Z}$.

\paragraph{Step 1: Construction of the solution for $s \geq 0$ with the Hille-Yosida theorem.} 

We write $\mathcal{X}$ for the Hilbert space $\mathcal{K}^1 \times \mathcal{K}^0$ and $A: \mathcal{X} \longrightarrow \mathcal{X}$ for the unbounded operator
\[A = 
\begin{pmatrix}
0 & - \mathrm{Id} \\
-\PP{1} & 0
\end{pmatrix}\]
with domain $D(A) = \mathcal{K}^2 \times \mathcal{K}^1$. We also write $U = \binom{u}{\partial_t u}$ so that the wave equation reads (formally at first)
\[\partial_t U + A U = 0 \quad \text{ with } \quad U(0) = \binom{u^0}{u^1}.\]

\begin{lem}[The iterated domains of $A$]\label{lem_iterated_domains}
For $k \in \mathbb{N}$, one has
\begin{equation}\label{eq_lem_iterated_domains}
D(A^k) = \mathcal{K}^{k + 1} \times \mathcal{K}^{k}.
\end{equation}
\end{lem}

\begin{proof}
By definition, one has $D(A^0) = \mathcal{K}^1 \times \mathcal{K}^0, \quad D(A^1) = \mathcal{K}^2 \times \mathcal{K}^1$, and 
\[D(A^{k + 1}) = \left\{U \in D(A^k), A^k U \in D(A) \right\}, \quad k \in \mathbb{N}^\ast.\]
Note that we can omit the embeddings here, as we are only working with subspaces of $L^2(M, \mathbb{C}^N)$. Fix $k \in \mathbb{N}$ such that (\ref{eq_lem_iterated_domains}) is true. The fact that 
\[\mathcal{K}^{k + 2} \times \mathcal{K}^{k + 1} \subset D(A^{k + 1})\]
follows from the definitions. Take $U = \left( u^0, u^1 \right) \in D(A^{k + 1})$. One has $U \in D(A^k)$ and $AU \in D(A^k)$. This gives $u^1 \in \mathcal{K}^{k + 1}$, $u^0 \in \mathcal{K}^{k + 1}$ and $\mathsf{P} u^0 \in \mathcal{K}^{k}$. By Proposition \ref{prop_main_properties_K^s}-\emph{(ii)}, one finds $u^0 \in \mathcal{K}^{k + 2}$ and so $U \in \mathcal{K}^{k + 2} \times \mathcal{K}^{k + 1}$.
\end{proof}

\begin{lem}[Assumptions of the Hille-Yosida theorem]\label{lem_assump_HY}
The operator $A$ is closed, and $D(A)$ is dense in $\mathcal{X}$. There exists $\omega \in \mathbb{R}$ such that the resolvent set of $A$ contains $(- \infty, - \omega)$ and such that for all $\lambda < - \omega$ and all $k \in \mathbb{N}^\ast$, one has
\[\left\Vert R_\lambda(A)^k \right\Vert_{\mathcal{L}(\mathcal{X})} \leq \frac{1}{\left\vert \omega + \lambda \right\vert ^k},\]
where $R_\lambda(A) = \left( \lambda \mathrm{Id}_\mathcal{X} - A \right)^{-1}$. The same is true for the operator $-A$.
\end{lem}

The proof of Lemma \ref{lem_assump_HY} is given below. Fix $s \in \mathbb{N}$ and 
\[\binom{u^0}{u^1} \in \mathcal{K}^{s + 1} \times \mathcal{K}^{s} = D(A^s).\]
Using the Hille-Yosida theorem (see for example \cite{Vrabie}, Theorem 3.3.1 and Corollary 2.4.1), together with Lemma \ref{lem_iterated_domains} and Lemma \ref{lem_assump_HY}, one obtains that there exists a unique solution
\[U \in \bigcap_{l = 0}^s \mathscr{C}^l(\mathbb{R}, D(A^{s-l})) = \bigcap_{l = 0}^s \mathscr{C}^l(\mathbb{R}, \mathcal{K}^{s - l + 1} \times \mathcal{K}^{s - l})\]
of 
\[\left \{
\begin{array}{rcc}
\partial_t U + AU & = & 0 \\
U(0) & = & \left( u^0, u^1 \right)
\end{array}.
\right.\]
One can check that in fact, $U$ is of the form $\left(u, \partial_t u\right)$, with 
\[u \in \bigcap_{l = 0}^{s + 1} \mathscr{C}^l(\mathbb{R}, \mathcal{K}^{s - l + 1}).\]
Note that this gives $u \in H^{s + 1}((0, T) \times M, \mathbb{C}^N)$.

\begin{proof}[Proof of Lemma \ref{lem_assump_HY}.] The fact that $D(A)$ is dense in $\mathcal{X}$ is well-known. We prove that $A$ is closed. Consider a sequence $\left((u^0_n, u^1_n)\right)_n$ of elements of $D(A)$ such that 
\[\binom{u^0_n}{u^1_n} \underset{n \rightarrow \infty}{\longrightarrow} \binom{u^0}{u^1} \in \mathcal{X} \quad \text{ and } \quad A \binom{u^0_n}{u^1_n} \underset{n \rightarrow \infty}{\longrightarrow} \binom{v^0}{v^1} \in \mathcal{X}.\]
By definition one has 
\[u^1_n \underset{n \rightarrow \infty}{\longrightarrow} -v^0 \quad \text{in } \mathcal{K}^1 \quad \text{ and } \quad \PP{1} u^0_n \underset{n \rightarrow \infty}{\longrightarrow} -v^1 \quad \text{in } \mathcal{K}^0.\]
In particular, this gives $u^1 = - v^0 \in \mathcal{K}^1$. One also has $u^0_n \longrightarrow u^0$ in $\mathcal{K}^1$ yielding $\PP{0} u^0_n \longrightarrow \PP{0} u^0$ in $\mathcal{K}^{-1}$. This implies $\PP{0} u^0 = -v^1 \in \mathcal{K}^0$, so the ellipticity estimate for $\mathsf{P}$ (Proposition \ref{prop_main_properties_K^s}-\emph{(ii)}) gives $u^0 \in \mathcal{K}^2$. Thus, $A$ is closed.

Now, we show that there exists $\omega \in \mathbb{R}$ such that if $\lambda \in \mathbb{R}$ satisfies $\vert \lambda\vert + \omega > 0$ then
\begin{equation}\label{eq_proof_LLT_dir_1}
\left\Vert \left( \lambda \mathrm{Id} - A \right) U \right\Vert_{\mathcal{X}} \geq \left( \vert \lambda\vert + \omega \right) \Vert U \Vert_{\mathcal{X}} 
\end{equation}
for all $U \in D(A)$. Recall that $\mathcal{X}$ is a Hilbert space for scalar product associated with the norm
\[\left\Vert \left( u^0, u^1 \right) \right\Vert_{\mathcal{X}}^2 = \left\Vert \nabla u^0 \right\Vert_{L^2(M, \mathbb{C}^N)}^2 + \left\Vert u^1 \right\Vert_{L^2(M, \mathbb{C}^N)}^2.\]
Fix $U = ( u^0, u^1 ) \in D(A)$, and write
\begin{align}
\left\Vert \left( \lambda \mathrm{Id} - A \right) U \right\Vert_{\mathcal{X}}^2 = \ & \vert \lambda\vert ^2 \left\Vert U \right\Vert_{\mathcal{X}}^2 + \left\Vert A U \right\Vert_{\mathcal{X}}^2 - 2 \Re \left\langle \lambda U, A U \right\rangle_{\mathcal{X}} \nonumber \\
\geq \ & \vert \lambda\vert ^2 \left\Vert U \right\Vert_{\mathcal{X}}^2 - 2 \lambda \Re \left\langle U, A U \right\rangle_{\mathcal{X}} .\label{eq_proof_LLT_dir_2}
\end{align}
By definition, one has
\begin{align*}
\Re \left\langle U, A U \right\rangle_{\mathcal{X}} 
& = - \Re \left\langle \nabla u^0, \nabla u^1 \right\rangle_{L^2(M, \mathbb{C}^N)} - \Re \left\langle \PP{1} u^0, u^1 \right\rangle_{L^2(M, \mathbb{C}^N)} \\
& = - \Re \left\langle \nabla u^0, \nabla u^1 \right\rangle_{L^2(M, \mathbb{C}^N)} - \Re \left\langle \left( \Delta - X - q \right) u^0, u^1 \right\rangle_{L^2(M, \mathbb{C}^N)}.
\end{align*}
Integrating by parts, one finds 
\[\Re \left\langle U, A U \right\rangle_{\mathcal{X}} = \Re \left\langle \left( X + q \right) u^0, u^1 \right\rangle_{L^2(M, \mathbb{C}^N)},\]
and using the Cauchy-Schwarz and Poincaré inequalities, this gives
\[\left\vert \Re \left\langle U, A U \right\rangle_{\mathcal{X}} \right\vert \lesssim \left\Vert \nabla u^0 \right\Vert_{L^2(M, \mathbb{C}^N)} \left\Vert u^1 \right\Vert_{L^2(M, \mathbb{C}^N)} \lesssim \Vert U \Vert_{\mathcal{X}}^2.\]
Coming back to (\ref{eq_proof_LLT_dir_2}), one finds that there exists $c > 0$ such that
\[\left\Vert \left( \lambda \mathrm{Id} - A \right) U \right\Vert_{\mathcal{X}}^2 \geq \left(\vert \lambda\vert ^2 - c \vert \lambda\vert \right) \left\Vert U \right\Vert_{\mathcal{X}}^2 \geq \left(\vert \lambda\vert - c \right)^2 \left\Vert U \right\Vert_{\mathcal{X}}^2 \]
for $\vert \lambda\vert > c$. This gives (\ref{eq_proof_LLT_dir_1}) with $\omega = -c$.

Next, we show that the operator $\lambda \mathrm{Id} - A: D(A) \rightarrow \mathcal{X}$ is onto if $\vert \lambda \vert$ is sufficiently large. Fix $(v^0, v^1) \in \mathcal{X}$. We seek $\left( u^0, u^1 \right) \in D(A)$ such that 
\[\left( \lambda \mathrm{Id} - A \right) \binom{u^0}{u^1} = \binom{v^0}{v^1},\]
which reads
\[\left \{
\begin{array}{rcccl}
u^1 & = & - v^0 + \lambda u^0 \\
\left( \Delta - X - q \right) u^0 - \lambda^2 u^0 & = & - v^1 - \lambda v^0
\end{array}
\right. .\]
Using the Lax-Milgram theorem and the ellipticity of $\mathsf{P}$, it follows if the coercivity inequality 
\begin{equation}\label{eq_proof_lem_assumptions_HY}
\left\vert \left\Vert \nabla u^0 \right\Vert_{L^2(M, \mathbb{C}^N)}^2 + \lambda^2 \left\Vert u^0 \right\Vert_{L^2(M, \mathbb{C}^N)}^2 + \left\langle (X + q) u^0, u^0 \right\rangle_{L^2(M, \mathbb{C}^N)} \right\vert \gtrsim \Vert \nabla u^0 \Vert_{L^2(M, \mathbb{C}^N)}^2
\end{equation}
is proven. As above, for $\epsilon > 0$, write
\[\left\vert \left\langle (X + q) u^0, u^0 \right\rangle_{L^2(M, \mathbb{C}^N)} \right\vert \lesssim \epsilon \Vert \nabla u^0 \Vert_{L^2(M, \mathbb{C}^N)}^2 + \frac{1}{\epsilon} \Vert u^0 \Vert_{L^2(M, \mathbb{C}^N)}^2,\]
and find, for $c_1 > 0$, $c_2 > 0$, by choosing $\epsilon$ sufficiently small,
\begin{align*}
& \left\vert \left\Vert \nabla u^0 \right\Vert_{L^2(M, \mathbb{C}^N)}^2 + \lambda^2 \left\Vert u^0 \right\Vert_{L^2(M, \mathbb{C}^N)}^2 + \left\langle (X + q) u^0, u^0 \right\rangle_{L^2(M, \mathbb{C}^N)} \right\vert \\
\geq \ & c_1 \left\Vert \nabla u^0 \right\Vert_{L^2(M, \mathbb{C}^N)}^2 + \left( \lambda^2 - c_2 \right)
\left\Vert u^0 \right\Vert_{L^2(M, \mathbb{C}^N)}^2 \\
\gtrsim \ & \Vert \nabla u^0 \Vert_{L^2(M, \mathbb{C}^N)}^2
\end{align*}
for $\vert \lambda \vert$ sufficiently large. This gives (\ref{eq_proof_lem_assumptions_HY}).

At this stage, one has proved that there exists $\omega \in \mathbb{R}$ such that $(- \infty, - \omega)$ is contained in the resolvent set of both $A$ and $-A$, and for $\lambda$ such that $\vert \lambda\vert + \omega > 0$, one has
\[\left\Vert R_\lambda(A) \right\Vert_{\mathcal{L}(\mathcal{X})} \leq \frac{1}{\left\vert \omega + \vert \lambda\vert \right\vert }.\]
This completes the proof of Lemma \ref{lem_assump_HY}.
\end{proof}

\paragraph{Step 2: Construction of the solution for $s < 0$.} 

The idea of this step is to construct the solution with the shift operator and the solution in $\mathcal{K}^3 \times \mathcal{K}^2$ or $\mathcal{K}^2 \times \mathcal{K}^1$, depending of the parity of $s$. Fix $s \in \mathbb{Z}$, $s < 0$. There exist $\sigma \in \mathbb{N}^\ast$ and $\alpha \in \{1, 2\}$ such that $s = - 2 \sigma + \alpha$. Take $\left( u^0, u^1 \right) \in \mathcal{K}^{s + 1} \times \mathcal{K}^{s}$. Let $\left( \tilde{u}^0, \tilde{u}^1 \right)$ be the unique element of $\mathcal{K}^{\alpha + 1} \times \mathcal{K}^{\alpha}$ such that
\[\left( u^0, u^1 \right) = \left( \Shift{s + 1 + \sigma}^\sigma \tilde{u}^0, \Shift{s + \sigma}^\sigma \tilde{u}^1 \right).\]
Let $\tilde{u}$ be the solution associated with $\left( \tilde{u}^0, \tilde{u}^1 \right)$ defined above. Set $u(t) = \Shift{s + 1 + \sigma}^\sigma \tilde{u}(t)$, for $t \in \mathbb{R}$. We will refer to $u$ as the solution of
\[\left \{
\begin{array}{rcccl}
\partial_t^2 u - \mathsf{P} u & = & 0 & \quad & \text{in } \mathbb{R} \times M, \\
\left( u(0, \cdot), \partial_t u(0, \cdot) \right) & = & \left( u^0, u^1 \right) & \quad & \text{in } M, \\
u & = & 0 & \quad & \text{on } \mathbb{R} \times \partial M.
\end{array}
\right.\]

We prove that
\[u \in \mathscr{C}^0(\mathbb{R}, \mathcal{K}^{s + 1}) \cap \mathscr{C}^1(\mathbb{R}, \mathcal{K}^{s}) \cap \mathscr{C}^2(\mathbb{R}, \mathcal{K}^{s - 1}),\]
and that $\partial_t^2 u = \PP{s} u$. In particular, if $s = -1$ or $-2$, it implies $u \in H^{s + 1}((0, T) \times M, \mathbb{C}^N)$ for $T > 0$, using the embedding
\[\mathscr{C}^0((0, T), H^{-1}(M, \mathbb{C}^N)) \hooklongrightarrow H^{- 1}((0, T) \times M, \mathbb{C}^N)\]
for the case $s = -2$.

\textbf{Continuity.} The continuity of $u$ is a consequence of that of the shift operator $\Shift{s + 1 + \sigma}^\sigma$ and that of $\tilde{u}$. In addition, for $T > 0$, one has
\[\left\Vert u \right\Vert_{L^\infty([0, T], \mathcal{K}^{s + 1})} \lesssim \left\Vert \tilde{u} \right\Vert_{L^\infty([0, T], \mathcal{K}^{\alpha + 1})} \lesssim \left\Vert \left( \tilde{u}^0, \tilde{u}^1 \right) \right\Vert_{\mathcal{K}^{\alpha + 1} \times \mathcal{K}^{\alpha + 1}} \lesssim \left\Vert \left( u^0, u^1 \right) \right\Vert_{\mathcal{K}^{s + 1} \times \mathcal{K}^{s}}.\]

\textbf{First-order time-derivative.} We show that
\[\iota_{\mathcal{K}^{s + 1} \rightarrow \mathcal{K}^{s}} \left( \frac{u(t + \epsilon) - u(t)}{\epsilon} \right) \underset{\epsilon \rightarrow 0}{\longrightarrow} \Shift{s + \sigma}^\sigma \partial_t \tilde{u}(t) \in \mathcal{K}^{s}, \quad t \in \mathbb{R}.\]
By Proposition \ref{prop_main_properties_K^s}, one has
\[\iota_{\mathcal{K}^{s + 1} \rightarrow \mathcal{K}^{s}} \circ \Shift{s + 1 + \sigma}^\sigma = \Shift{s + \sigma}^\sigma \circ \iota_{\mathcal{K}^{\alpha + 1} \rightarrow \mathcal{K}^{\alpha}} = \Shift{s + \sigma}^\sigma,\]
where the last embedding can be omitted as it is just an inclusion. Hence, for $t \in \mathbb{R}$ and $\epsilon \in \mathbb{R}$, we can write
\begin{align*}
\left\Vert \iota_{\mathcal{K}^{s + 1} \rightarrow \mathcal{K}^{s}} \left( \frac{u(t + \epsilon) - u(t)}{\epsilon} \right) - \Shift{s + \sigma}^\sigma \partial_t \tilde{u}(t) \right\Vert_{\mathcal{K}^{s}} 
& = \left\Vert \Shift{s + \sigma}^\sigma \left( \frac{\tilde{u}(t + \epsilon) - \tilde{u}(t)}{\epsilon} - \partial_t \tilde{u}(t) \right) \right\Vert_{\mathcal{K}^{s}} \\
& \lesssim \left\Vert \frac{\tilde{u}(t + \epsilon) - \tilde{u}(t)}{\epsilon} - \partial_t \tilde{u}(t) \right\Vert_{\mathcal{K}^{\alpha}} \\
& \underset{\epsilon \rightarrow 0}{\longrightarrow} 0.
\end{align*}
As above, one also has
\[\left\Vert \partial_t u \right\Vert_{L^\infty([0, T], \mathcal{K}^{s})} \lesssim \left\Vert \partial_t \tilde{u} \right\Vert_{L^\infty([0, T], \mathcal{K}^{\alpha})} \lesssim \left\Vert \left( \tilde{u}^0, \tilde{u}^1 \right) \right\Vert_{\mathcal{K}^{\alpha + 1} \times \mathcal{K}^{\alpha + 1}} \lesssim \left\Vert \left( u^0, u^1 \right) \right\Vert_{\mathcal{K}^{s + 1} \times \mathcal{K}^{s}}, \quad T > 0.\]

\textbf{Second-order time-derivative.} As above, one shows that $u \in \mathscr{C}^2(\mathbb{R}, \mathcal{K}^{s - 1})$, with 
\[\partial_t^2 u(t) = \Shift{s - 1 + \sigma}^\sigma \partial_t^2 \tilde{u}(t), \quad t \in \mathbb{R},\]
and
\[\left\Vert \partial_t^2 u \right\Vert_{L^\infty([0, T], \mathcal{K}^{s - 1})} \lesssim \left\Vert \left( u^0, u^1 \right) \right\Vert_{\mathcal{K}^{s + 1} \times \mathcal{K}^{s}}, \quad T > 0.\]
In particular, for $t \in \mathbb{R}$, one finds $\partial_t^2 u(t) = \Shift{s - 1 + \sigma}^\sigma \partial_t^2 \tilde{u}(t) = \Shift{s - 1 + \sigma}^\sigma \PP{\alpha} \tilde{u}(t)$, and by Proposition \ref{prop_main_properties_K^s}-\emph{(iii)}, this gives
\[\partial_t^2 u(t) = \PP{s} \Shift{s + 1 + \sigma}^\sigma \tilde{u}(t) = \PP{s} u(t).\]

\paragraph{Step 3: Regularity, uniqueness and approximation.} Here, we prove the uniqueness result of Theorem \ref{thm_LLT_dir} for $s \in \mathbb{Z}$, the regularity result \emph{(i)}, and then \emph{(ii)} and the uniqueness result of Theorem \ref{thm_LLT_dir} for $s \in \mathbb{R}$.

\textbf{Uniqueness for $s \in \mathbb{Z}$.} For $s \in \mathbb{N}$, the uniqueness result of Theorem \ref{thm_LLT_dir} is given by the Hille-Yosida Theorem. Fix $s \in \mathbb{Z}$, $s < 0$, and $\left( u^0, u^1 \right) \in \mathcal{K}^{s + 1} \times \mathcal{K}^{s}$. Let $v \in \mathscr{C}^0(\mathbb{R}, \mathcal{K}^{s + 1}) \cap \mathscr{C}^1(\mathbb{R}, \mathcal{K}^{s}) \cap \mathscr{C}^2(\mathbb{R}, \mathcal{K}^{s - 1})$ be such that $\partial_t^2 v(t) = \PP{s} v(t)$ for all $t \in \mathbb{R}$, and $\left( v(0, \cdot), \partial_t v(0, \cdot) \right) = \left( u^0, u^1 \right)$ in $M$. As above, let $\sigma \in \mathbb{N}^\ast$ and $\alpha \in \{1, 2\}$ be such that $s = - 2 \sigma + \alpha$. For $t \in \mathbb{R}$, define 
\[\tilde{v}(t) = \left( \Shift{s + 1 + \sigma}^\sigma \right)^{-1} v(t).\]
As in Step 1, one shows that $\tilde{v} \in \mathscr{C}^0(\mathbb{R}, \mathcal{K}^{\alpha + 1}) \cap \mathscr{C}^1(\mathbb{R}, \mathcal{K}^{\alpha}) \cap \mathscr{C}^2(\mathbb{R}, \mathcal{K}^{\alpha - 1})$, with $\partial_t \tilde{v}(t) = \left( \Shift{s + \sigma}^\sigma \right)^{-1} \partial_t v(t)$ and $\partial_t^2 \tilde{v}(t) = \left( \Shift{s - 1 + \sigma}^\sigma \right)^{-1} \partial_t^2 v(t)$. By Proposition \ref{prop_main_properties_K^s}-\emph{(iii)}, one finds
\[\partial_t^2 \tilde{v}(t) = \left( \Shift{s - 1 + \sigma}^\sigma \right)^{-1} \circ \PP{s} v(t) = \PP{\alpha} \circ \left( \Shift{s + 1 + \sigma}^\sigma \right)^{-1} v(t) = \PP{\alpha} \tilde{v}(t).\]
Hence, the functions $\tilde{v}$ and $\tilde{u}$ (defined in the previous step) satisfy the same wave equation. Using the uniqueness in the case $s \geq 0$, one finds $\tilde{v} = \tilde{u}$, and so
\[v(t) = \Shift{s + 1 + \sigma}^\sigma \tilde{v}(t) = \Shift{s + 1 + \sigma}^\sigma \tilde{u}(t) = u(t).\]

\textbf{The regularity result \emph{(i)}.} Take $s \in \mathbb{Z}$ and $\left( u^0, u^1 \right) \in \mathcal{K}^{s + 1} \times \mathcal{K}^{s}$. We know that 
\[u \in \mathscr{C}^0(\mathbb{R}, \mathcal{K}^{s + 1}) \cap \mathscr{C}^1(\mathbb{R}, \mathcal{K}^{s}) \cap \mathscr{C}^2(\mathbb{R}, \mathcal{K}^{s - 1}),\]
and we show that 
\[u \in \bigcap_{k \in \mathbb{N}} \mathscr{C}^k(\mathbb{R}, \mathcal{K}^{s + 1 - k}).\]
One can prove that
\[\PP{s} u \in \mathscr{C}^0(\mathbb{R}, \mathcal{K}^{s - 1}) \cap \mathscr{C}^1(\mathbb{R}, \mathcal{K}^{s - 2}) \cap \mathscr{C}^2(\mathbb{R}, \mathcal{K}^{s - 3}),\]
and $\partial_t^2 \left( \PP{s} u \right) = \PP{s - 2} \circ \PP{s} u$. For example, to show that $\PP{s} u \in \mathscr{C}^1(\mathbb{R}, \mathcal{K}^{s - 2})$, write
\begin{align*}
& \left\Vert \iota_{\mathcal{K}^{s - 1} \rightarrow \mathcal{K}^{s - 2}} \left( \frac{\PP{s} u(t + \epsilon) - \PP{s} u(t)}{\epsilon} \right) - \PP{s - 1} \partial_t u(t) \right\Vert_{\mathcal{K}^{s - 2}} \\
= & \left\Vert \PP{s - 1} \left( \iota_{\mathcal{K}^{s + 1} \rightarrow \mathcal{K}^{s}} \left( \frac{u(t + \epsilon) - u(t)}{\epsilon} \right) - \partial_t u(t) \right) \right\Vert_{\mathcal{K}^{s - 2}} \\
\lesssim & \left\Vert \iota_{\mathcal{K}^{s + 1} \rightarrow \mathcal{K}^{s}} \left( \frac{u(t + \epsilon) - u(t)}{\epsilon} \right) - \partial_t u(t) \right\Vert_{\mathcal{K}^{s}} \\
& \underset{\epsilon \rightarrow 0}{\longrightarrow} 0.
\end{align*}
By uniqueness, $\PP{s} u = \partial_t^2 u$ is the solution associated with the initial data $\left(\PP{s} u^0, \PP{s - 1} u^1 \right)$, implying
\[u \in \mathscr{C}^3(\mathbb{R}, \mathcal{K}^{s - 2}) \cap \mathscr{C}^4(\mathbb{R}, \mathcal{K}^{s - 3}).\]
One also has
\[\left\Vert \partial_t^3 u \right\Vert_{L^\infty((0, T), \mathcal{K}^{s - 2})} = \left\Vert \PP{s - 1} \partial_t u \right\Vert_{L^\infty((0, T), \mathcal{K}^{s - 2})} \lesssim \left\Vert \partial_t u \right\Vert_{L^\infty((0, T), \mathcal{K}^{s})} \lesssim \left\Vert \left( u^0, u^1 \right) \right\Vert_{\mathcal{K}^{s + 1} \times \mathcal{K}^{s}},\]
and 
\[\left\Vert \partial_t^4 u \right\Vert_{L^\infty((0, T), \mathcal{K}^{s - 3})} = \left\Vert \PP{s - 1}^2 u \right\Vert_{L^\infty((0, T), \mathcal{K}^{s - 3})} \lesssim \left\Vert u \right\Vert_{L^\infty((0, T), \mathcal{K}^{s + 1})} \lesssim \left\Vert \left( u^0, u^1 \right) \right\Vert_{\mathcal{K}^{s + 1} \times \mathcal{K}^{s}},\]
for $T > 0$. The result follows by iteration.

\textbf{Proof of \emph{(ii)}.} Take $s \in \mathbb{Z}$ and $\delta > 0$. For $\left( u^0, u^1 \right) \in \mathcal{K}^{s + \delta + 1} \times \mathcal{K}^{s + \delta}$, if $u$ is the solution with initial data $\left( u^0, u^1 \right)$, then arguing as above, one has
\[\iota_{\mathcal{K}^{s + \delta + 1} \rightarrow \mathcal{K}^{s + 1}} u \in \mathscr{C}^0(\mathbb{R}, \mathcal{K}^{s + 1}) \cap \mathscr{C}^1(\mathbb{R}, \mathcal{K}^{s}) \cap \mathscr{C}^2(\mathbb{R}, \mathcal{K}^{s - 1}),\]
with $\partial_t \left( \iota_{\mathcal{K}^{s + \delta + 1} \rightarrow \mathcal{K}^{s + 1}} u \right) = \iota_{\mathcal{K}^{s + \delta} \rightarrow \mathcal{K}^{s}} \partial_t u$ and
\[\partial_t^2 \left( \iota_{\mathcal{K}^{s + \delta + 1} \rightarrow \mathcal{K}^{s + 1}} u \right) = \iota_{\mathcal{K}^{s + \delta - 1} \rightarrow \mathcal{K}^{s - 1}} \partial_t^2 u = \iota_{\mathcal{K}^{s + \delta - 1} \rightarrow \mathcal{K}^{s - 1}} \circ \PP{s + \delta} u = \PP{s} \circ \iota_{\mathcal{K}^{s + \delta + 1} \rightarrow \mathcal{K}^{s + 1}} u.\]
By uniqueness, one finds
\[\iota_{\mathcal{K}^{s + \delta + 1} \rightarrow \mathcal{K}^{s + 1}} u = \tilde{u},\]
where $\tilde{u}$ is the solution associated with $\left( \iota_{\mathcal{K}^{s + \delta + 1} \rightarrow \mathcal{K}^{s + 1}} u^0, \iota_{\mathcal{K}^{s + \delta} \rightarrow \mathcal{K}^{s}} u^1 \right)$. By interpolation, this is in fact true for all $s \in \mathbb{R}$. 

We prove the approximation result of \emph{(ii)}. Consider $s \in \mathbb{R}$, $\delta > 0$, and $\left( u^0, u^1 \right) \in \mathcal{K}^{s + 1} \times \mathcal{K}^{s}$. Let $u$ be the solution with initial data $\left( u^0, u^1 \right)$. By Proposition \ref{prop_main_properties_K^s}-\emph{(i)}, there exists a sequence $\left(( \tilde{u}_k^0, \tilde{u}_k^1 )\right)_{k \in \mathbb{N}}$ of elements of $\mathcal{K}^{s + 1 + \delta} \times \mathcal{K}^{s + \delta}$ such that, writing $u_k^0 = \iota_{\mathcal{K}^{s + 1 + \delta} \rightarrow \mathcal{K}^{s + 1}} \tilde{u}_k^0$ and $u_k^1 = \iota_{\mathcal{K}^{s + \delta} \rightarrow \mathcal{K}^s} \tilde{u}_k^1$, one has
\[\left( u_k^0, u_k^1 \right) \underset{k \rightarrow \infty}{\longrightarrow} \left( u^0, u^1 \right) \text{ in } \mathcal{K}^{s + 1} \times \mathcal{K}^{s}.\]
Denote $\tilde{u}_k \in \mathscr{C}^0(\mathbb{R}, \mathcal{K}^{s + 1 + \delta}) \cap \mathscr{C}^1(\mathbb{R}, \mathcal{K}^{s + \delta})$ and $u_k \in \mathscr{C}^0(\mathbb{R}, \mathcal{K}^{s + 1}) \cap \mathscr{C}^1(\mathbb{R}, \mathcal{K}^{s})$ the solutions with initial data $\left( \tilde{u}_k^0, \tilde{u}_k^1 \right)$ and $\left( u_k^0, u_k^1 \right)$. One has $\iota_{\mathcal{K}^{s + 1 + \delta} \rightarrow \mathcal{K}^{s + 1}} \tilde{u}_k = u_k$ for all $k \in \mathbb{N}$, and
\[\sum_{j = 0}^2 \left\Vert \partial_t^j ( u_k - u ) \right\Vert_{L^\infty([0, T], \mathcal{K}^{s + 1 - j})} \lesssim \left\Vert \left( u_k^0, u_k^1 \right) - \left( u^0, u^1 \right) \right\Vert_{\mathcal{K}^{s + 1} \times \mathcal{K}^{s}} \underset{k \rightarrow \infty}{\longrightarrow} 0,\]
for $T > 0$. Hence, one obtains
\[\iota_{\mathcal{K}^{s + 1 + \delta} \rightarrow \mathcal{K}^{s + 1}} \tilde{u}_k \underset{k \rightarrow \infty}{\longrightarrow} u,\]
in $\mathscr{C}^0([0, T], \mathcal{K}^{s + 1}) \cap \mathscr{C}^1([0, T], \mathcal{K}^{s}) \cap \mathscr{C}^2([0, T], \mathcal{K}^{s - 1})$. In Step 5, we define the normal derivative of a solution, and we show that
\[\left\Vert \partial_\nu u \right\Vert_{H^s((0, T) \times \partial M, \mathbb{C}^N)} \lesssim \left\Vert \left( u^0, u^1 \right) \right\Vert_{\mathcal{K}^{s + 1} \times \mathcal{K}^{s}}.\]
for all $\left( u^0, u^1 \right) \in \mathcal{K}^{s + 1} \times \mathcal{K}^{s}$. Hence, we also have 
\[\left\Vert \partial_\nu \left( u_k \right) - \partial_\nu u \right\Vert_{H^s((0, T) \times \partial M, \mathbb{C}^N)} \lesssim \left\Vert \left( u_k^0, u_k^1 \right) - \left( u^0, u^1 \right)\right\Vert_{\mathcal{K}^{s + 1} \times \mathcal{K}^{s}}\]
and so
\[\partial_\nu \left( u_k \right) \underset{k \rightarrow \infty}{\longrightarrow} \partial_\nu u,\]
in $H^s((0, T) \times \partial M, \mathbb{C}^N)$.

\textbf{Uniqueness for $s \in \mathbb{R}$.} Lastly, we show that \emph{(ii)} implies the uniqueness result of Theorem \ref{thm_LLT_dir} for $s \in \mathbb{R}$. If $u$ and $v$ are two solutions of the wave equation starting from $\left( u^0, u^1 \right) \in \mathcal{K}^{s + 1} \times \mathcal{K}^{s}$, then using the uniqueness result for $\tilde{s} \in \mathbb{Z}$ such that $s > \tilde{s}$, one has 
\[\iota_{\mathcal{K}^{s + 1} \rightarrow \mathcal{K}^{\tilde{s} + 1}} u = \iota_{\mathcal{K}^{s + 1} \rightarrow \mathcal{K}^{\tilde{s} + 1}} v.\]
This gives $u = v$, as the map $\iota_{\mathcal{K}^{s + 1} \rightarrow \mathcal{K}^{\tilde{s} + 1}}$ is one-to-one.

\paragraph{Step 4: Study of the Duhamel term.} 
In this step, we construct the solution of the wave equation with a source term. We define the solution of
\begin{equation}\label{eq_proof_LLT_dir_5}
\left \{
\begin{array}{rcccl}
\partial_t^2 u - \mathsf{P} u & = & F & \quad & \text{in } (0, T) \times M, \\
\left( u(0, \cdot), \partial_t u(0, \cdot) \right) & = & 0 & \quad & \text{in } M, \\
u & = & 0 & \quad & \text{on } (0, T) \times \partial M,
\end{array}
\right.
\end{equation}
for $F \in L^1((0, T), H_0^s(M, \mathbb{C}^N))$, $s \in \mathbb{N}$. The solution could be constructed with $F \in L^1((0, T), \mathcal{K}^{s})$ instead, but this is of no use for our main results.

Fix $s \in \mathbb{N}$, $T > 0$ and $F \in \mathscr{C}^0([0, T], H_0^s(M, \mathbb{C}^N))$. For $\tau \in [0, T]$, let $u_\tau$ be the solution of 
\[\left \{
\begin{array}{rcccl}
\partial_t^2 u_\tau - \mathsf{P} u_\tau & = & 0 & \quad & \text{in } (0, T) \times M, \\
\left( u(0, \cdot), \partial_t u(0, \cdot) \right) & = & (0, F(\tau)) & \quad & \text{in } M, \\
u & = & 0 & \quad & \text{on } (0, T) \times \partial M.
\end{array}
\right.\]
As in the classical Duhamel formula, the solution of (\ref{eq_proof_LLT_dir_5}) is given by
\[\Psi(t) = \int_0^t u_\tau(t - \tau) \mathrm{d} \tau, \quad t \in [0, T].\]
The function $\Psi$ is a one-parameter integral, and as $H_0^s(M, \mathbb{C}^N) \subset \mathcal{K}^{s}$, one has
\[u_\tau \in \mathscr{C}^0([0, T], \mathcal{K}^{s + 1}) \cap \mathscr{C}^1([0, T], \mathcal{K}^{s}) \cap \mathscr{C}^2([0, T], \mathcal{K}^{s - 1})\]
for all $\tau \in [0, T]$. Thus, the following regularity results hold. First, one has $\Psi \in \mathscr{C}^0([0, T], \mathcal{K}^{s + 1})$, with 
\begin{align*}
\left\Vert \Psi \right\Vert_{L^\infty([0, T], \mathcal{K}^{s + 1})} & \leq \int_0^T \sup_{t \in [0, T]} \left\Vert u_\tau(t) \right\Vert_{\mathcal{K}^{s + 1}} \mathrm{d} \tau \\
& \lesssim \int_0^T \left\Vert \left( u_\tau(0), \partial_t u_\tau(0) \right) \right\Vert_{\mathcal{K}^{s + 1} \times \mathcal{K}^{s}} \mathrm{d} \tau \\
& \lesssim \left\Vert F \right\Vert_{L^1([0, T], \mathcal{K}^{s})} = \left\Vert F \right\Vert_{L^1([0, T], H^s)}.
\end{align*}
Second, one has $\Psi \in \mathscr{C}^1([0, T], \mathcal{K}^{s})$, with
\[\partial_t \Psi(t) = \int_0^t \partial_t u_\tau(t - \tau) \mathrm{d} \tau, \quad t \in [0, T],\]
and 
\[\left\Vert \partial_t \Psi \right\Vert_{L^\infty([0, T], \mathcal{K}^{s})} \leq \int_0^T \sup_{t \in [0, T]} \left\Vert \partial_t u_\tau(t) \right\Vert_{\mathcal{K}^{s}} \mathrm{d} \tau \lesssim \left\Vert F \right\Vert_{L^1([0, T], H^s)}.\]
As $\mathscr{C}^0([0, T], H_0^s(M, \mathbb{C}^N))$ is dense in $L^1((0, T), H_0^s(M, \mathbb{C}^N))$, the previous results hold for all $F \in L^1((0, T), H_0^s(M, \mathbb{C}^N))$. Third, one has $\Psi \in \mathscr{C}^2([0, T], \mathcal{K}^{s - 1})$, with
\[\partial_t^2 \Psi(t) = \partial_t u_t(0) + \int_0^t \partial_t^2 u_\tau(t - \tau) \mathrm{d} \tau = F(t) + \mathsf{P} \Psi(t)\]
for $t \in [0, T]$, and 
\begin{align*}
\left\Vert \partial_t^2 \Psi \right\Vert_{L^\infty([0, T], \mathcal{K}^{s - 1})} 
& \leq \left\Vert F \right\Vert_{L^\infty([0, T], \mathcal{K}^{s - 1})} + \int_0^T \sup_{t \in [0, T]} \left\Vert \partial_t^2 u_\tau(t) \right\Vert_{\mathcal{K}^{s - 1}} \mathrm{d} \tau \\
& \lesssim \left\Vert F \right\Vert_{L^\infty([0, T], \mathcal{K}^{s - 1})} + \left\Vert F \right\Vert_{L^1([0, T], H^s)}. 
\end{align*}
As $\mathscr{C}^0([0, T], H_0^s(M, \mathbb{C}^N))$ is dense in 
\[L^1((0, T), H_0^s(M, \mathbb{C}^N)) \cap \mathscr{C}^0([0, T], H^{s - 1}(M, \mathbb{C}^N)),\]
the previous results hold for $F$ in the latter space.

The following duality result will be useful later. 

\begin{lem}\label{lem_LLT_Duhamel_duality}
For $F_1, F_2 \in L^2((0, T) \times M, \mathbb{C}^N)$, one has
\begin{equation}\label{eq_proof_LLT_Duhamel_duality}
\left\langle u, F_2 \right\rangle_{L^2((0, T) \times M, \mathbb{C}^N)} = \left\langle F_1, v \right\rangle_{L^2((0, T) \times M, \mathbb{C}^N)} + \left\langle u^0, \partial_t v(T) \right\rangle_{L^2(M, \mathbb{C}^N)} - \left\langle u^1, v(T) \right\rangle_{L^2(M, \mathbb{C}^N)},
\end{equation}
where $u$ and $v$ are the solutions of 
\[\left \{
\begin{array}{rcccl}
\partial_t^2 u - \mathsf{P} u & = & F_1 & \quad & \text{in } (0, T) \times M, \\
\left( u(T, \cdot), \partial_t u(T, \cdot) \right) & = & \left(u^0, u^1 \right) & \quad & \text{in } M, \\
u & = & 0 & \quad & \text{on } (0, T) \times \partial M,
\end{array}
\right.\]
\[\left \{
\begin{array}{rcccl}
\partial_t^2 v - \mathsf{P}^\ast v & = & F_2 & \quad & \text{in } (0, T) \times M, \\
\left( v(0, \cdot), \partial_t v(0, \cdot) \right) & = & 0 & \quad & \text{in } M, \\
v & = & 0 & \quad & \text{on } (0, T) \times \partial M.
\end{array}
\right.\]
\end{lem}

\begin{proof}
For $F_1, F_2 \in \mathscr{C}^0([0, T], \mathcal{K}^1)$, an integration by parts gives (\ref{eq_proof_LLT_Duhamel_duality}). Both sides of (\ref{eq_proof_LLT_Duhamel_duality}) are continuous with respect to the norm of $L^2((0, T) \times M, \mathbb{C}^N)$, (\ref{eq_proof_LLT_Duhamel_duality}) holds for all $F_1$ and $F_2$ in $L^2((0, T) \times M, \mathbb{C}^N)$ by density.
\end{proof}

\paragraph{Step 5: Regularity of the normal derivative.} 
Fix $T > 0$. If $s = 0$, then the standard scalar proof works without any change (see for example \cite{Las-Lio-Tri}). In addition, for $\left( u^0, u^1 \right) \in \mathcal{K}^1 \times \mathcal{K}^0$ and $F \in L^1((0, T), \mathcal{K}^0)$, if $u$ is the solution of
\[\left \{
\begin{array}{rcccl}
\partial_t^2 u - \mathsf{P} u & = & F & \quad & \text{in } (0, T) \times M, \\
\left( u(0, \cdot), \partial_t u(0, \cdot) \right) & = & \left( u^0, u^1 \right) & \quad & \text{in } M, \\
u & = & 0 & \quad & \text{on } (0, T) \times \partial M,
\end{array}
\right.\]
then one has
\[\left\Vert \partial_\nu u \right\Vert_{L^2((0, T) \times \partial M, \mathbb{C}^N)} \lesssim \left\Vert \left( u^0, u^1 \right) \right\Vert_{\mathcal{K}^1 \times \mathcal{K}^0} + \left\Vert F \right\Vert_{L^1((0, T), \mathcal{K}^0)}.\]

\textbf{Case $s > 0$.} For $s \in \mathbb{N}^\ast$, we prove that
\[\partial_\nu u \in H^s((0, T) \times \partial M, \mathbb{C}^N)\]
with
\begin{equation}\label{eq_proof_LLT_dir_6}
\left\Vert \partial_\nu u \right\Vert_{H^s((0, T) \times \partial M, \mathbb{C}^N)} \lesssim \left\Vert \left( u^0, u^1 \right) \right\Vert_{\mathcal{K}^{s + 1} \times \mathcal{K}^{s}} + \left\Vert F \right\Vert_{L^1((0, T), H^s)}
\end{equation}
for $\left( u^0, u^1 \right) \in \mathcal{K}^{s + 1} \times \mathcal{K}^{s}$ and $F \in L^1((0, T), H_0^s(M, \mathbb{C}^N))$, where $u$ is the solution of the wave equation with initial data $\left( u^0, u^1 \right)$ and with source term $F$. 

We start with the case $F = 0$. By density it suffices to prove that for $\left( u^0, u^1 \right) \in \mathcal{K}^{s + 2} \times \mathcal{K}^{s + 1}$, one has
\[\left\Vert \partial_\nu u \right\Vert_{H^s((0, T) \times \partial M, \mathbb{C}^N)} \lesssim \left\Vert \left( u^0, u^1 \right) \right\Vert_{\mathcal{K}^{s + 1} \times \mathcal{K}^{s}}.\]

Fix $k \in \llbracket 0, s \rrbracket$, and let $L_1, \cdots, L_k$ be smooth vector fields on the Riemannian manifold $\partial M$. We prove
\begin{equation}\label{eq_proof_LLT_dir_6_bis}
L_1 \cdots L_k \partial_t^{s - k} \partial_\nu u \in L^2((0, T) \times \partial M, \mathbb{C}^N).
\end{equation}
For $j \in \llbracket 0, k \rrbracket$, there exists a smooth vector field $\tilde{L}_j$ on $M$ such that $\tilde{L}_j = L_j$ on the boundary. Define
\[v = \tilde{L}_1 \cdots \tilde{L}_k \partial_t^{s - k} u.\]
Note that $u \in \mathscr{C}^{s-k}(\mathbb{R}, \mathcal{K}^{k + 2})$, so that $v \in H^2(M, \mathbb{C}^N)$. As $s$ is positive here, we can omit the subscript of $\mathsf{P}$ and use the usual differential operator $\mathsf{P}$. We can write
\[\left( \partial_t^2 - \mathsf{P} \right) v = \left[ \left( \partial_t^2 - \mathsf{P} \right), \tilde{L}_1 \cdots \tilde{L}_k \partial_t^{s - k} \right] u = R u \]
where $R$ is a differential operator of order $s + 1$. As $u \in H^{s + 2}((0, T) \times M, \mathbb{C}^N)$, one has $R u \in H^1((0, T) \times M, \mathbb{C}^N)$. We claim that 
\begin{equation}\label{eq_proof_LLT_dir_7}
\left(v(0), \partial_t v(0)\right) \in H_0^1(M) \times L^2(M).
\end{equation}
Then, by the standard case, one has 
\[\left\Vert \partial_\nu v \right\Vert_{L^2((0, T) \times \partial M, \mathbb{C}^N)} \lesssim \left\Vert \left( v(0), \partial_t v(0) \right) \right\Vert_{\mathcal{K}^1 \times \mathcal{K}^0} + \left\Vert R u \right\Vert_{L^2((0, T) \times M, \mathbb{C}^N)} \lesssim \left\Vert \left( u^0, u^1 \right) \right\Vert_{\mathcal{K}^{s + 1} \times \mathcal{K}^{s}}\]
implying (\ref{eq_proof_LLT_dir_6_bis}). Indeed, if $\mathtt{N}$ is a smooth vector field on $M$ that coincides with the unit normal vector at the boundary, then one has
\[\partial_\nu v = \left( \mathtt{N} v \right)_{\vert \partial M} = \left(\tilde{L}_1 \cdots \tilde{L}_k \partial_t^{s - k} ( \mathtt{N} u ) \right)_{\vert \partial M} + \left(R \partial_t^{s - k} u \right)_{\vert \partial M}\]
where $R$ is a time-independent differential operator of order $k - 1$. Using $u \in \mathscr{C}^{s - k}((0, T), \mathcal{K}^{k + 1})$ and the fact that for $j \in \llbracket 0, k \rrbracket$, the vector field $\tilde{L}_j$ is tangent to the boundary, one finds (\ref{eq_proof_LLT_dir_6_bis}).

One has $\partial_t^{s - k} u \in \mathscr{C}^0(\mathbb{R}, \mathcal{K}^{k + 2})$, and (\ref{eq_proof_LLT_dir_7}) follows if one proves that $w \in \mathcal{K}^{k + 2}$ implies $\tilde{L}_1 \cdots \tilde{L}_k w \in H_0^1(M, \mathbb{C}^N)$. For $w \in H^{k + 1}(M, \mathbb{C}^N)$, one has
\[\left(\tilde{L}_1 \cdots \tilde{L}_k w \right)_{\vert \partial M} = L_1 \cdots L_k \left( w_{\vert \partial M}\right) \in H^{\frac{1}{2}}(\partial M, \mathbb{C}^N).\]
Indeed, it is true if $w \in \mathscr{C}^\infty(M, \mathbb{C}^N)$ and both sides are continuous with respect to the norm of $H^{k + 1}(M, \mathbb{C}^N)$. Thus, for $w \in H^{k + 1}(M, \mathbb{C}^N) \cap H_0^1(M, \mathbb{C}^N)$, one has
\[\left(\tilde{L}_1 \cdots \tilde{L}_k w \right)_{\vert \partial M} = 0 \in H^{\frac{1}{2}}(\partial M, \mathbb{C}^N)\]
and this gives (\ref{eq_proof_LLT_dir_7}). 

Now, we prove (\ref{eq_proof_LLT_dir_6}) in the case $F \neq 0$. Note that the previous proof gives 
\[\left\Vert \partial_\nu u \right\Vert_{H^s((0, T) \times \partial M, \mathbb{C}^N)} \lesssim \left\Vert \left( u^0, u^1 \right) \right\Vert_{\mathcal{K}^{s + 1} \times \mathcal{K}^{s}} + \left\Vert F \right\Vert_{H^s((0, T) \times M, \mathbb{C}^N)},\]
a weaker result. By linearity, we may assume that $\left( u^0, u^1 \right) = 0$. By density, it suffices to prove that for $F \in \mathscr{C}^\infty([0, T], \mathscr{C}^\infty_\mathrm{c}(\inte M, \mathbb{C}^N))$, one has
\[\left\Vert \partial_\nu \Psi \right\Vert_{H^s((0, T) \times \partial M, \mathbb{C}^N)} \lesssim \left\Vert F \right\Vert_{L^1((0, T), H^s)}\]
where $\Psi$ is the Duhamel term defined above. Fix $k \in \llbracket 0, s \rrbracket$, and let $L_1, \cdots, L_k$ be smooth vector fields on the Riemannian manifold $\partial M$. We prove
\[\left\Vert L_1 \cdots L_k \partial_t^{s - k} \partial_\nu \Psi \right\Vert_{L^2((0, T) \times \partial M, \mathbb{C}^N)} \lesssim \left\Vert F \right\Vert_{L^1((0, T), H^s)}.\]

For $\tau \in [0, T]$ and $k \in \mathbb{N}$, one has
\[\partial_t^{2k} u_\tau(0) = 0, \quad 2k \in \llbracket 0, s \rrbracket,\]
and 
\[\partial_t^{2k + 1} u_\tau(0) = \mathsf{P}^{k} F(\tau), \quad 2k + 1 \in \llbracket 0, s \rrbracket.\]
Hence, for $k \in \llbracket 0, s \rrbracket$, there exists a differential operator $R_k$ such that
\[\partial_t^k \Psi(t) = (R_k F)(t) + \int_0^t \partial_t^k u_\tau(t - \tau) \mathrm{d} \tau.\]
As $F(t)$ is compactly supported in $\inte M$, this gives
\[\partial_t^k \partial_\nu \Psi(t) = \int_0^t \partial_t^k \partial_\nu u_\tau(t - \tau) \mathrm{d} \tau.\]
Hence, one has
\begin{align*}
\left\Vert L_1 \cdots L_k \partial_t^{s - k} \partial_\nu \Psi \right\Vert_{L^2((0, T) \times \partial M, \mathbb{C}^N)} & = \left\Vert \int_0^T \mathds{1}_{\tau \leq t} L_1 \cdots L_k \partial_t^{s - k} \partial_\nu u_\tau(t - \tau, x) \mathrm{d} \tau \right\Vert_{L^2((0, T) \times \partial M, \mathbb{C}^N)} \\
& \leq \int_0^T \left\Vert L_1 \cdots L_k \partial_t^{s - k} \partial_\nu u_\tau \right\Vert_{L^2((0, T) \times \partial M, \mathbb{C}^N)} \mathrm{d} \tau \\
& \lesssim \int_0^T \left\Vert \partial_\nu u_\tau \right\Vert_{H^s((0, T) \times \partial M, \mathbb{C}^N)} \mathrm{d} \tau.
\end{align*}
By (\ref{eq_proof_LLT_dir_6}) in the case $F = 0$, we get
\[\left\Vert L_1 \cdots L_k \partial_t^{s - k} \partial_\nu \Psi \right\Vert_{L^2((0, T) \times \partial M, \mathbb{C}^N)} \lesssim \int_0^T \left\Vert F(\tau) \right\Vert_{\mathcal{K}^{s}} \mathrm{d} \tau = \Vert F \Vert_{L^1((0, T), H^s)}.\]

\textbf{The case $s < 0$.} Note that in the sense of classical trace theorem, the normal derivative of a solution does not exist in that case. Take $s \in \mathbb{Z}$, $s < 0$, and $\left( u^0, u^1 \right) \in \mathcal{K}^{s + 1} \times \mathcal{K}^{s}$. There exist $\sigma \in \mathbb{N}^\ast$ and $\alpha \in \{1, 2\}$ such that $s = - 2 \sigma + \alpha$. Let $\left( \tilde{u}^0, \tilde{u}^1 \right)$ be the unique element of $\mathcal{K}^{\alpha + 1} \times \mathcal{K}^{\alpha}$ such that
\[\left( u^0, u^1 \right) = \left( \Shift{s + 1 + \sigma}^\sigma \tilde{u}^0, \Shift{s + \sigma}^\sigma \tilde{u}^1 \right).\]
Recall that the solution $u$ associated with $\left( u^0, u^1 \right)$ is defined by $u = \Shift{s + 1 + \sigma}^\sigma \tilde{u}$, where $\tilde{u}$ is the solution associated with $\left( \tilde{u}^0, \tilde{u}^1 \right)$. Using Proposition \ref{prop_main_properties_K^s}, we can write
\begin{align*}
u(t) & = \Shift{s + 2}^1 \circ \cdots \circ \Shift{\alpha}^1 \left( \tilde{u}(t) \right) \\
& = \left( \PP{s + 2} + i \mu \iota_{\mathcal{K}^{s + 3} \rightarrow \mathcal{K}^{s + 1}} \right) \circ \cdots \circ \left( \PP{\alpha} + i \mu \iota_{\mathcal{K}^{\alpha + 1} \rightarrow \mathcal{K}^{\alpha - 1}} \right) \left( \tilde{u}(t) \right) \\
& = \sum_{k = 0}^\sigma \binom{\sigma}{k} (i \mu)^{\sigma - k} \iota_{\mathcal{K}^{\alpha + 1 - 2 k} \rightarrow \mathcal{K}^{s + 1}} \left( \PP{\alpha + 1 - k}^k \tilde{u}(t) \right)
\end{align*}
and by Theorem \ref{thm_LLT_dir}-\emph{(i)}, we get
\[u(t) = \sum_{k = 0}^\sigma \binom{\sigma}{k} (i \mu)^{\sigma - k} \iota_{\mathcal{K}^{\alpha + 1 - 2 k} \rightarrow \mathcal{K}^{s + 1}} \left( \partial_t^{2 k} \tilde{u}(t) \right)\]
for $t \in \mathbb{R}$. As $\partial_\nu \tilde{u} \in H^\alpha((0, T) \times \partial M, \mathbb{C}^N)$, we define $\partial_\nu u$ by
\[\partial_\nu u = \sum_{k = 0}^\sigma \binom{\sigma}{k} (i \mu)^{\sigma - k} \iota_{H^{\alpha - 2 k} \rightarrow H^s} \left( \partial_t^{2 k} \partial_\nu \tilde{u} \right)\]
where $\iota_{H^{\alpha - 2 k} \rightarrow H^s}$ is the embedding from $H^{\alpha - 2 k}((0, T) \times \partial M, \mathbb{C}^N)$ to $H^s((0, T) \times \partial M, \mathbb{C}^N)$. Clearly, one has 
\begin{align*}
\left\Vert \partial_\nu u \right\Vert_{H^s((0, T) \times \partial M, \mathbb{C}^N)} & \lesssim \sum_{k = 0}^\sigma \left\Vert \iota_{H^{\alpha - 2 k} \rightarrow H^s} \left( \partial_t^{2k} \partial_\nu \tilde{u} \right) \right\Vert_{H^s((0, T) \times \partial M, \mathbb{C}^N)} \\
& \lesssim \left\Vert \partial_\nu \tilde{u} \right\Vert_{H^\alpha((0, T) \times \partial M, \mathbb{C}^N)} \lesssim \left\Vert \left( \tilde{u}^0, \tilde{u}^1 \right) \right\Vert_{\mathcal{K}^{\alpha + 1} \times \mathcal{K}^{\alpha}} \lesssim \left\Vert \left( u^0, u^1 \right) \right\Vert_{\mathcal{K}^{s + 1} \times \mathcal{K}^{s}}.
\end{align*}
To complete the proof, one has to show the two additional results of Theorem \ref{thm_LLT_dir}-\emph{(iii)}.

\textbf{Connection with the usual normal derivative.} Here, we show that our definition of the normal derivative of a solution coincide with the usual normal derivative for a regular solution. More precisely, we prove that for all $s \in \mathbb{R}$, $\delta > 0$ and $\left( u^0, u^1 \right) \in \mathcal{K}^{s + \delta + 1} \times \mathcal{K}^{s + \delta}$, one has
\begin{equation}\label{eq_proof_LLT_dir_8}
\partial_\nu \left( \iota_{\mathcal{K}^{s + \delta + 1} \rightarrow \mathcal{K}^{s + 1}} u \right) = \iota_{H^{s + \delta} \rightarrow H^s} \partial_\nu u.
\end{equation}
By interpolation, it suffices to prove (\ref{eq_proof_LLT_dir_8}) for $s \in \mathbb{Z}$. 

\begin{lem}\label{lem_proof_LLT_dir_normal_derivative}
For $s \in \mathbb{Z}$, $s \leq -1$, and $\left( u^0, u^1 \right) \in \mathcal{K}^2 \times \mathcal{K}^1$, one has
\[\partial_\nu \left( \iota_{\mathcal{K}^2 \rightarrow \mathcal{K}^{s + 1}} u \right) = \iota_{H^1 \rightarrow H^s} \partial_\nu u.\]
\end{lem}

\begin{proof}
Write $s = - 2 \sigma + \alpha$, with $\sigma \in \mathbb{N}^\ast$ and $\alpha \in \{1, 2\}$, and let $\left( \tilde{u}^0, \tilde{u}^1 \right) \in \mathcal{K}^{\alpha + 1} \times \mathcal{K}^{\alpha}$ be given by
\[\left( \iota_{\mathcal{K}^2 \rightarrow \mathcal{K}^{s + 1}} u^0, \iota_{\mathcal{K}^1 \rightarrow \mathcal{K}^{s}} u^1 \right) = \left( \Shift{s + 1 + \sigma}^\sigma \tilde{u}^0, \Shift{s + \sigma}^\sigma \tilde{u}^1 \right).\]
One has $\iota_{\mathcal{K}^2 \rightarrow \mathcal{K}^{s + 1}} u = \Shift{s + 1 + \sigma}^\sigma \tilde{u}$, where $\tilde{u}$ is the solution with initial data $\left( \tilde{u}^0, \tilde{u}^1 \right)$. By definition, one has 
\begin{equation}\label{eq_proof_LLT_dir_10}
\partial_\nu \left( \iota_{\mathcal{K}^2 \rightarrow \mathcal{K}^{s + 1}} u \right) = \sum_{k = 0}^\sigma \binom{\sigma}{k} (i \mu)^{\sigma - k} \iota_{H^{\alpha - 2 k} \rightarrow H^s} \left( \partial_t^{2 k} \partial_\nu \tilde{u} \right).
\end{equation}
Writing
\begin{align*}
\left( \tilde{u}^0, \tilde{u}^1 \right) & = \left( \left( \Shift{s + 1 + \sigma}^\sigma \right)^{- 1} \circ \iota_{\mathcal{K}^2 \rightarrow \mathcal{K}^{s + 1}} \left( u^0 \right), \left( \Shift{s + \sigma}^\sigma \right)^{- 1} \circ \iota_{\mathcal{K}^1 \rightarrow \mathcal{K}^{s}} \left( u^1 \right) \right) \\
& = \left( \iota_{\mathcal{K}^{2 \sigma + 2} \rightarrow \mathcal{K}^{\alpha + 1}} \circ \left( \Shift{2 + \sigma}^\sigma \right)^{- 1} \left( u^0 \right), \iota_{\mathcal{K}^{2 \sigma + 1} \rightarrow \mathcal{K}^{\alpha}} \circ \left( \Shift{1 + \sigma}^\sigma \right)^{- 1} \left( u^1 \right) \right),
\end{align*}
one finds $\tilde{u} = \iota_{\mathcal{K}^{2 \sigma + 2} \rightarrow \mathcal{K}^{\alpha + 1}} v$, where $v$ is the solution with initial data 
\[\left(\left( \Shift{2 + \sigma}^\sigma \right)^{- 1} u^0, \left( \Shift{1 + \sigma}^\sigma \right)^{- 1} u^1 \right) \in \mathcal{K}^{2 \sigma + 2} \times \mathcal{K}^{2 \sigma + 1}.\]
One also has $u = \Shift{2 + \sigma}^\sigma v$. As $\partial_t^{2 k} \partial_\nu \tilde{u} = \partial_t^{2 k} \partial_\nu v$ in $\mathscr{D}^\prime((0, T) \times \partial M, \mathbb{C}^N)$, one finds
\[\partial_t^{2 k} \partial_\nu \tilde{u} = \iota_{H^{2 \sigma + 1 - 2 k} \rightarrow H^{\alpha - 2 k}} \partial_t^{2 k} \partial_\nu v\]
for all $k \in \mathbb{N}$. Coming back to (\ref{eq_proof_LLT_dir_10}), one obtains
\[\partial_\nu \left( \iota_{\mathcal{K}^2 \rightarrow \mathcal{K}^{s + 1}} u \right) = \sum_{k = 0}^\sigma \binom{\sigma}{k} (i \mu)^{\sigma - k} \iota_{H^{1 + 2 \sigma - 2 k} \rightarrow H^s} \left( \partial_t^{2 k} \partial_\nu v \right).\]
For $k \in \llbracket 0, \sigma \rrbracket$, one has $1 + 2 \sigma - 2 k \geq 1$ implying
\begin{equation}\label{eq_proof_LLT_dir_10bis}
\partial_\nu \left( \iota_{\mathcal{K}^2 \rightarrow \mathcal{K}^{s + 1}} u \right) = \iota_{H^1 \rightarrow H^s} \sum_{k = 0}^\sigma \binom{\sigma}{k} (i \mu)^{\sigma - k} \iota_{H^{1 + 2 \sigma - 2 k} \rightarrow H^1} \left( \partial_t^{2 k} \partial_\nu v \right).
\end{equation}
Omitting the embeddings in $H^1((0, T) \times \partial M, \mathbb{C}^N)$, one finds
\[\sum_{k = 0}^\sigma \binom{\sigma}{k} (i \mu)^{\sigma - k} \iota_{H^{1 + 2 \sigma - 2 k} \rightarrow H^1} \left( \partial_t^{2 k} \partial_\nu v \right) = \partial_\nu \left( \Shift{2 + \sigma}^\sigma v \right) = \partial_\nu u.\]
Together with (\ref{eq_proof_LLT_dir_10bis}), this completes the proof.
\end{proof}

To prove (\ref{eq_proof_LLT_dir_8}), we distinguish three cases. First, if $s + 1 \geq 0$ then (\ref{eq_proof_LLT_dir_8}) is true. Second, if $s + 1 \leq - 1$ and $s + \delta + 1 \geq 2$, then using Lemma \ref{lem_proof_LLT_dir_normal_derivative} and the first case, one finds
\begin{align*}
\partial_\nu \left( \iota_{\mathcal{K}^{s + \delta + 1} \rightarrow \mathcal{K}^{s + 1}} u \right) 
& = \partial_\nu \left( \iota_{\mathcal{K}^2 \rightarrow \mathcal{K}^{s + 1}} \circ \iota_{\mathcal{K}^{s + \delta + 1} \rightarrow \mathcal{K}^2} u \right) \\
& = \iota_{H^1 \rightarrow H^s} \partial_\nu \left( \iota_{\mathcal{K}^{s + \delta + 1} \rightarrow \mathcal{K}^2} u \right) \\
& = \iota_{H^1 \rightarrow H^s} \circ \iota_{H^{s + \delta} \rightarrow H^1} \left( \partial_\nu u \right) \\
& = \iota_{H^{s + \delta} \rightarrow H^s} \partial_\nu u.
\end{align*}
Finally, if $s + 1 \leq - 1$ and $s + \delta + 1 < 2$, then we consider an approximation of $u$: take a sequence $\left( \left( u_k^0, u_k^1 \right) \right)_{k \in \mathbb{N}}$ of elements of $\mathcal{K}^2 \times \mathcal{K}^1$ such that 
\[\left( \iota_{\mathcal{K}^2 \rightarrow \mathcal{K}^{s + \delta + 1}} u_k^0, \iota_{\mathcal{K}^1 \rightarrow \mathcal{K}^{s + \delta}} u_k^1 \right) \underset{k \rightarrow \infty}{\longrightarrow} \left( u^0, u^1 \right).\]
For $k \in \mathbb{N}$, let $u_k$ be the solution associated with $\left( u_k^0, u_k^1 \right)$. Set
\[w_k = \iota_{\mathcal{K}^{s + \delta + 1} \rightarrow \mathcal{K}^{s + 1}} u - \iota_{\mathcal{K}^2 \rightarrow \mathcal{K}^{s + 1}} u_k.\]
As $w_k$ is a solution of the wave equation, one has
\[\left\Vert \partial_\nu \left( \iota_{\mathcal{K}^{s + \delta + 1} \rightarrow \mathcal{K}^{s + 1}} u \right) - \partial_\nu \left( \iota_{\mathcal{K}^2 \rightarrow \mathcal{K}^{s + 1}} u_k \right) \right\Vert_{H^s((0, T) \times \partial M, \mathbb{C}^N)} \lesssim \left\Vert \left( w_k(0), \partial_t w_k(0) \right) \right\Vert_{\mathcal{K}^{s + 1} \times \mathcal{K}^{s}}.\]
Writing 
\[\left( w_k(0), \partial_t w_k(0) \right) = \left( \iota_{\mathcal{K}^{s + \delta + 1} \rightarrow \mathcal{K}^{s + 1}} \left( u^0 - \iota_{\mathcal{K}^2 \rightarrow \mathcal{K}^{s + \delta + 1}} u_k^0 \right), \iota_{\mathcal{K}^{s + \delta} \rightarrow \mathcal{K}^{s}} \left( u^1 - \iota_{\mathcal{K}^1 \rightarrow \mathcal{K}^{s + \delta}} u_k^1 \right) \right)\]
one finds
\begin{align}
& \left\Vert \partial_\nu \left( \iota_{\mathcal{K}^{s + \delta + 1} \rightarrow \mathcal{K}^{s + 1}} u \right) - \partial_\nu \left( \iota_{\mathcal{K}^2 \rightarrow \mathcal{K}^{s + 1}} u_k \right) \right\Vert_{H^s((0, T) \times \partial M, \mathbb{C}^N)} \nonumber \\
\lesssim & \left\Vert \left( u^0 - \iota_{\mathcal{K}^2 \rightarrow \mathcal{K}^{s + \delta + 1}} u_k^0, u^1 - \iota_{\mathcal{K}^1 \rightarrow \mathcal{K}^{s + \delta}} u_k^1 \right) \right\Vert_{\mathcal{K}^{s + \delta + 1} \times \mathcal{K}^{s + \delta}} \underset{k \rightarrow \infty}{\longrightarrow} 0. \label{eq_proof_LLT_dir_10bis2}
\end{align}
On the other hand, using Lemma \ref{lem_proof_LLT_dir_normal_derivative}, one has 
\begin{align*}
& \left\Vert \iota_{H^{s + \delta} \rightarrow H^s} \partial_\nu u - \iota_{H^1 \rightarrow H^s} \partial_\nu u_k \right\Vert_{H^s((0, T) \times \partial M, \mathbb{C}^N)} \\
= & \left\Vert \iota_{H^{s + \delta} \rightarrow H^s} \left( \partial_\nu u - \iota_{H^1 \rightarrow H^{s + \delta}} \partial_\nu u_k \right) \right\Vert_{H^s((0, T) \times \partial M, \mathbb{C}^N)} \\
= & \left\Vert \iota_{H^{s + \delta} \rightarrow H^s} \left( \partial_\nu u - \partial_\nu \left( \iota_{\mathcal{K}^2 \rightarrow \mathcal{K}^{s + \delta + 1}} u_k \right) \right) \right\Vert_{H^s((0, T) \times \partial M, \mathbb{C}^N)} \\
= & \left\Vert \partial_\nu u - \partial_\nu \left( \iota_{\mathcal{K}^2 \rightarrow \mathcal{K}^{s + \delta + 1}} u_k \right) \right\Vert_{H^{s + \delta}((0, T) \times \partial M, \mathbb{C}^N)}.
\end{align*}
As above, one finds
\begin{equation}\label{eq_proof_LLT_dir_10bis3}
\left\Vert \iota_{H^{s + \delta} \rightarrow H^s} \partial_\nu u - \iota_{H^1 \rightarrow H^s} \partial_\nu u_k \right\Vert_{H^s((0, T) \times \partial M, \mathbb{C}^N)} \underset{k \rightarrow \infty}{\longrightarrow} 0.
\end{equation}
Lemma \ref{lem_proof_LLT_dir_normal_derivative} gives $\iota_{H^1 \rightarrow H^s} \partial_\nu u_k = \partial_\nu \left( \iota_{\mathcal{K}^2 \rightarrow \mathcal{K}^{s + 1}} u_k \right)$, and with (\ref{eq_proof_LLT_dir_10bis2}) and (\ref{eq_proof_LLT_dir_10bis3}), this completes the proof of the third case.

\textbf{The normal derivative and the time-derivative commute.} Take $s \in \mathbb{R}$, $k \in \mathbb{N}$ and $\left( u^0, u^1 \right) \in \mathcal{K}^{s + 1} \times \mathcal{K}^{s}$. Here, we show that 
\begin{equation}\label{eq_proof_LLT_dir_11}
\partial_\nu \partial_t^{2k} u = \partial_t^{2k} \partial_\nu u.
\end{equation}
Note that the left-hand side is well-defined as we know that $\partial_t^{2k} u$ is a solution of the wave equation. By interpolation, we may assume that $s \in \mathbb{Z}$. If $s - 2 k \geq 0$, then (\ref{eq_proof_LLT_dir_11}) holds true, so we can assume that $s - 2 k \leq -1$. As above, using an approximation, it suffices to prove that
\begin{equation}\label{eq_proof_LLT_dir_12}
\partial_\nu \partial_t^{2k} \left( \iota_{\mathcal{K}^{2 k + 1} \rightarrow \mathcal{K}^{s + 1}} u \right) = \partial_t^{2k} \partial_\nu \left( \iota_{\mathcal{K}^{2 k + 1} \rightarrow \mathcal{K}^{s + 1}} u \right)
\end{equation}
for all $\left( u^0, u^1 \right) \in \mathcal{K}^{2 k + 1} \times \mathcal{K}^{2 k}$. Using (\ref{eq_proof_LLT_dir_8}), one finds
\[\partial_\nu \partial_t^{2k} \left( \iota_{\mathcal{K}^{2 k + 1} \rightarrow \mathcal{K}^{s + 1}} u \right) = \partial_\nu \left( \iota_{\mathcal{K}^1 \rightarrow \mathcal{K}^{s - 2 k + 1}} \partial_t^{2k} u \right) = \iota_{L^2 \rightarrow H^{s - 2 k}} \left( \partial_\nu \partial_t^{2k} u \right).\]
Note that $\partial_\nu \partial_t^{2k} u = \partial_t^{2k} \partial_\nu u$, as $\left( u^0, u^1 \right) \in \mathcal{K}^{2 k + 1} \times \mathcal{K}^{2 k}$. One has $\partial_\nu u = \iota_{H^{2k} \rightarrow H^s} \partial_\nu u$ in $\mathscr{D}^\prime((0, T) \times \partial M, \mathbb{C}^N)$, implying
\[\partial_t^{2 k} \partial_\nu u = \partial_t^{2 k} \iota_{H^{2k} \rightarrow H^s} \partial_\nu u\]
in $\mathscr{D}^\prime((0, T) \times \partial M, \mathbb{C}^N)$. This gives
\[\iota_{L^2 \rightarrow H^{s - 2 k}} \partial_t^{2 k} \partial_\nu u = \partial_t^{2 k} \iota_{H^{2k} \rightarrow H^s} \partial_\nu u.\]
Hence, one obtains
\[\partial_\nu \partial_t^{2k} \left( \iota_{\mathcal{K}^{2 k + 1} \rightarrow \mathcal{K}^{s + 1}} u \right) = \partial_t^{2 k} \iota_{H^{2k} \rightarrow H^s} \partial_\nu u.\]
Using (\ref{eq_proof_LLT_dir_8}) again, one finds (\ref{eq_proof_LLT_dir_12}).

\subsubsection{Proof of Theorem \ref{thm_LLT_inhomogeneous} in negative regularity}

Here, we prove Theorem \ref{thm_LLT_inhomogeneous} for $s \leq 0$. An integration by parts gives the following identity.

\begin{lem}\label{lem_IPP_LLT_inhomogeneous}
For $u$ and $v$ in
\[\mathscr{C}^0([0, T], H^2(M, \mathbb{C}^N)) \cap \mathscr{C}^1([0, T], H^1(M, \mathbb{C}^N)) \cap \mathscr{C}^2([0, T], L^2(M, \mathbb{C}^N))\]
one has
\begin{align*}
& \left\langle \left( \partial_t^2 - \mathsf{P} \right) u, v \right\rangle_{L^2((0, T) \times M, \mathbb{C}^N)} - \left\langle u, \left( \partial_t^2 - \mathsf{P}^\ast \right) v \right\rangle_{L^2((0, T) \times M, \mathbb{C}^N)} \\
= & \left[ \left\langle \partial_t u(t), v(t) \right\rangle_{L^2(M, \mathbb{C}^N)} - \left\langle u(t), \partial_t v(t) \right\rangle_{L^2(M, \mathbb{C}^N)} \right]_0^T - \left\langle \left\langle X, \nu \right\rangle_g u, v \right\rangle_{L^2((0, T) \times \partial M, \mathbb{C}^N)} \\
& \ - \left\langle \partial_\nu u, v \right\rangle_{L^2((0, T) \times \partial M, \mathbb{C}^N)} + \left\langle u, \partial_\nu v \right\rangle_{L^2((0, T) \times \partial M, \mathbb{C}^N)}.
\end{align*}
\end{lem}

\paragraph{Step 1: Definition of the solution.} Take $f \in \mathscr{C}^\infty((0, T) \times \partial M, \mathbb{C}^N)$. If there exists a smooth solution $v$ of (\ref{eq_thm_LLT_inhomogeneous}), then for all smooth function $u$, one has
\begin{align*}
& \left\langle \left( \partial_t^2 - \mathsf{P} \right) u, v \right\rangle_{L^2((0, T) \times M, \mathbb{C}^N)} = \left\langle \partial_t u(T), v(T) \right\rangle_{L^2(M, \mathbb{C}^N)} - \left\langle u(T), \partial_t v(T) \right\rangle_{L^2(M, \mathbb{C}^N)} \\
& \quad - \left\langle \left\langle X, \nu \right\rangle_g u + \partial_\nu u, \diag(\Theta) f \right\rangle_{L^2((0, T) \times \partial M, \mathbb{C}^N)} + \left\langle u, \partial_\nu v \right\rangle_{L^2((0, T) \times \partial M, \mathbb{C}^N)}.
\end{align*}
In particular, if $u$ is such that $u_{\vert (0, T) \times \partial M} = 0$ and $\left( u(T), \partial_t u(T) \right) = 0$ then
\[\left\langle \left( \partial_t^2 - \mathsf{P} \right) u, v \right\rangle_{L^2((0, T) \times M, \mathbb{C}^N)} = - \left\langle \partial_\nu u, \diag(\Theta) f \right\rangle_{L^2((0, T) \times \partial M, \mathbb{C}^N)}.\]
Hence, if $u$ is a smooth solution of 
\begin{equation}\label{eq_proof_LLT_inhom_low_1}
\left \{
\begin{array}{rcccl}
\partial_t^2 u - \mathsf{P} u & = & F & \quad & \text{in } (0, T) \times M, \\
\left( u(T, \cdot), \partial_t u(T, \cdot) \right) & = & 0 & \quad & \text{in } M, \\
u & = & 0 & \quad & \text{on } (0, T) \times \partial M.
\end{array}
\right.
\end{equation}
then one has
\[\left\langle F, v \right\rangle_{L^2((0, T) \times M, \mathbb{C}^N)} = - \left\langle \partial_\nu u, \diag(\Theta) f \right\rangle_{L^2((0, T) \times \partial M, \mathbb{C}^N)}.\]
We use this as the definition of $v$. More precisely, take $s \leq 0$ and define
\[\begin{array}{cccc}
L_s: & L^1((0, T), H_0^{- s}(M, \mathbb{C}^N)) & \longrightarrow & H_0^{- s}((0, T) \times \partial M, \mathbb{C}^N) \\
& F & \longmapsto & - \diag(\Theta) \partial_\nu u
\end{array}\]
where $u$ is the solution of (\ref{eq_proof_LLT_inhom_low_1}). By Theorem \ref{thm_LLT_dir}, the operator $L_s$ is well-defined and continuous. For $f \in H^s((0, T) \times \partial M, \mathbb{C}^N)$, we define the solution $v$ of (\ref{eq_thm_LLT_inhomogeneous}) by $v = L_s^\ast f$. One has $v \in L^\infty((0, T), H^s(M, \mathbb{C}^N))$. In the next step, we show that $v$ is more regular.

\paragraph{Step 2: Regularity of the solution.}
Fix $s \leq 0$. In this step, we show that for $f \in H^s((0, T) \times \partial M, \mathbb{C}^N)$, one has
\begin{equation}\label{eq_proof_LLT_inhom_low_2}
v = L_s^\ast(f) \in \mathscr{C}^0([0, T], H^s(M, \mathbb{C}^N)) \cap \mathscr{C}^1([0, T], H^{s - 1}(M, \mathbb{C}^N)) \cap \mathscr{C}^2([0, T], H^{s - 2}(M, \mathbb{C}^N))
\end{equation}
with an inequality, and $\partial_t^2 v = \PP{\mathscr{D}^\prime}^\ast v$ in $\mathscr{D}^\prime((0, T) \times M, \mathbb{C}^N)$. To get (\ref{eq_proof_LLT_inhom_low_2}) for all $f$, it suffices to show that (\ref{eq_proof_LLT_inhom_low_2}) holds for $f$ smooth, with an inequality of the form
\begin{equation}\label{eq_proof_LLT_inhom_low_3}
\left\Vert v \right\Vert_{\mathscr{C}^0(H^s) \cap \mathscr{C}^1(H^{s - 1}) \cap \mathscr{C}^2(H^{s - 2})} \lesssim \left\Vert f \right\Vert_{H^s((0, T) \times \partial M, \mathbb{C}^N)}.
\end{equation}

\textbf{Proof of (\ref{eq_proof_LLT_inhom_low_2}).} 
Suppose $f \in \mathscr{C}^\infty((0, T) \times \partial M, \mathbb{C}^N)$, and denote by $\tilde{f} \in \mathscr{C}_{\mathrm{c}}^\infty((0, T) \times M, \mathbb{C}^N)$ an extension of $\diag(\Theta) f$. One writes $v = \tilde{f} + w$, where $w$ is a solution of the wave equation with homogeneous Dirichlet boundary condition, as follows. Set 
\[F = - \left(\partial_t^2 - \mathsf{P}^\ast\right) \tilde{f} \in \mathscr{C}^\infty((0, T) \times M, \mathbb{C}^N).\]
Since $F \in L^1((0, T), L^2(M, \mathbb{C}^N)) \cap \mathscr{C}^0((0, T), H^{- 1}(M, \mathbb{C}^N))$, the solution $w$ of 
\begin{equation}\label{eq_proof_LLT_inhom_low_3_bis}
\left \{
\begin{array}{rcccl}
\partial_t^2 w - \mathsf{P}^\ast w & = & F & \quad & \text{in } (0, T) \times M, \\
\left( w(0, \cdot), \partial_t w(0, \cdot) \right) & = & 0 & \quad & \text{in } M, \\
w & = & 0 & \quad & \text{on } (0, T) \times \partial M,
\end{array}
\right.
\end{equation}
is well-defined, and $w \in \mathscr{C}^0([0, T], H_0^1(M, \mathbb{C}^N)) \cap \mathscr{C}^1([0, T], L^2(M, \mathbb{C}^N)) \cap \mathscr{C}^2([0, T], H^{-1}(M, \mathbb{C}^N))$ by Theorem \ref{thm_LLT_dir}-\emph{(iv)}. We claim that $v = \tilde{f} + w$, that is,
\begin{equation}\label{eq_proof_LLT_inhom_low_4}
\left\langle v, \phi \right\rangle_{L^\infty((0, T), H^s), L^1((0, T), H_0^{- s})} = \left\langle w + \tilde{f}, \overline{\phi} \right\rangle_{L^2((0, T) \times M, \mathbb{C}^N)}
\end{equation}
for all $\phi \in L^1((0, T), H_0^{- s}(M))$. By density, it suffices to prove (\ref{eq_proof_LLT_inhom_low_4}) for $\phi \in \mathscr{C}_\mathrm{c}^\infty((0, T) \times \inte M, \mathbb{C}^N)$. Note that Lemma \ref{lem_IPP_LLT_inhomogeneous} does not apply to $w + \tilde{f}$, due to the lack of regularity of $w$. However, it applies to $\tilde{f}$, and Lemma \ref{lem_LLT_Duhamel_duality} can be used for $w$. Consider $\phi \in \mathscr{C}_\mathrm{c}^\infty((0, T) \times \inte M, \mathbb{C}^N)$, and let $u$ be the solution of 
\begin{equation}\label{eq_proof_LLT_inhom_low_5}
\left \{
\begin{array}{rcccl}
\partial_t^2 u - \mathsf{P} u & = & \phi & \quad & \text{in } (0, T) \times M, \\
\left( u(T, \cdot), \partial_t u(T, \cdot) \right) & = & 0 & \quad & \text{in } M, \\
u & = & 0 & \quad & \text{on } (0, T) \times \partial M.
\end{array}
\right.
\end{equation}
One has
\begin{align*}
& \left\langle w + \tilde{f}, \overline{\phi} \right\rangle_{L^2((0, T) \times M, \mathbb{C}^N)} \\
= & \left\langle w, \overline{\phi} \right\rangle_{L^2((0, T) \times M, \mathbb{C}^N)} + \left\langle \tilde{f}, \overline{\phi} \right\rangle_{L^2((0, T) \times M, \mathbb{C}^N)} \\
= & \left\langle F, \overline{u} \right\rangle_{L^2((0, T) \times M, \mathbb{C}^N)} + \left\langle \left(\partial_t^2 - \mathsf{P}^\ast\right) \tilde{f}, \overline{u} \right\rangle_{L^2((0, T) \times M, \mathbb{C}^N)} - \left\langle \tilde{f}, \partial_\nu \overline{u} \right\rangle_{L^2((0, T) \times \partial M, \mathbb{C}^N)} \\
= & - \left\langle \diag(\Theta) f, \partial_\nu \overline{u} \right\rangle_{L^2((0, T) \times \partial M, \mathbb{C}^N)},
\end{align*}
where one has used the fact that $\tilde{f}$ is compactly supported in $(0, T)$. This implies
\[\left\langle w + \tilde{f}, \overline{\phi} \right\rangle_{L^2((0, T) \times M, \mathbb{C}^N)} = \left\langle f, \overline{L_s \phi} \right\rangle_{L^2((0, T) \times \partial M, \mathbb{C}^N)} = \left\langle v, \phi \right\rangle_{L^\infty((0, T), H^s), L^1((0, T), H_0^{- s})},\]
as $v = L_s^\ast f$. This proves (\ref{eq_proof_LLT_inhom_low_4}). In particular, this gives (\ref{eq_proof_LLT_inhom_low_2}). Note that, for now, (\ref{eq_proof_LLT_inhom_low_2}) has only been proved for smooth $f$.

\textbf{Proof of (\ref{eq_proof_LLT_inhom_low_3}).}
We prove (\ref{eq_proof_LLT_inhom_low_3}) for $f$ smooth. First, note that the operator
\[L_s^\ast: H^s((0, T) \times \partial M, \mathbb{C}^N) \longrightarrow L^\infty((0, T), H^s(M, \mathbb{C}^N))\]
is continuous, as $L_s$ is continuous. This gives
\[\left\Vert v \right\Vert_{L^\infty((0, T), H^s(M, \mathbb{C}^N))} \lesssim \left\Vert f \right\Vert_{H^s((0, T) \times \partial M, \mathbb{C}^N)},\]
implying $v = L_s^\ast f \in \mathscr{C}^0([0, T], H^s(M, \mathbb{C}^N))$ if $f \in H^s((0, T) \times \partial M, \mathbb{C}^N)$. 

Second, we prove
\begin{equation}\label{eq_proof_LLT_inhom_low_7}
\left\Vert \partial_t v \right\Vert_{L^\infty((0, T), H^{s - 1}(M, \mathbb{C}^N))} \lesssim \left\Vert f \right\Vert_{H^s((0, T) \times \partial M, \mathbb{C}^N)}.
\end{equation}
Consider $f \in \mathscr{C}^\infty((0, T) \times \partial M, \mathbb{C}^N)$, and $\phi \in \mathscr{C}\mathscr{c}_\mathrm{c}^\infty((0, T) \times \inte M, \mathbb{C}^N)$. By definition, one has
\begin{align*}
\left\langle \partial_t v, \phi \right\rangle_{\mathscr{D}^\prime((0, T) \times M, \mathbb{C}^N), \mathscr{D}((0, T) \times M, \mathbb{C}^N)} & = - \left\langle v, \partial_t \phi \right\rangle_{L^\infty((0, T), H^s), L^1((0, T), H_0^{- s})} \\
& = \left\langle \diag(\Theta) f, \overline{\partial_\nu \tilde{u}} \right\rangle_{L^2((0, T) \times \partial M, \mathbb{C}^N)}
\end{align*}
where $\tilde{u}$ is the solution of 
\[\left \{
\begin{array}{rcccl}
\partial_t^2 \tilde{u} - \mathsf{P} \tilde{u} & = & \partial_t \phi & \quad & \text{in } (0, T) \times M, \\
(\tilde{u}(T, \cdot), \partial_t \tilde{u}(T, \cdot)) & = & 0 & \quad & \text{in } M, \\
\tilde{u} & = & 0 & \quad & \text{on } (0, T) \times \partial M.
\end{array}
\right.\]
Note that $\tilde{u} = \partial_t u$, where $u$ is the solution of (\ref{eq_proof_LLT_inhom_low_5}). Indeed, one has $\partial_t u(T) = 0$ and 
\[\partial_t^2 u(T) = \mathsf{P} u(T) + \phi(T) = 0\]
as $\phi$ is compactly supported. Hence, one finds 
\[\left\vert \left\langle \partial_t v, \phi \right\rangle_{\mathscr{D}^\prime((0, T) \times M, \mathbb{C}^N), \mathscr{D}((0, T) \times M, \mathbb{C}^N)} \right\vert \leq \left\Vert \diag(\Theta) \partial_\nu \partial_t u \right\Vert_{H^{- s}((0, T) \times \partial M, \mathbb{C}^N)} \left\Vert f \right\Vert_{H^s((0, T) \times \partial M, \mathbb{C}^N)}.\]

By Theorem \ref{thm_LLT_dir}, one has 
\[\left\Vert \diag(\Theta) \partial_\nu \partial_t u \right\Vert_{H^{- s}((0, T) \times \partial M, \mathbb{C}^N)} \lesssim \left\Vert \partial_\nu u \right\Vert_{H^{- s + 1}((0, T) \times \partial M, \mathbb{C}^N)} \lesssim \left\Vert \phi \right\Vert_{L^1((0, T), H_0^{- s + 1}(M, \mathbb{C}^N))}.\]
As $\mathscr{C}_\mathrm{c}^\infty((0, T) \times M, \inte \mathbb{C}^N)$ is dense in $L^1((0, T), H_0^{- s + 1}(M, \mathbb{C}^N))$, this gives (\ref{eq_proof_LLT_inhom_low_7}). 

Third, the proof of 
\[\left\Vert \partial_t^2 v \right\Vert_{L^\infty((0, T), H^{s - 2}(M, \mathbb{C}^N))} \lesssim \left\Vert f \right\Vert_{H^s((0, T) \times \partial M, \mathbb{C}^N)}\]
is similar: one writes 
\begin{align*}
\left\vert \left\langle \partial_t^2 v, \phi \right\rangle_{\mathscr{D}^\prime((0, T) \times M, \mathbb{C}^N), \mathscr{D}((0, T) \times M, \mathbb{C}^N)} \right\vert & \leq \left\Vert \diag(\Theta) \partial_\nu \partial_t^2 u \right\Vert_{H^{- s}((0, T) \times \partial M, \mathbb{C}^N)} \left\Vert f \right\Vert_{H^s((0, T) \times \partial M, \mathbb{C}^N)} \\
& \lesssim \left\Vert \phi \right\Vert_{L^1((0, T), H_0^{- s + 2}(M, \mathbb{C}^N))} \left\Vert f \right\Vert_{H^s((0, T) \times \partial M, \mathbb{C}^N)}.
\end{align*}
Note that the equality $\partial_\nu \partial_t^2 u = \partial_t^2 \partial_\nu u$ is obvious, as $\phi$ is smooth.

\textbf{Connection with the wave equation.}
We show that for $f \in H^s((0, T) \times \partial M, \mathbb{C}^N)$, one has
\begin{equation}\label{eq_proof_LLT_inhom_low_8}
\partial_t^2 v - \PP{\mathscr{D}^\prime}^\ast v = 0, \quad \text{ in } \mathscr{D}^\prime((0, T) \times M, \mathbb{C}^N).
\end{equation} 
As $v = L_s^\ast f \in L^\infty((0, T), H^s(M, \mathbb{C}^N))$ is continuous with respect to $f \in H^s((0, T) \times \partial M, \mathbb{C}^N)$, we may assume that $f \in \mathscr{C}^\infty((0, T) \times \partial M, \mathbb{C}^N)$. As above, write $v = w + \tilde{f}$, where $\tilde{f} \in \mathscr{C}_{\mathrm{c}}^\infty((0, T) \times M, \mathbb{C}^N)$ is an extension of $\diag(\Theta) f$, and $w$ is the solution of (\ref{eq_proof_LLT_inhom_low_3_bis}). For all $t \in [0, T]$, one has $\partial_t^2 w(t) - \PP{0}^\ast w(t) = - \left( \partial_t^2 - \mathsf{P}^\ast \right) \tilde{f}(t)$ in $H^{- 1}(M, \mathbb{C}^N)$. Hence, one obtains
\[\left( \partial_t^2 - \PP{\mathscr{D}^\prime}^\ast \right) v(t) = - \left(\partial_t^2 - \mathsf{P}^\ast\right) \tilde{f}(t) + \left(\partial_t^2 - \mathsf{P}^\ast\right) \tilde{f}(t) = 0, \quad t \in [0, T],\]
in $H^{- 1}(M, \mathbb{C}^N)$. This gives (\ref{eq_proof_LLT_inhom_low_8}).

\paragraph{Step 3: The additional regularity result.} Here, we complete the proof of Theorem \ref{thm_LLT_inhomogeneous} for $s \leq 0$, by proving 
\[\left(v(T), \partial_t v(T)\right) \in \mathcal{K}_\ast^{s} \times \mathcal{K}_\ast^{s - 1},\]
and the duality equality (\ref{eq_thm_LLT_inhomogeneous_dual_eq}), for $f \in H^s((0, T) \times \partial M, \mathbb{C}^N)$ and $s \in \mathbb{Z}$, $s \leq 0$. We start with some remarks. We know that $v(T) \in H^s(M, \mathbb{C}^N)$, that is, $v(T)$ is a continuous linear form on $H_0^{- s}(M, \mathbb{C}^N)$. As one has $H_0^{- s}(M, \mathbb{C}^N) \subset \mathcal{K}^{- s}$, we prove that $v(T)$ can be extended as a continuous linear form on $\mathcal{K}^{- s}$. Such an extension is not unique: however, we seek an extension such that (\ref{eq_thm_LLT_inhomogeneous_dual_eq}) holds true, and
\begin{equation}\label{eq_proof_LLT_inhom_low_8_bis}
\left\Vert v(T) \right\Vert_{\mathcal{K}_\ast^{s}} \lesssim \left\Vert f \right\Vert_{H^s((0, T) \times \partial M, \mathbb{C}^N)}.
\end{equation}
The same remarks can be made for $\partial_t v(T)$.

Consider $f \in \mathscr{C}^\infty((0, T) \times \partial M, \mathbb{C}^N)$, and write $v$, $\tilde{f}$, $F$ and $w$ as above. Because of the support of $\tilde{f}$, one has $\left( v(T), \partial_t v(T) \right) = \left( w(T) , \partial_t w(T) \right)$ in $H^s(M, \mathbb{C}^N) \times H^{s - 1}(M, \mathbb{C}^N)$. Consider $\left( u^0, u^1 \right) \in \mathcal{K}^2 \times \mathcal{K}^1$, and write $u$ for the solution of 
\begin{equation}\label{eq_proof_LLT_inhom_low_8_bisbis}
\left \{
\begin{array}{rcccl}
\partial_t^2 u - \mathsf{P} u & = & 0 & \quad & \text{in } (0, T) \times M, \\
\left( u(T, \cdot), \partial_t u(T, \cdot) \right) & = & \left( u^0, u^1 \right) & \quad & \text{in } M, \\
u & = & 0 & \quad & \text{on } (0, T) \times \partial M.
\end{array}
\right.
\end{equation}

Applying Lemma \ref{lem_IPP_LLT_inhomogeneous} to $u$ and $\tilde{f}$, one finds
\[\left\langle u, \left( \partial_t^2 - \mathsf{P}^\ast \right) \tilde{f} \right\rangle_{L^2((0, T) \times M, \mathbb{C}^N)} = \left\langle \partial_\nu u, \diag(\Theta) f \right\rangle_{L^2((0, T) \times \partial M, \mathbb{C}^N)}.\]
By Lemma \ref{lem_LLT_Duhamel_duality}, one has
\[\left\langle u, F \right\rangle_{L^2((0, T) \times M, \mathbb{C}^N)} = \left\langle u^0, \partial_t w(T) \right\rangle_{L^2(M, \mathbb{C}^N)} - \left\langle u^1, w(T) \right\rangle_{L^2(M, \mathbb{C}^N)}\]
Thus, one obtains
\[\left\langle u^1, v(T) \right\rangle_{L^2(M, \mathbb{C}^N)} - \left\langle u^0, \partial_t v(T) \right\rangle_{L^2(M, \mathbb{C}^N)} = \left\langle \partial_\nu u, \diag(\Theta) f \right\rangle_{L^2((0, T) \times \partial M, \mathbb{C}^N)}.\]
If $s \leq -1$, then this is in particular true for all $\left( u^0, u^1 \right) \in H_0^{- s + 1}(M, \mathbb{C}^N) \times H_0^{- s}(M, \mathbb{C}^N)$. If $s = 0$, then this is true for all $\left( u^0, u^1 \right) \in H_0^{- s + 1}(M, \mathbb{C}^N) \times H_0^{- s}(M, \mathbb{C}^N)$ by density. Hence, one has 
\begin{equation}\label{eq_proof_LLT_inhom_low_9}
\left\langle v(T), u^1 \right\rangle_{H^s, H_0^{- s}} - \left\langle \partial_t v(T), u^0 \right\rangle_{H^{s - 1}, H_0^{- s + 1}} = \left\langle f, \diag(\Theta) \partial_\nu u \right\rangle_{H^s, H_0^{- s}}
\end{equation}
for all $f$ smooth. By density and continuity, this is true for all $f \in H^s((0, T) \times \partial M, \mathbb{C}^N)$. Yet, the right-hand side is well-defined for $\left( u^0, u^1 \right) \in \mathcal{K}^{- s + 1} \times \mathcal{K}^{- s}$, and one has
\[\left\vert \left\langle f, \diag(\Theta) \partial_\nu u \right\rangle_{H^s, H_0^{- s}} \right\vert \lesssim \left\Vert \left( u^0, u^1 \right) \right\Vert_{\mathcal{K}^{s + 1} \times \mathcal{K}^{s}} \left\Vert f \right\Vert_{H^s((0, T) \times \partial M, \mathbb{C}^N)}.\]
Hence, (\ref{eq_proof_LLT_inhom_low_9}) yields a unique extension of $\left(v(T), \partial_t v(T)\right)$ as a linear form on $\mathcal{K}^{- s} \times \mathcal{K}^{1 - s}$, which satisfies (\ref{eq_thm_LLT_inhomogeneous_dual_eq}) and (\ref{eq_proof_LLT_inhom_low_8_bis}).

\subsubsection{Proof of Theorem \ref{thm_LLT_inhomogeneous} in positive regularity }

Here, we prove Theorem \ref{thm_LLT_inhomogeneous} for $s > 0$. We know how to construct the solution $v$ of (\ref{eq_thm_LLT_inhomogeneous}) if $f \in L^2((0, T) \times \partial M, \mathbb{C}^N)$, and one has 
\[v \in \mathscr{C}^0([0, T], L^2(M, \mathbb{C}^N)) \cap \mathscr{C}^1([0, T], H^{- 1}(M, \mathbb{C}^N)).\]
If $f \in H^s((0, T) \times \partial M, \mathbb{C}^N)$ with $s > 0$, one can define $v$ as in the case $s = 0$. We show that
\[v \in \mathscr{C}^0([0, T], H^s(M, \mathbb{C}^N)) \cap \mathscr{C}^1([0, T], H^{s - 1}(M, \mathbb{C}^N)) \cap \mathscr{C}^2([0, T], H^{s - 2}(M, \mathbb{C}^N)).\]

We will need the following regularity result, which is an easy consequence of the corresponding scalar result. Set $W = \left\{ u \in L^2(M, \mathbb{C}^N), \PP{\mathscr{D}^\prime} u \in H^{- 1}(M, \mathbb{C}^N) \right\}$.

\begin{lem}\label{lem_reg_ell_inhomogeneous}
The Dirichlet trace $H^1(M, \mathbb{C}^N) \rightarrow H^\frac{1}{2}(M, \mathbb{C}^N)$ has a continuous extension as an operator from $W$ to $H^{-\frac{1}{2}}(\partial M, \mathbb{C}^N)$, and there exists $C > 0$ such that 
\[\left\Vert u_{\vert \partial M} \right\Vert_{H^{-\frac{1}{2}}(\partial M, \mathbb{C}^N)} \leq C \left( \left\Vert \PP{\mathscr{D}^\prime} u \right\Vert_{H^{- 1}(M, \mathbb{C}^N)} + \left\Vert u \right\Vert_{L^2(M, \mathbb{C}^N)} \right), \quad u \in W.\]
In addition, for $m \in \mathbb{N}$, if $u \in L^2(M, \mathbb{C}^N)$ satisfies $\PP{\mathscr{D}^\prime} u \in H^{m - 1}(M, \mathbb{C}^N)$ and $u_{\vert \partial M} \in H^{m + \frac{1}{2}}(M, \mathbb{C}^N)$, then $u \in H^{m + 1}(M, \mathbb{C}^N)$ and 
\[\Vert u \Vert_{H^{m + 1}(M, \mathbb{C}^N)} \leq C \left( \left\Vert \PP{\mathscr{D}^\prime} u \right\Vert_{H^{m - 1}(M, \mathbb{C}^N)} + \left\Vert u_{\vert \partial M} \right\Vert_{H^{m + \frac{1}{2}}(\partial M, \mathbb{C}^N)} + \left\Vert u \right\Vert_{L^2(M, \mathbb{C}^N)} \right)\]
with $C > 0$ independent of $u$.
\end{lem}

\begin{proof}
This result is well-know in the scalar case $N = 1$. We show that the vector-valued case is a consequence of the scalar case. 

Take $u = (u^1, \cdots, u^N) \in L^2(M, \mathbb{C}^N)$ such that $\PP{\mathscr{D}^\prime} u \in H^{- 1}(M, \mathbb{C}^N)$. Write $(\pi_1, \cdots, \pi_N)$ for the projections associated with the canonical basis of $\mathbb{C}^N$. For $k \in \llbracket 1, N \rrbracket$, one has $u^k \in L^2(M, \mathbb{C})$ and
\[\Delta_{\mathscr{D}^\prime} u^k - \pi^k \left( X u + q u \right) \in H^{- 1}(M, \mathbb{C})\]
so that $\Delta_{\mathscr{D}^\prime} u^k \in H^{- 1}(M, \mathbb{C})$. Hence, the scalar case gives $u_{\vert \partial M} \in H^{- \frac{1}{2}}(M, \mathbb{C}^N)$, with
\[\left\Vert u_{\vert \partial M} \right\Vert_{H^{-\frac{1}{2}}(\partial M, \mathbb{C}^N)} \lesssim \left\Vert \Delta_{\mathscr{D}^\prime} u \right\Vert_{H^{- 1}(M, \mathbb{C}^N)} + \left\Vert u \right\Vert_{L^2(M, \mathbb{C}^N)}.\]
Writing
\begin{align*}
\left\Vert \Delta_{\mathscr{D}^\prime} u \right\Vert_{H^{- 1}(M, \mathbb{C}^N)} & \leq \left\Vert \PP{\mathscr{D}^\prime} u \right\Vert_{H^{- 1}(M, \mathbb{C}^N)} + \left\Vert X u + q u \right\Vert_{H^{- 1}(M, \mathbb{C}^N)} \\
& \lesssim \left\Vert \PP{\mathscr{D}^\prime} u \right\Vert_{H^{- 1}(M, \mathbb{C}^N)} + \left\Vert u \right\Vert_{L^2(M, \mathbb{C}^N)},
\end{align*}
one obtains the first part of the lemma.

We prove the second part of the lemma by induction. Start with $m = 0$ and $u \in L^2(M, \mathbb{C}^N)$ such that $\PP{\mathscr{D}^\prime} u \in H^{- 1}(M, \mathbb{C}^N)$ and $u_{\vert \partial M} \in H^{\frac{1}{2}}(M, \mathbb{C}^N)$. As above, for $k \in \llbracket 1, N \rrbracket$, one has $\Delta_{\mathscr{D}^\prime} u^k \in H^{- 1}(M, \mathbb{C})$, so the scalar case gives $u \in H^1(M, \mathbb{C}^N)$ and 
\begin{align*}
\Vert u \Vert_{H^1(M, \mathbb{C}^N)} 
& \lesssim \left\Vert \Delta_{\mathscr{D}^\prime} u \right\Vert_{H^{- 1}(M, \mathbb{C}^N)} + \left\Vert u_{\vert \partial M} \right\Vert_{H^{\frac{1}{2}}(\partial M, \mathbb{C}^N)} + \left\Vert u \right\Vert_{L^2(M, \mathbb{C}^N)} \\
& \leq \left\Vert \PP{\mathscr{D}^\prime} u \right\Vert_{H^{- 1}(M, \mathbb{C}^N)} + \left\Vert X u + q u \right\Vert_{H^{- 1}(M, \mathbb{C}^N)} + \left\Vert u_{\vert \partial M} \right\Vert_{H^{\frac{1}{2}}(\partial M, \mathbb{C}^N)} + \left\Vert u \right\Vert_{L^2(M, \mathbb{C}^N)} \\
& \lesssim \left\Vert \PP{\mathscr{D}^\prime} u \right\Vert_{H^{- 1}(M, \mathbb{C}^N)} + \left\Vert u_{\vert \partial M} \right\Vert_{H^{\frac{1}{2}}(\partial M, \mathbb{C}^N)} + \left\Vert u \right\Vert_{L^2(M, \mathbb{C}^N)}.
\end{align*}

Finally, assume that the result holds for some $m \in \mathbb{N}$. Take $u \in L^2(M, \mathbb{C}^N)$ such that $\PP{\mathscr{D}^\prime} u \in H^{(m + 1) - 1}(M, \mathbb{C}^N)$ and $u_{\vert \partial M} \in H^{m + 1 + \frac{1}{2}}(M, \mathbb{C}^N)$. By induction, $u \in H^{m + 1}(M, \mathbb{C}^N)$ and 
\begin{equation}\label{eq_proof_lem_reg_ell_inhom}
\Vert u \Vert_{H^{m + 1}(M, \mathbb{C}^N)} \lesssim \left\Vert \PP{\mathscr{D}^\prime} u \right\Vert_{H^{m - 1}(M, \mathbb{C}^N)} + \left\Vert u_{\vert \partial M} \right\Vert_{H^{m + \frac{1}{2}}(\partial M, \mathbb{C}^N)} + \left\Vert u \right\Vert_{L^2(M, \mathbb{C}^N)}.
\end{equation}
Hence, for $k \in \llbracket 1, N \rrbracket$, one has 
\[\Delta_{\mathscr{D}^\prime} u^k = \pi^k \left( \PP{\mathscr{D}^\prime} u + X u + q u \right) \in H^m(M, \mathbb{C})\]
so the scalar case gives $u^k \in H^{m + 2}(M, \mathbb{C})$ and
\begin{align*}
\Vert u \Vert_{H^{m + 2}(M, \mathbb{C}^N)} 
& \lesssim \left\Vert \Delta_{\mathscr{D}^\prime} u \right\Vert_{H^m(M, \mathbb{C}^N)} + \left\Vert u_{\vert \partial M} \right\Vert_{H^{m + \frac{3}{2}}(\partial M, \mathbb{C}^N)} + \left\Vert u \right\Vert_{L^2(M, \mathbb{C}^N)} \\
& \leq \left\Vert \PP{\mathscr{D}^\prime} u \right\Vert_{H^m(M, \mathbb{C}^N)} + \left\Vert X u + q u \right\Vert_{H^m(M, \mathbb{C}^N)} + \left\Vert u_{\vert \partial M} \right\Vert_{H^{m + \frac{3}{2}}(\partial M, \mathbb{C}^N)} + \left\Vert u \right\Vert_{L^2(M, \mathbb{C}^N)} \\
& \lesssim \left\Vert \PP{\mathscr{D}^\prime} u \right\Vert_{H^m(M, \mathbb{C}^N)} + \left\Vert u \right\Vert_{H^{m + 1}(M, \mathbb{C}^N)} + \left\Vert u_{\vert \partial M} \right\Vert_{H^{m + \frac{3}{2}}(\partial M, \mathbb{C}^N)} + \left\Vert u \right\Vert_{L^2(M, \mathbb{C}^N)}.
\end{align*}
Using (\ref{eq_proof_lem_reg_ell_inhom}), one finds
\[\Vert u \Vert_{H^{m + 2}(M, \mathbb{C}^N)} \lesssim \left\Vert \PP{\mathscr{D}^\prime} u \right\Vert_{H^m(M, \mathbb{C}^N)} + \left\Vert u_{\vert \partial M} \right\Vert_{H^{m + \frac{3}{2}}(\partial M, \mathbb{C}^N)} + \left\Vert u \right\Vert_{L^2(M, \mathbb{C}^N)}\]
and this completes the proof.
\end{proof}

We prove Theorem \ref{thm_LLT_inhomogeneous} by induction on $s \geq 1$. 

\paragraph{Step 1: The case $s = 1$.} Fix $f \in H^1((0, T) \times \partial M, \mathbb{C}^N)$, and write $v$ for the associated solution. From the case $s = 0$, one has 
\[v \in \mathscr{C}^0([0, T], L^2(M, \mathbb{C}^N)) \cap \mathscr{C}^1([0, T], H^{- 1}(M, \mathbb{C}^N)) \cap \mathscr{C}^2([0, T], H^{- 2}(M, \mathbb{C}^N))\]
and $\partial_t^2 v = \PP{\mathscr{D}^\prime}^\ast v$ in $\mathscr{D}^\prime((0, T) \times M, \mathbb{C}^N)$. We prove first that 
\begin{equation}\label{eq_proof_LLT_inhom_high_1}
v \in \mathscr{C}^1([0, T], L^2(M, \mathbb{C}^N)) \cap \mathscr{C}^2([0, T], H^{- 1}(M, \mathbb{C}^N))
\end{equation}
with an inequality. Then, using Lemma \ref{lem_reg_ell_inhomogeneous}, we show that 
\begin{equation}\label{eq_proof_LLT_inhom_high_3}
v \in \mathscr{C}^0([0, T], H^1(M, \mathbb{C}^N)),
\end{equation}
with an inequality, and
\begin{equation}\label{eq_proof_LLT_inhom_high_4}
v(t)_{\vert \partial M} = \left(\diag(\Theta) f \right)_{\vert \{t\} \times \partial M}
\end{equation}
in $H^\frac{1}{2}(\partial M, \mathbb{C}^N)$ for all $t \in [0, T]$.

We prove (\ref{eq_proof_LLT_inhom_high_1}). Let $\tilde{v}$ be the solution of 
\begin{equation}\label{eq_proof_LLT_inhom_high_4_bis}
\left \{
\begin{array}{rcccl}
\partial_t^2 \tilde{v} - \mathsf{P}^\ast \tilde{v} & = & 0 & \quad & \text{in } (0, T) \times M, \\
(\tilde{v}(0, \cdot), \partial_t \tilde{v}(0, \cdot)) & = & 0 & \quad & \text{in } M, \\
\tilde{v} & = & \partial_t (\diag(\Theta) f) & \quad & \text{on } (0, T) \times \partial M.
\end{array}
\right.
\end{equation}
For $\tilde{\Theta} \in \mathscr{C}_\mathrm{c}^\infty((0, T) \times \partial M, \mathbb{C}^N)$ such that for all $k \in \llbracket 1, N \rrbracket$, $\pi_k \tilde{\Theta} = 1$ in a neighbourhood of $\supp \pi_k \Theta$, one has $\partial_t \left( \diag(\Theta) f \right) = \diag(\tilde{\Theta}) \partial_t \left( \diag(\Theta) f \right)$, implying that $\tilde{v}$ is well-defined. One has 
\[\partial_t \left( \diag(\Theta) f \right) \in L^2((0, T) \times \partial M, \mathbb{C}^N),\]
yielding $\tilde{v} \in \mathscr{C}^0([0, T], L^2(M, \mathbb{C}^N)) \cap \mathscr{C}^1([0, T], H^{- 1}(M, \mathbb{C}^N))$ by Theorem \ref{thm_LLT_inhomogeneous} in the case $s = 0$. We show that $\partial_t v = \tilde{v}$. Consider $\phi \in \mathscr{C}_\mathrm{c}^\infty((0, T) \times \inte M, \mathbb{C}^N)$. Let $u$ and $\tilde{u}$ be the solutions of 
\[\left \{
\begin{array}{rcccl}
\partial_t^2 u - \mathsf{P} u & = & \phi & \quad & \text{in } (0, T) \times M, \\
\left( u(T, \cdot), \partial_t u(T, \cdot) \right) & = & 0 & \quad & \text{in } M, \\
u & = & 0 & \quad & \text{on } (0, T) \times \partial M,
\end{array}
\right.\]
\[\left \{
\begin{array}{rcccl}
\partial_t^2 \tilde{u} - \mathsf{P} \tilde{u} & = & \partial_t \phi & \quad & \text{in } (0, T) \times M, \\
(\tilde{u}(T, \cdot), \partial_t \tilde{u}(T, \cdot)) & = & 0 & \quad & \text{in } M, \\
\tilde{u} & = & 0 & \quad & \text{on } (0, T) \times \partial M.
\end{array}
\right.\]
One has $\partial_t u = \tilde{u}$, yielding, by definition of $v$ and $\tilde{v}$,
\begin{align*}
\left\langle \partial_t v, \phi \right\rangle_{L^\infty((0, T), H^{-1}), L^1((0, T), H_0^1)}
& = - \left\langle v, \partial_t \overline{\phi} \right\rangle_{L^2((0, T) \times M, \mathbb{C}^N)} \\
& = \left\langle f, \diag(\Theta) \partial_\nu \overline{\tilde{u}} \right\rangle_{L^2((0, T) \times \partial M, \mathbb{C}^N)} \\
& = - \left\langle \partial_t \left( \diag(\Theta) f \right), \partial_\nu \overline{u} \right\rangle_{L^2((0, T) \times \partial M, \mathbb{C}^N)} \\
& = \left\langle \tilde{v}, \phi \right\rangle_{L^\infty((0, T), L^2), L^1((0, T), L^2)}.
\end{align*}
This gives (\ref{eq_proof_LLT_inhom_high_1}). In addition, one has 
\begin{align*}
\Vert v \Vert_{\mathscr{C}^1([0, T], L^2) \cap \mathscr{C}^2([0, T], H^{- 1})} & \lesssim \left\Vert v \right\Vert_{\mathscr{C}^0([0, T], L^2) \cap \mathscr{C}^1([0, T], H^{- 1})} +  \left\Vert \tilde{v} \right\Vert_{\mathscr{C}^0([0, T], L^2) \cap \mathscr{C}^1([0, T], H^{- 1})} \\
& \lesssim \Vert f \Vert_{L^2((0, T) \times \partial M, \mathbb{C}^N)} + \left\Vert \partial_t \left( \diag(\Theta) f \right) \right\Vert_{L^2((0, T) \times \partial M, \mathbb{C}^N)} \\
& \lesssim \Vert f \Vert_{H^1((0, T) \times \partial M, \mathbb{C}^N)}.
\end{align*}

Now, we prove (\ref{eq_proof_LLT_inhom_high_3}). Consider $f \in \mathscr{C}^\infty((0, T) \times \partial M, \mathbb{C}^N)$, and write $v = w + \tilde{f}$ as in the proof of Theorem \ref{thm_LLT_inhomogeneous} in negative regularity. As $w \in \mathscr{C}^0([0, T], H_0^1(M, \mathbb{C}^N))$, (\ref{eq_proof_LLT_inhom_high_3}) is true for $f$ smooth. To get (\ref{eq_proof_LLT_inhom_high_3}) for all $f$, we prove 
\begin{equation}\label{eq_proof_LLT_inhom_high_4_bisbis}
\left\Vert v \right\Vert_{L^\infty((0, T), H^1(M, \mathbb{C}^N))} \lesssim \left\Vert f \right\Vert_{H^1((0, T) \times \partial M, \mathbb{C}^N)}.
\end{equation}
For $t \in [0, T]$, one has
\[\PP{\mathscr{D}^\prime}^\ast v(t) = \partial_t^2 v(t) \in H^{- 1}(M, \mathbb{C}^N)\]
and
\begin{equation}\label{eq_proof_LLT_inhom_high_5}
v(t)_{\vert \partial M} = w(t)_{\vert \partial M} + \tilde{f}(t)_{\vert \partial M} = 0 + \left( \diag(\Theta) f \right)_{\vert \{ t \} \times \partial M}
\end{equation}
in $H^{\frac{1}{2}}(\partial M, \mathbb{C}^N)$. Hence, Lemma \ref{lem_reg_ell_inhomogeneous} gives 
\begin{align*}
\Vert v(t) \Vert_{H^1(M, \mathbb{C}^N)} & \lesssim \left\Vert \partial_t^2 v(t) \right\Vert_{H^{- 1}(M, \mathbb{C}^N)} + \left\Vert \left( \diag(\Theta) f \right)_{\vert \{ t \} \times \partial M} \right\Vert_{H^\frac{1}{2}(\partial M, \mathbb{C}^N)} + \left\Vert v \right\Vert_{L^\infty((0, T), L^2)} \\
& \lesssim \left\Vert f \right\Vert_{H^1((0, T) \times \partial M, \mathbb{C}^N)}.
\end{align*}
for all $t \in [0, T]$. This gives (\ref{eq_proof_LLT_inhom_high_4_bisbis}). By density, it holds for all $f \in H^1((0, T) \times \partial M, \mathbb{C}^N)$, yielding (\ref{eq_proof_LLT_inhom_high_3}). Note also that (\ref{eq_proof_LLT_inhom_high_4}) holds for smooth $f$ by (\ref{eq_proof_LLT_inhom_high_5}). As both sides of (\ref{eq_proof_LLT_inhom_high_4}) are continuous with respect to $f \in H^1((0, T) \times \partial M, \mathbb{C}^N)$, we obtain (\ref{eq_proof_LLT_inhom_high_4}) for all $f \in H^1((0, T) \times \partial M, \mathbb{C}^N)$.

\paragraph{Step 2: The case $s \in \mathbb{N}^\ast$.}
We show by induction on $s \in \mathbb{N}^\ast$ that
\begin{equation}\label{eq_proof_LLT_inhom_high_5bis}
v \in \mathscr{C}^0([0, T], H^s(M, \mathbb{C}^N)) \cap \mathscr{C}^1([0, T], H^{s - 1}(M, \mathbb{C}^N)) \cap \mathscr{C}^2([0, T], H^{s - 2}(M, \mathbb{C}^N)),
\end{equation}
if $f \in H^s((0, T) \times \partial M, \mathbb{C}^N)$, with an inequality, and $v(t)_{\vert \partial M} = \left( \diag(\Theta) f \right)_{\vert \{t\} \times \partial M}$ in $H^{s - \frac{1}{2}}(M, \mathbb{C}^N)$ for $t \in [0, T]$. Assume that the result holds for some $s \in \mathbb{N}^\ast$, and consider $f \in H^{s + 1}((0, T) \times \partial M, \mathbb{C}^N)$.

By induction, one has $v(t)_{\vert \partial M} = \left( \diag(\Theta) f \right)_{\vert \{t\} \times \partial M}$ in $H^{s - \frac{1}{2}}(M, \mathbb{C}^N)$ for $t \in [0, T]$. In particular, this gives $v(t)_{\vert \partial M} \in H^{s + \frac{1}{2}}(M, \mathbb{C}^N)$. As in the case $s = 1$, one has $\partial_t v = \tilde{v}$, where $\tilde{v}$ is the solution of (\ref{eq_proof_LLT_inhom_high_4_bis}), and this gives 
\[\Vert v \Vert_{\mathscr{C}^1([0, T], H^s) \cap \mathscr{C}^2([0, T], H^{s - 1})} \lesssim \Vert f \Vert_{H^{s + 1}((0, T) \times \partial M, \mathbb{C}^N))},\]
since $\tilde{v}$ fulfills (\ref{eq_proof_LLT_inhom_high_5bis}). In particular, one has $\PP{\mathscr{D}^\prime}^\ast v(t) = \partial_t^2 v(t) \in H^{s - 1}(M, \mathbb{C}^N)$ for $t \in [0, T]$. Hence, for $t \in [0, T]$, Lemma \ref{lem_reg_ell_inhomogeneous} gives $v(t) \in H^{s + 1}(M, \mathbb{C}^N)$, and
\begin{align*}
\left\Vert v(t) \right\Vert_{H^{s + 1}(M, \mathbb{C}^N)} & \lesssim \left\Vert \PP{\mathscr{D}^\prime}^\ast v(t) \right\Vert_{H^{s - 1}(M, \mathbb{C}^N)} + \left\Vert v(t)_{\vert \partial M} \right\Vert_{H^{s + \frac{1}{2}}(\partial M, \mathbb{C}^N)} + \left\Vert v(t) \right\Vert_{L^2(M, \mathbb{C}^N)} \\
& \lesssim \left\Vert \partial_t^2 v(t) \right\Vert_{H^{s - 1}(M, \mathbb{C}^N)} + \left\Vert \left( \diag(\Theta) f \right)_{\vert \{t\} \times \partial M} \right\Vert_{H^{s + \frac{1}{2}}(\partial M, \mathbb{C}^N)} + \left\Vert v \right\Vert_{L^\infty((0, T), L^2)} \\
& \lesssim \left\Vert f \right\Vert_{H^{s + 1}((0, T) \times \partial M, \mathbb{C}^N)}.
\end{align*}
This gives $v \in L^\infty([0, T], H^{s + 1}(M, \mathbb{C}^N))$. To complete the proof, it suffices to show that 
\[v \in \mathscr{C}^0([0, T], H^{s + 1}(M, \mathbb{C}^N))\]
for $f$ smooth. For such $f$, Lemma \ref{lem_reg_ell_inhomogeneous} gives
\begin{align*}
& \left\Vert v(t + \epsilon) - v(t) \right\Vert_{H^{s + 1}(M, \mathbb{C}^N)} \\
\lesssim & \left\Vert \partial_t^2 v(t + \epsilon) - \partial_t^2 v(t) \right\Vert_{H^{s - 1}(M, \mathbb{C}^N)} + \left\Vert f(t + \epsilon) - f(t) \right\Vert_{H^{s + \frac{1}{2}}(\partial M, \mathbb{C}^N)} + \left\Vert v(t + \epsilon) - v(t) \right\Vert_{L^2(M, \mathbb{C}^N)},
\end{align*}
and one concludes 
\[\left\Vert v(t + \epsilon) - v(t) \right\Vert_{H^{s + 1}(M, \mathbb{C}^N)} \underset{\epsilon \rightarrow 0}{\longrightarrow} 0.\]
This gives (\ref{eq_proof_LLT_inhom_high_5bis}).

\paragraph{Step 3: The additional regularity result.} Here, we complete the proof of Theorem \ref{thm_LLT_inhomogeneous} by proving 
\begin{equation}\label{eq_proof_LLT_inhom_high_6}
\left(v(T), \partial_t v(T)\right) \in \mathcal{K}_\ast^{s} \times \mathcal{K}_\ast^{s - 1}
\end{equation}
for $f \in H^s((0, T) \times \partial M, \mathbb{C}^N)$ and $s \in \mathbb{N}^\ast$. We will also prove the duality equality of Theorem \ref{thm_LLT_inhomogeneous}. As $\mathcal{K}_\ast^{s} \times \mathcal{K}_\ast^{s - 1}$ is a closed subspace of $H^s(M, \mathbb{C}^N) \times H^{s - 1}(M, \mathbb{C}^N)$, it suffices to prove (\ref{eq_proof_LLT_inhom_high_6}) for $f$ smooth. We proceed by induction. The result is true for $s = 0$, but one has to treat the case $s = 1$ separately. 

\textbf{Case 1:} $s = 1$. Take $f \in \mathscr{C}^\infty((0, T) \times \partial M, \mathbb{C}^N)$, and write $v$ for the associated solution. We prove that $v(T) \in H_0^1(M, \mathbb{C}^N)$. Write $v = w + \tilde{f}$ as above. One has $w \in \mathscr{C}^0([0, T], H_0^1(M, \mathbb{C}^N))$, implying
\[v(T) = w(T) + \tilde{f}(T) = w(T) \in H_0^1(M, \mathbb{C}^N),\]
since $\tilde{f}$ is compactly supported in $(0, T) \times M$. Note that $w \in \mathscr{C}^0([0, T], \mathcal{K}_\ast^{s})$ for $s \geq 2$ is not always true, preventing a straightforward generalization of this argument.

\textbf{Case 2:} $s$ odd. Assume that the result holds true for some odd $s \in \mathbb{N}^\ast$. Write $s = 2 \sigma + 1$, and consider $f \in \mathscr{C}^\infty((0, T) \times \partial M, \mathbb{C}^N)$. We prove that $\left(v(T), \partial_t v(T)\right) \in \mathcal{K}_\ast^{2 \sigma + 2} \times \mathcal{K}_\ast^{2 \sigma + 1}$. By induction, one has $\left(v(T), \partial_t v(T)\right) \in \mathcal{K}_\ast^{2 \sigma + 1} \times \mathcal{K}_\ast^{2 \sigma}$, and as $f \in H^{s + 1}((0, T) \times \partial M, \mathbb{C}^N)$, we know that
\[v \in \mathscr{C}^0([0, T], H^{s + 1}(M, \mathbb{C}^N)) \cap \mathscr{C}^1([0, T], H^s(M, \mathbb{C}^N)).\]
By definition, $H^{2 \sigma + 2}(M, \mathbb{C}^N)) \cap \mathcal{K}_\ast^{2 \sigma + 1} = \mathcal{K}_\ast^{2 \sigma + 2}$, implying $v(T) \in \mathcal{K}_\ast^{2 \sigma + 2}$. Set $\tilde{v} = \partial_t v$, solution to (\ref{eq_proof_LLT_inhom_high_4_bis}). By induction, one has $\partial_t v(T) = \tilde{v}(T) \in \mathcal{K}_\ast^{2 \sigma + 1}$, completing the proof in the case $s$ is odd.

\textbf{Case 3:} $s$ even. Assume that the result holds for some even $s \in \mathbb{N}^\ast$. Write $s = 2 \sigma$, and consider $f \in \mathscr{C}^\infty((0, T) \times \partial M, \mathbb{C}^N)$. We show that $\left(v(T), \partial_t v(T)\right) \in \mathcal{K}_\ast^{2 \sigma + 1} \times \mathcal{K}_\ast^{2 \sigma}$. By induction, one has $\left(v(T), \partial_t v(T)\right) \in \mathcal{K}_\ast^{2 \sigma} \times \mathcal{K}_\ast^{2 \sigma - 1}$, and we know that $\left(v(T), \partial_t v(T)\right) \in H^{s + 1}(M, \mathbb{C}^N) \times H^s(M, \mathbb{C}^N)$, yielding $\partial_t v(T) \in \mathcal{K}_\ast^{2 \sigma}$ as above. The definition of $\mathcal{K}_\ast^{2 \sigma + 1}$ reads
\[\mathcal{K}_\ast^{2 \sigma + 1} = \left\{ u \in H^{2 \sigma + 1}(M, \mathbb{C}^N)) \cap \mathcal{K}_\ast^{2 \sigma}, \PP{\mathscr{D}^\prime}^{\ast \ \sigma} u \in H_0^1(M, \mathbb{C}^N)) \right\},\]
implying that one has $v(T) \in \mathcal{K}_\ast^{2 \sigma + 1}$, if $\PP{\mathscr{D}^\prime}^{\ast \ \sigma} v(T) \in H_0^1(M, \mathbb{C}^N))$. Again, write $\tilde{v} = \partial_t v$. One has $\PP{\mathscr{D}^\prime}^\ast v(T) = \partial_t^2 v(T) = \partial_t \tilde{v}(T)$, and by induction, one finds $\partial_t \tilde{v}(T) \in \mathcal{K}_\ast^{2 \sigma - 1}$. Thus, as $\sigma > 0$, one obtains
\[\PP{\mathscr{D}^\prime}^{\ast \ \sigma} v(T) = \PP{\mathscr{D}^\prime}^{\ast \ \sigma - 1} \partial_t \tilde{v}(T) \in H_0^1(M, \mathbb{C}^N),\]
completing the proof in the case $s$ is even.

Finally, we prove the duality equality of Theorem \ref{thm_LLT_inhomogeneous} for $s \geq 1$. Consider $s \geq 1$, $f \in H^s((0, T) \times \partial M, \mathbb{C^N})$ and $\left( u^0, u^1 \right) \in \mathcal{K}^{- s + 1} \times \mathcal{K}^{- s}$. Write $u$ for the solution of (\ref{eq_proof_LLT_inhom_low_8_bisbis}). We show that
\begin{equation}\label{eq_proof_LLT_inhom_high_7}
\left\langle u^1, v(T) \right\rangle_{\mathcal{K}^{- s + 1}, \mathcal{K}_\ast^{s - 1}} - \left\langle u^0, \partial_t v(T) \right\rangle_{\mathcal{K}^{- s}, \mathcal{K}_\ast^{s}} = \left\langle \partial_\nu u, \diag(\Theta) f \right\rangle_{H^{- s}, H_0^s}.
\end{equation}

With the approximation result of Theorem \ref{thm_LLT_dir}, consider a sequence $\left(( u_k^0, u_k^1 )\right)_{k \in \mathbb{N}}$ of elements of $\mathcal{K}^2 \times \mathcal{K}^1$ such that
\[\left( \iota_{\mathcal{K}^2 \rightarrow \mathcal{K}^{- s + 1}} u_k^0, \iota_{\mathcal{K}^1 \rightarrow \mathcal{K}^{- s}} u_k^1 \right) \underset{k \rightarrow \infty}{\longrightarrow} \left( u^0, u^1 \right) \text{ in } \mathcal{K}^{- s + 1} \times \mathcal{K}^{- s}.\]
If $u_k$ is the solution of (\ref{eq_proof_LLT_inhom_low_8_bisbis}) associated with $\left( u_k^0, u_k^1 \right)$, then one has 
\[\iota_{\mathcal{K}^2 \rightarrow \mathcal{K}^{- s + 1}} u_k \underset{k \rightarrow \infty}{\longrightarrow} u \text{ in } \mathscr{C}^0([0, T], \mathcal{K}^{- s + 1}) \cap \mathscr{C}^1([0, T], \mathcal{K}^{- s}),\]
and 
\[\partial_\nu \left( \iota_{\mathcal{K}^2 \rightarrow \mathcal{K}^{- s + 1}} u_k \right) \underset{k \rightarrow \infty}{\longrightarrow} \partial_\nu u \text{ in } H^{- s}((0, T) \times \partial M, \mathbb{C}^N).\]
The duality equality of Theorem \ref{thm_LLT_inhomogeneous} for $s = 0$ gives
\[\left\langle u_k^1, v(T) \right\rangle_{L^2(M, \mathbb{C}^N)} - \left\langle u_k^0, \partial_t v(T) \right\rangle_{L^2(M, \mathbb{C}^N)} = \left\langle \diag(\Theta) \partial_\nu u_k, f \right\rangle_{L^2((0, T) \times \partial M, \mathbb{C}^N)}.\]
Since $\left(v(T), \partial_t v(T)\right) \in \mathcal{K}_\ast^{s} \times \mathcal{K}_\ast^{s - 1}$, one finds
\[\left\langle \iota_{\mathcal{K}^2 \rightarrow \mathcal{K}^{- s + 1}} u_k^1, v(T) \right\rangle_{\mathcal{K}^{- s + 1}, \mathcal{K}_\ast^{s - 1}} - \left\langle \iota_{\mathcal{K}^1 \rightarrow \mathcal{K}^{- s}} u_k^0, \partial_t v(T) \right\rangle_{\mathcal{K}^{- s}, \mathcal{K}_\ast^{s}} = \left\langle \iota_{H^1 \rightarrow H^{- s}} \partial_\nu u_k, \diag(\Theta) f \right\rangle_{H^{- s}, H_0^s}.\]
From Theorem \ref{thm_LLT_dir}, one has $\iota_{H^1 \rightarrow H^{- s}} \partial_\nu u_k = \partial_\nu \iota_{\mathcal{K}^2 \rightarrow \mathcal{K}^{- s + 1}} u_k$, yielding (\ref{eq_proof_LLT_inhom_high_7}).

\printbibliography

\noindent
\textsc{Perrin Thomas:} \texttt{perrin@math.univ-paris13.fr}

\noindent
\textit{Laboratoire Analyse Géométrie et Application, Institut Galilée - UMR 7539, CNRS/Université Sorbonne Paris Nord, 99 avenue J.B. Clément, 93430 Villetaneuse, France}

\end{document}